\documentclass{article}
\usepackage[english]{babel}
\usepackage[margin=36mm]{geometry}

\usepackage{amsmath}
\usepackage{mathtools}
\usepackage{graphicx}
\usepackage{amssymb}
\usepackage[colorlinks=true, allcolors=blue]{hyperref}
\usepackage{amsthm}
\usepackage{enumitem}
\usepackage{tikz}
\usepackage{tikz-cd}
\usepackage{todonotes}
\usetikzlibrary{arrows.meta,
                decorations.markings, 
                 positioning}
\usetikzlibrary{topaths,calc}
\usepackage{pgfplots}
\pgfplotsset{compat=1.18}
\usepackage{adjustbox}
\usepackage{subcaption} 
\usepackage{caption} 
\usepackage{floatrow}
\usepackage{nicematrix}
\usepackage{algpseudocode}
\usepackage{appendix}
\usepackage{comment}

\usepackage{bbm}
\usepackage{bm}

\usepackage[capitalise]{cleveref}
\usepackage{placeins}
\usepackage{cases}

\usepackage{ifthen}

\newcommand{\Email}[1]{\ifthenelse{\equal{#1}{}}{}{\par\noindent {\rm E-mail: }{\it  #1} \par}}
\newcommand{\EmailMarked}[1]{\ifthenelse{\equal{#1}{}}{}{\par\noindent $^*$~{\rm E-mail: }{\it  #1} \par}}
\newcommand{\URLaddress}[1]{\ifthenelse{\equal{#1}{}}{}{\par\noindent {\rm URL: }{\tt  #1} \par}}
\newcommand{\URLaddressMarked}[1]{\ifthenelse{\equal{#1}{}}{}{\par\noindent $^*$~{\rm URL: }{\tt  #1} \par}}
\newcommand{\EmailD}[1]{\ifthenelse{\equal{#1}{}}{}{\par\noindent {$\phantom{^{{\rm a)}}}$~\rm E-mail: }{\it  #1} \par}}
\newcommand{\EmailDD}[1]{\ifthenelse{\equal{#1}{}}{}{\par\noindent {$\phantom{{}^{\dag^1}}$~\rm E-mail: }{\it  #1} \par}}

\newtheorem{resultx}{Result}

\newtheorem{theorem}{Theorem}[section]
\newtheorem{lemma}{Lemma}[section]
\newtheorem{proposition}{Proposition}[section]
\newtheorem{corollary}{Corollary}[section]
\newtheorem{definition}{Definition}[section]
\newtheorem{example}{Example}[section]
\newtheorem{remark}{Remark}[section]

\newcommand{\R}{\mathbb{R}}

\newcommand{\KK}{\mathcal{K}}
\newcommand{\LL}{\mathcal{L}}

\newcommand{\im}{\mathrm{im}}

\newcommand{\id}{\mathrm{Id}}

\newcommand{\sgn}{\,\operatorname{Sgn}}

\newcommand{\dist}{\operatorname{dist}}

\newcommand{\Ft}{\operatorname{Ft}}
\newcommand{\diam}{\operatorname{diam}}
\newcommand{\one}{\mathbf{1}}

\DeclareMathOperator*{\argmin}{arg\,min}
\DeclareMathOperator*{\spec}{spec}

\usepackage{authblk}
\date{} 

\title{Cheeger Inequalities for the Persistent Laplacian}
\author[1]{Magnus Bakke Botnan}
\author[1]{Rui Dong}

\affil[1]{Department of Mathematics, Vrije Universiteit Amsterdam}

\begin{document}
\maketitle

\begin{abstract}
We study Cheeger-type inequalities for persistent Laplacians associated with inclusions of simplicial complexes \(\KK\hookrightarrow \LL\). We introduce a persistent up \(p\)-Laplacian \(\Delta_{q,p,\mathrm{up}}^{\KK,\LL}\) for \(p\geq 1\). For \(p=2\), this recovers the usual persistent up Laplacian, while for \(p=1\) it yields a nonzero persistent Cheeger constant \(\varphi_q^{\KK,\LL}\). We prove a Cheeger-type inequality relating \(\varphi_q^{\KK,\LL}\) to the smallest nonzero eigenvalue of \(\Delta_{q,\mathrm{up}}^{\KK,\LL}\). This gives a persistent extension of recent work by Jost and Zhang \cite{jost2024cheeger}.

We then study two structured settings. Under a locally complete \(q\)-skeleton assumption on \(\KK\), we extend the complete-skeleton isoperimetric inequality of Parzanchevski--Rosenthal--Tessler \cite{MR3516884} to the persistent setting. For orientable \((q+1)\)-dimensional pseudomanifolds, we prove a Kron-type reduction of the persistent up Laplacian to a vertex- and edge-weighted graph Laplacian, possibly with Dirichlet boundary terms, and obtain two-sided Cheeger inequalities; this is related to the dual-graph perspective in the work of Steenbergen--Klivans--Mukherjee \cite{MR3194207}. We also describe the nonzero persistent Cheeger constant \(\varphi_q^{\KK,\LL}\) explicitly in terms of the dual graph in the non-branching pseudomanifold case.

Finally, we specialize our two constructions to graph inclusions and compare them with the graph-pair theory of Mémoli--Wan--Wang \cite{pers_lap}. Mémoli, Wan, and Wang introduced persistent Cheeger constants via Kron reduction/effective conductance and proved a corresponding persistent Cheeger inequality. Our constructions also yield persistent Cheeger constants and inequalities for graph inclusions, but through the $p$-Laplacian and complete-skeleton approaches described above rather than through Kron reduction. We establish precise relationships between the two persistent Cheeger constants arising from our constructions and the corresponding constants derived via Kron reduction.

\end{abstract}

\tableofcontents

\section{Introduction}\label{sec:introduction}

Cheeger inequalities play a fundamental role in graph theory, relating the combinatorial expansion properties of a graph to the smallest nonzero eigenvalue of its Laplacian. Several works have extended such inequalities to simplicial complexes, typically under additional assumptions such as orientability or completeness of the \(q\)-skeleton \cite{jost2024cheeger,MR3382297,MR3516884,MR3194207}. In this paper, we aim for a further generalization to the \emph{persistent} Laplacian.

To recall the graph-theoretic starting point, let \(G=(V,E)\) be a finite graph. For \(S\subseteq V\), write \(S^c:=V\setminus S\), and let
\[
E(S,S^c):=\{\,\{u,v\}\in E : u\in S,\ v\in S^c\,\}
\]
be the set of edges crossing from \(S\) to its complement.

\begin{definition}[{\cite[Section 2.5]{MR1421568}}]\label{def:graph_asymm_cheeger}
The Cheeger constant \(\varphi(G)\) associated with \(G\) is defined by
\begin{equation}\label{eq:graph_asymm_cheeger}
\varphi(G)=\min_{0<|S|\leq |V|/2}\frac{|E(S,S^c)|}{|S|}.
\end{equation}
\end{definition}

Thus \(\varphi(G)\) measures how efficiently one can separate the graph into two parts. This combinatorial quantity is intimately linked with the eigenvalues of the (unnormalized) \emph{graph Laplacian} \(L=II^T\), where \(I\) is the \emph{incidence matrix} defined by
\[
I_{ij}=
\begin{cases}
-1, & \text{if edge } j \text{ leaves vertex } i,\\
1, & \text{if edge } j \text{ enters vertex } i,\\
0, & \text{otherwise}.
\end{cases}
\]
\begin{theorem}[{\cite[Section 2.5]{MR1421568}}]
Let $G$ be a connected graph and let \(\lambda(G)\) be the second smallest eigenvalue of \(L\). Then
\begin{equation}\label{eq:graph_symm_cheeger_ineq}
\frac{\varphi^2(G)}{2\max\limits_v \deg(v)}\leq \lambda(G) \leq 2\varphi(G).
\end{equation}
\end{theorem}

Analogous bounds exist for higher eigenvalues in terms of how efficiently the graph can be partitioned into \(k\) pieces \cite{lee2014multiway}. Cheeger inequalities are therefore a cornerstone of spectral graph theory; see, for example, \cite{MR1421568,nica2018brief}.

\paragraph{From graphs to complexes.}
Extending Cheeger inequalities from graphs to simplicial complexes is subtle. In degree zero, a cut is simply a partition of vertices. In higher degree, the relevant objects are no longer only subsets of vertices: orientations, multiplicities, homology, and the interaction between adjacent simplices all enter into the Laplacian. Consequently, there is no unique higher-dimensional analogue of the graph Cheeger constant, and different definitions are appropriate under different assumptions. Recent work of Jost and Zhang \cite{jost2024cheeger} gives a systematic approach to this problem using the \(1\)-Laplacian and, more generally, nonlinear \(p\)-Laplacians. The \(1\)-Laplacian plays a distinguished role: while the ordinary Hodge Laplacian is an \(\ell^2\)-object, Cheeger constants are \(\ell^1\)-type quantities, and in the case \(p=1\) the corresponding Cheeger-type inequalities become equalities. This provides a bridge between the spectral theory of the \(2\)-Laplacian and the combinatorics of cuts. Based on this framework, Jost and Zhang obtain Cheeger-type inequalities for simplicial complexes, although the abstract nature of the resulting Cheeger constants makes their geometric interpretation difficult in full generality. We return to this setup in \cref{sec:cheeger_perslap}. For a broader overview of Cheeger inequalities for simplicial complexes, we refer the reader to the introduction of \cite{jost2024cheeger}.

\paragraph{Pairs of complexes.}
In this paper, our focus is the persistent Laplacian \cite{MR4164275,pers_lap,gen_pers_lap}, a generalization of the classical Laplacian to pairs of spaces. It was introduced independently by Lieutier et al. \cite{lapslides} and Wang et al. \cite{MR4164275}, and was later studied systematically by M\'emoli et al. \cite{pers_lap}, 
who proved the persistent Hodge/nullity theorem, developed algorithms for computing persistent Laplacians, related the persistent
 up Laplacian to Schur complements, and, in the graph-pair case, introduced the first persistent Cheeger constant and persistent Cheeger inequality for persistent Laplacians via the Kron reduction/effective conductance interpretation.

For an inclusion of simplicial complexes \(\KK\hookrightarrow \LL\), the \(q\)-th persistent Laplacian \(\Delta_q^{\KK,\LL}\) is defined as the sum of the persistent up Laplacian \(\Delta_{q,\mathrm{up}}^{\KK,\LL}\) and the standard down Laplacian \(\Delta_{q,\mathrm{down}}^\KK\) on \(\KK\).
A
 key property, 
 proved in this persistent setting in \cite[Theorem 2.7]{pers_lap},
 is that Eckmann’s discrete Hodge theorem admits a persistent analogue: the nullity of the $q$-th persistent Laplacian is equal to the $q$-th persistent Betti number of the inclusion $\KK\hookrightarrow \LL$.
Thus the barcode of persistent homology is encoded in the kernels of persistent Laplacians. Beyond the kernel, the spectra of persistent Laplacians contain finer geometric information about the underlying filtered space. The aim of this paper is to understand such spectral information through Cheeger-type inequalities.

The persistent Laplacian has become an active area of research, with recent developments in theoretical, applied, and algorithmic directions \cite{davies23c,wolf2025generalized,MR5053447,jones2025petls}. We refer the reader to \cite{wei2023} for a recent overview.

The goal of this paper is to extend several higher-dimensional Cheeger frameworks for simplicial complexes to the persistent-Laplacian setting, 
then restrict to $1$-dimensional cases and compare these different $1$-dimensional persistent Cheeger theories with 
the graph-pair persistent Cheeger theory developed in \cite{pers_lap}.
We focus on three existing approaches to Cheeger theory for simplicial complexes \cite{jost2024cheeger,MR3516884,MR3194207} and show that each admits a natural persistent extension. Similar extensions are likely to exist for other versions of Cheeger inequalities. We leave the development of a more unified framework for spectral inequalities of persistent Laplacians as an open direction in the discussion.

\subsection{Outline and contributions}
\label{subsec:overview_contributions}

We begin in \cref{sec.graphs} by recalling the graph-theoretic material needed later, including edge- and vertex-weighted graph Laplacians, incidence matrices, graph Cheeger constants, and Dirichlet boundary conditions. We also prove a Cheeger inequality for weighted graph Laplacians with Dirichlet boundary conditions. This result is used in the pseudomanifold setting, where the persistent Laplacian naturally reduces to a vertex- and edge-weighted graph Laplacian with possible boundary terms.

In \cref{sec.laplacians}, we recall the combinatorial Hodge Laplacian and the persistent Laplacian associated with an inclusion \(\KK\hookrightarrow\LL\). We also recall the Schur-complement formula for the persistent up Laplacian and show that this operator may be viewed as a discrete Dirichlet-to-Neumann map, providing a novel interpretation of the persistent Laplacian, which plays an integral role in \cref{sec.graphs}.

The main contributions of the paper are organized around four themes. 
\paragraph{Theme 1: a general framework.} In \cref{sec:cheeger_perslap}, we extend the \(1\)-Laplacian approach to Cheeger inequalities of Jost--Zhang \cite{jost2024cheeger} to the persistent setting. We show that the persistent up Laplacian can be written as \(BPB^T\), where \(P\) is an orthogonal projection. This projection formula allows us to define a persistent \(p\)-Laplacian and, for \(p=1\), a nonzero persistent Cheeger constant \(\varphi_q^{\KK,\LL}\). 
\begin{resultx}
\label{res:general_persistent_cheeger}
Let \(\lambda_{\min}^+\) denote the smallest nonzero eigenvalue of the persistent up Laplacian $\Delta_{q,\mathrm{up}}^{\KK,\LL}$
associated with an inclusion \(\KK\hookrightarrow \LL\) of simplicial complexes. Then
\[
\frac{(\varphi_q^{\KK,\LL})^2}{|S_{q+1}^{\LL}|}
\leq
\lambda_{\min}^+
\leq
|S_q^\KK|(\varphi_q^{\KK,\LL})^2,
\]
 where $|S^\LL_{q+1}|$ denotes the number of $(q+1)$-simplices in $\LL$, and similarly for $|S_q^\KK|$. 
When \(\KK=\LL\), this gives a Cheeger inequality for the ordinary up Laplacian; see \cref{cor:hodge_cheeger_ineq}.
\end{resultx}
Our arguments are similar to those in \cite{jost2024cheeger} although we work with a slightly different eigenvalue; the approach applies to both settings \cref{remark.connection.to.jost}. 

The strength of this result is its generality: it holds without additional assumptions on the simplicial complexes. At this level of generality, however, the bounds retain some coarse features, since the constants involve the number of simplices in the complexes, and the Cheeger constant itself is not always easy to interpret geometrically. In \cref{sec:psd_mfld}, we show that in the pseudomanifold setting the same quantity admits a concrete dual-graph interpretation.

That said, it is natural to ask whether additional constraints on the simplicial complexes lead to Cheeger inequalities closer to those familiar from graph theory. This is the motivation behind the next two themes.

\paragraph{Theme 2: complete skeleta.} In \cref{sec:loc_cmplt_skl}, we extend the isoperimetric inequality of Parzanchevski--Rosenthal--Tessler \cite{MR3516884} for complexes with complete skeleta to inclusions \(\KK\hookrightarrow\LL\), under the assumption that each connected component of \(\KK\) has a complete \(q\)-skeleton. The corresponding Cheeger constant is a persistent version of the complete-skeleton constant from \cite{MR3516884} which is a generalization of the graph Cheeger constant:
\[
h_q^{\LL}
=
\min_{V=\coprod_{i=0}^{q+1}A_i}
\frac{|V|\cdot |F(A_0,\ldots,A_{q+1})|}
{|A_0|\cdots |A_{q+1}|}.
\]
Here $F(A_0, \ldots, A_{q+1})$ counts the number of $(q+1)$-simplices with vertices in different sets in the partition of the vertices $V$ in $\LL$. Our persistent Cheeger constant $h_q^{\KK,\LL}$ is rather involved in general, but it has a clear description when $\KK$ is connected; see \cref{eq:nt_cheeger_inq_graph}.

\begin{resultx}
\label{res:locally_complete_skeletons}
Let \(\KK\hookrightarrow \LL\) be an inclusion of simplicial complexes such that each connected component of \(\KK\) has a complete \(q\)-skeleton. Let
$\lambda_{+}$
be the smallest \emph{non-trivial} eigenvalue of the persistent up Laplacian $\Delta_{q,\mathrm{up}}^{\KK,\LL}$. Then,
\[
\lambda_{+}\leq h_q^{\KK,\LL};
\]
see \cref{prop:nt_cheeger_ineq_cmplx}. When \(\KK=\LL\), this recovers the non-persistent complete-skeleton upper Cheeger inequality.
\end{resultx}
Note that, in contrast to $\lambda_{\min}^+$, $\lambda_{+}$ can be $0$, and this happens precisely when the induced map $H_q(\KK;\R)\to H_q(\LL;\R)$ of homology vector spaces is non-trivial. If the induced map is trivial, the two eigenvalues coincide. 

\paragraph{Theme 3: pseudomanifolds.} In \cref{sec:psd_mfld}, inspired by the dual-graph viewpoint in the work of Steenbergen--Klivans--Mukherjee \cite{MR3194207}, we study persistent Laplacians over orientable pseudomanifolds. In this setting, the persistent up Laplacian admits a Kron-type reduction to an edge- and vertex-weighted graph Laplacian on the dual graph associated with the \((q+1)\)-chains of \(\LL\) relative to \(\KK\). This Kron reduction is significantly easier to interpret than the Kron reduction to a hypergraph Laplacian for (potentially unoriented) pseudomanifolds in \cite{MR5053447}.

Importantly, our Kron reduction leads to two natural Cheeger constants: \(h_{\mathrm{vol}}\) in the boundaryless case, and \(h_{\mathrm{vol}}^{\mathrm{Bdy}}\) when boundary terms are present. The term $\Delta$ is the maximal weighted degree (in the dual graph), normalized by the corresponding vertex weight.

\begin{resultx}
\label{res:pseudomanifold_cheeger}
Let \(\KK\hookrightarrow \LL\) be an inclusion of simplicial complexes, where \(\LL\) is an orientable \((q+1)\)-dimensional pseudomanifold. Let \(\lambda_{\min}^+\) denote the smallest nonzero eigenvalue of
the persistent up Laplacian \(\Delta_{q,\mathrm{up}}^{\KK,\LL}\). 
\begin{enumerate}
    \item If \(\LL\) is boundaryless, then
    \[
    \frac{h_{\mathrm{vol}}^2}{2\Delta}\leq \lambda_{\min}^+\leq 2h_{\mathrm{vol}}.
    \]
    \item Otherwise,
    \[
    \frac{(h^{\mathrm{Bdy}}_{\mathrm{vol}})^2}{2\Delta}\leq \lambda_{\min}^+\leq h^{\mathrm{Bdy}}_{\mathrm{vol}}.
    \]
\end{enumerate}
See \cref{thm.pers.cheeger}.
\end{resultx}
We also show that the orientability assumption is essential for the lower bound.

\paragraph{Theme 4: graphs.} In \cref{sec.graph.comparisons}, we compare the different persistent Cheeger constants in the special case of graph inclusions. M\'emoli et al. introduced a persistent Cheeger constant for graph pairs via Kron reduction/effective conductance and proved a persistent Cheeger inequality for it \cite[Definition 4.18 and Theorem 4.21]{pers_lap}.
We compare those graph-pair constants with the ones arising from the general persistent definitions above.

\begin{resultx}
\label{res:graph_comparison}
For an inclusion of graphs \(H\hookrightarrow G\), we have
\[
h_{\operatorname{Kron}}^{H,G}\leq h_0^{H,G},
\qquad
\varphi_{\operatorname{Kron}}^{H,G}\leq \varphi_0^{H,G};
\]
see \cref{prop:ineq_cheger_kron_nonzero_1,prop:ineq_cheger_kron_nonzero}.
\end{resultx}

These inequalities are not equivalences in general. Under additional assumptions on the inclusion \(H\hookrightarrow G\), namely that \(H^c\) is independent and every vertex of \(H\) has exactly one neighbor in \(H^c\), we obtain the two-sided comparison
\[
\varphi_{\operatorname{Kron}}^{H,G}
\leq
\varphi_0^{H,G}
\leq
2\varphi_{\operatorname{Kron}}^{H,G};
\]
see \cref{prop:kron_nonzero_cheeger_equiv}.

\section{Spectral graph theory}
\label{sec.graphs}

Throughout this section, let \(G=(V,E)\) be a finite connected graph with vertex set \(V=\{1,\ldots,m\}\). We allow symmetric edge weights \(w_{ij}=w_{ji}\geq 0\) for \(i\neq j\), with \(w_{ij}=0\) whenever \(i\) and \(j\) are not adjacent.

\subsection{Edge- and vertex-weighted graph Laplacians}
\label{sec:edge-vertex-laplacian}

\begin{definition}
The \emph{edge-weighted graph Laplacian} of \(G\) is the matrix
\(L^{\mathrm{edge}}\in\mathbb{R}^{m\times m}\) with entries
\[
L^{\mathrm{edge}}_{ij}
=
\begin{cases}
-w_{ij}, & i\neq j,\\[4pt]
\displaystyle\sum_{k\neq i}w_{ik}, & i=j.
\end{cases}
\]
Equivalently, \(L^{\mathrm{edge}}\) is the unique symmetric matrix whose quadratic form is the Dirichlet energy
\[
g^T L^{\mathrm{edge}}g
=
\sum_{i<j}w_{ij}(g_i-g_j)^2,
\qquad g\in\mathbb{R}^m.
\]
In particular, all row sums vanish:
\[
\sum_{j=1}^m L^{\mathrm{edge}}_{ij}=0
\qquad
\text{for all } i.
\]
\end{definition}
In terms of the weighted incidence matrix \(B=I\Omega^{1/2}\), where \(w_e=w_{ij}\), this becomes
\[
L^{\mathrm{ev}}
=
W^{-1/2}BB^T W^{-1/2}.
\]

We now introduce vertex weights by conjugating \(L^{\mathrm{edge}}\) with a positive diagonal matrix.

\begin{definition}
\label{def.vertex.edge.weighted.laplacian}
Let \(V_1,\ldots,V_m>0\) be \emph{vertex weights}, and set
\[
W:=\operatorname{diag}(V_1,\ldots,V_m).
\]
The \emph{edge- and vertex-weighted graph Laplacian} associated with
\((G,(w_{ij}),(V_i))\) is the symmetric matrix
\[
L^{\mathrm{ev}}:=W^{-1/2}L^{\mathrm{edge}}W^{-1/2}.
\]
\end{definition}

To interpret this definition, let \(x\in\mathbb{R}^m\) and write \(g:=W^{-1/2}x\). Then
\[
x^T L^{\mathrm{ev}}x
=
g^T L^{\mathrm{edge}}g
=
\sum_{i<j}w_{ij}(g_i-g_j)^2,
\qquad
x^Tx
=
g^TWg
=
\sum_{i=1}^m V_i g_i^2.
\]
Thus the energy is still the usual edge-based Dirichlet energy, while the norm is weighted by the vertex measure. This differs from the standard normalized graph Laplacian, where the vertex weights are chosen to be the degrees; here the vertex weights are arbitrary positive numbers.

\subsection{Dirichlet boundary conditions}
\label{sec.dir}
Let
$B \subseteq V$ be a nonempty set of \emph{boundary vertices}.
The complementary set $I := V \setminus B$
is called the set of \emph{interior vertices}. In addition to \cite{MR1421568}, \cite[Chamber 0]{MR4383783} is an excellent resource for graphs with Dirichlet boundary conditions.

\begin{definition}
A function $f \colon V \to \mathbb{R}$ is said to satisfy a
\emph{Dirichlet boundary condition} on $B$ if
\[
f(i) = 0 \qquad \text{for all } i \in B.
\]
Equivalently, in vector form, the components of $f$ indexed by $B$
are fixed to zero.
\end{definition}

Given any (weighted or unweighted) Laplacian matrix $L$ on $V$, the \emph{Dirichlet
Laplacian} is obtained by restricting $L$ to the interior vertices:
\[
L_{II}
\quad \in \mathbb{R}^{|I|\times|I|},
\]
namely, the principal submatrix of $L$ with both rows and columns indexed
by $I$.

\begin{definition}
The Dirichlet Laplacian associated with the boundary set $B$ is
the matrix $L_{II}$.
It acts on functions $f$ satisfying the Dirichlet condition
$f|_{B} = 0$ via
\[
(L_{II} f_{I})_i
= (Lf)(i),
\qquad i\in I.
\]
\end{definition}
The proof of the following proposition is straightforward. We refer the reader to \cite{MR4383783} for a classification of matrices arising from various types of graph Laplacians.
\begin{proposition}\label{prop.dualgraph}
Let $X$ be an $m\times m$ symmetric matrix with strictly positive diagonal entries and non-positive off-diagonal entries, and let $W$ be a diagonal matrix with positive entries. Assume that all row sums of $X$ are non-negative. Construct a weighted graph $G$ on vertices $\{1,\dots,m\}$ by assigning an edge between $i$ and $j$ of weight $w_{ij} := -X_{ij}$ whenever $i\neq j$ and $X_{ij}<0$.

\begin{enumerate}
\label{prop.reduce.to.graph.lap}
    \item If all row sums of $X$ are equal to $0$, then $X$ is the (edge-weighted) graph Laplacian $L_G$ of $G$ (i.e.,\ the $0$-th up Laplacian). 
    
    \item If all row sums of $X$ are non-negative and at least one is strictly positive, then $X$ is the graph Laplacian $L^\mathrm{Dir}_G$ of the weighted graph $G$ with Dirichlet boundary conditions at the vertices whose row sums are positive.
\end{enumerate}
In both cases, the matrix $W^{-1/2}XW^{-1/2}$, is a \emph{vertex and edge} weighted graph Laplacian (with boundary conditions), where the vertex weight at index $i$ is $W_{ii}$; see \cref{def.vertex.edge.weighted.laplacian}.
\end{proposition}

We see that 
any symmetric matrix with positive diagonal entries and non-positive 
off-diagonal entries can be realized as the Dirichlet Laplacian of a 
weighted graph. Indeed, given such a matrix $A \in \mathbb{R}^{m \times m}$, 
one may construct an extended weighted graph by adding a set of dummy 
boundary vertices and assigning suitable edge weights so that every row sum 
of the full Laplacian vanishes. The original matrix $A$ is then obtained 
precisely as the Dirichlet restriction of this Laplacian to the interior 
vertices. 

\subsection{Cheeger bounds for edge- and vertex-weighted graphs}
Let \(X\) be an \(m\times m\) symmetric matrix with strictly positive diagonal entries, non-positive off-diagonal entries, and non-negative row sums. Let \(W=\operatorname{diag}(\mu_1,\ldots,\mu_m)\) be a diagonal matrix with \(\mu_i>0\). From \cref{prop.reduce.to.graph.lap}, we regard \(X\) as a weighted graph Laplacian, potentially with Dirichlet boundary conditions. Thus, for \(i\neq j\), the (interior) edge weight between \(i\) and \(j\) is
\[
w_{ij}:=-X_{ij},
\]
whenever \(X_{ij}<0\), and \(w_{ij}=0\) otherwise. The row-sum $
b_i:=\sum_{j=1}^m X_{ij}$ records the total boundary weight attached to the interior vertex \(i\). Equivalently,
\[
X_{ii}=\sum_{j\neq i}w_{ij}+b_i.
\]

The edge- and vertex-weighted Dirichlet Laplacian associated with this data is
\[
M:=W^{-1/2}XW^{-1/2}.
\]

We now introduce the relevant Cheeger constants. 

\begin{definition}

For \(S\subseteq \{1,\ldots,m\}\), define
\[
\operatorname{Vol}(S):=\sum_{i\in S}\mu_i,\qquad
\operatorname{Cut}(S):=\sum_{\substack{i\in S\\ j\notin S}}w_{ij},\qquad
\operatorname{Bdy}(S):=\sum_{i\in S}b_i.
\]
If all row sums of \(X\) vanish, set
\[
h_{\mathrm{vol}}:=
\min_{\substack{\varnothing\neq S\subseteq \{1,\ldots,m\}\\ \operatorname{Vol}(S)\leq V_{\mathrm{tot}}/2}}
\frac{\operatorname{Cut}(S)}{\operatorname{Vol}(S)}.
\]
If at least one row sum of \(X\) is positive, set
\[
h_{\mathrm{vol}}^{\mathrm{Bdy}}:=
\min_{\varnothing\neq S\subseteq \{1,\ldots,m\}}
\frac{\operatorname{Cut}(S)+\operatorname{Bdy}(S)}{\operatorname{Vol}(S)}.
\]
\end{definition}

\begin{remark}
Geometrically, the constant \(h_{\mathrm{vol}}\) corresponds to the usual Cheeger problem for a graph without Dirichlet boundary: it measures how efficiently one can separate the vertex set \(S\) from its complement using only \emph{interior} edges. In contrast, \(h_{\mathrm{vol}}^{\mathrm{Bdy}}\) is the appropriate quantity when the graph has Dirichlet boundary contributions, equivalently when at least one row sum of \(X\) is positive. In that case, \(S\) can be separated not only along edges connecting it to its complement, but also along the boundary encoded by the row-sum surplus. Hence, different Cheeger constants are appropriate in the two cases.
\end{remark}
We now state a two-sided Cheeger inequality for edge- and vertex-weighted graphs, possibly with Dirichlet boundary conditions. Many related results appear in the literature, including Cheeger inequalities for graphs with Dirichlet boundary conditions \cite{MR1421568}. However, we were unable to find a statement in precisely the form needed here. Since this estimate is used in the proof of the pseudomanifold result below, we include the details in \cref{sec.app.proof.dirichlet}.

Let \(V_{\mathrm{tot}}:=\operatorname{Vol}(\{1,\ldots,m\})\), and set
\[
\Delta:=\max_{1\leq i\leq m}\frac{X_{ii}}{\mu_i},
\]
which is the maximal weighted degree, normalized by the vertex weight.

\begin{theorem}
\label{thm:weighted_graph_cheeger_dirichlet}
Let \(X,W,M,\Delta\) be as above, and assume that the interior graph associated with \(X\) is connected.
\begin{enumerate}
    \item (Boundaryless) If all row sums of \(X\) are equal to \(0\), then
    \[
    \frac{h_{\mathrm{vol}}^2}{2\Delta}\leq \lambda_1(M)\leq 2h_{\mathrm{vol}},
    \]
    where \(\lambda_1(M)\) denotes the smallest nonzero eigenvalue of \(M\).
    \item (Dirichlet) If at least one row sum of \(X\) is positive, then \(M\) is positive definite and
    \[
    \frac{(h_{\mathrm{vol}}^{\mathrm{Bdy}})^2}{2\Delta}\leq \lambda_{1}(M)\leq h_{\mathrm{vol}}^{\mathrm{Bdy}},
    \]
    where \(\lambda_1(M)\) denotes the smallest eigenvalue of \(M\).
\end{enumerate}
\end{theorem}

\section{Combinatorial Laplacians}
\label{sec.laplacians}
The graph Laplacian is the degree-zero instance of the Hodge Laplacian on a simplicial complex. Precisely, for $\LL$ a simplicial complex, we denote by $C_q^\LL$ the $q$-chains of $\LL$ over real numbers $\mathbb{R}$,and equip each $C_{q}^\LL$ an inner product structure. While the definitions carry through for arbitrary inner products, we shall focus on the unweighted case, i.e., when $\langle \sigma_i, \sigma_j\rangle = \delta_{ij}$ for $q$-simplices $\sigma_i$ and $\sigma_j$. 

The \emph{combinatorial Hodge $q$-Laplacian} is defined by
\[
\Delta_{q}^\LL=
\underbrace{\partial_{q+1}^\LL(\partial_{q+1}^\LL)^*}_{\text{up}} + \underbrace{(\partial_q^\LL)^* \partial_q^\LL}_{\text{down}}.
\]
We call the first part the \emph{ up Laplacian} and the second part the  \emph{down Laplacian}, and denote them by $\Delta_{q, \text{up}}^\LL$ and $\Delta_{q, \text{down}}^\LL$ respectively. These operators are often called Eckmann Laplacians, after Eckmann's discrete Hodge theorem \cite{MR0013318}, and were studied extensively for different choices of inner products in \cite{MR3077874}. Their kernels encode homology, while their nonzero eigenvalues reflect finer quantitative information about the way \(q\)-simplices are connected through \((q+1)\)-simplices. 
\paragraph{Decomposing the spectrum}
For a self-adjoint operator \(A\) on a finite-dimensional inner product space, we write $\spec(A)$ for the multiset of eigenvalues of \(A\), counted with algebraic multiplicity. We write $\spec_{\neq 0}(A)$
for the multiset obtained by removing the zero eigenvalues.

It is well known (see e.g., the introduction in \cite{MR3077874}) that the nonzero spectra of \(\Delta_q^\LL\), \(\Delta_{q,\mathrm{up}}^\LL\), and \(\Delta_{q,\mathrm{down}}^\LL\) are positive and satisfy
\[
\spec_{\neq 0}(\Delta_q^\LL)
=
\spec_{\neq 0}(\Delta_{q,\mathrm{up}}^\LL)
\sqcup
\spec_{\neq 0}(\Delta_{q,\mathrm{down}}^\LL),
\]
where \(\sqcup\) denotes disjoint union of multisets. Moreover,
\[
\spec_{\neq 0}(\Delta_{q,\mathrm{up}}^\LL)
=
\spec_{\neq 0}(\Delta_{q+1,\mathrm{down}}^\LL).
\]
(This is also a direct consequence of \cref{prop:laplaceusefulLA}.)

Therefore, it is enough to study one of the up or down families. In particular, much of the earlier work focuses on bounding the nonzero eigenvalues of the up Laplacian \(\Delta_{q,\mathrm{up}}^\LL\).

\subsection{The persistent Laplacian}
The persistent Laplacian was first independently introduced by Lieutier et al. in \cite{lapslides} and by Wang et al. in \cite{MR4164275}, 
and then systematically studied by M\'{e}moli et al. in \cite{pers_lap}.
We recall the persistent Laplacian with respect to a pair of simplicial complexes $\KK\subset \LL$.
We refer to \cite{MR4164275, pers_lap} for more details.
Let $C_q^\KK$ (resp. $C_q^\LL$) be the $q$-chain space with coefficients in $\R$ of $\KK$ (resp. $\LL$), 
and $\partial_{q}^\KK: C_q^\KK\to C_{q-1}^\KK$ (resp. $\partial_{q}^\LL: C_q^\LL\to C_{q-1}^\LL$) be the $q$-th boundary operator over $\KK$ (resp. $\LL$).
We always assume that the simplicial complex $\LL$ is unweighted,
i.e.,
each simplex in $\LL$ has the uniform weight $1$.
We then equip the $q$-chain space $C_q^\LL$ with the standard inner product structure for each $q$,
and $C_q^\KK$ is equipped with the restricted inner product structure.
Let 
\[C_q^{\LL, \KK}:= \{c\in C_q^\LL: \partial_q^\LL(c)\in C_{q-1}^\KK\}\subset C_q^\LL,\]
this sub-chain space $C_q^{\LL, \KK}$ is also endowed with the restricted inner product structure, 
and we denote by $\partial_{q}^{\LL, \KK}=\partial_{q}^\LL\big|_{C_{q}^{\LL, \KK}}$ the restriction of the boundary operator $\partial_{q}^\LL$ over $C_{q+1}^{\LL, \KK}$.
The \emph{$q$-th persistent Laplacian} $\Delta_{q}^{\KK, \LL}: C_q^\KK\to C_q^\KK$ is defined as (see \cref{fig_dig_per_lap})
\[
 \Delta_{q}^{\KK, \LL}\coloneqq \partial_{q+1}^{\LL, \KK}\circ \left(\partial_{q+1}^{\LL, \KK}\right)^\ast +
 \left(\partial_q^\KK\right)^\ast \circ \partial_q^\KK,
\]
here $\left(\partial_{q+1}^{\LL, \KK}\right)^\ast$ and $\left(\partial_q^\KK\right)^\ast$ refer to the adjoints of $\partial_{q+1}^{\LL, \KK}$ and $\partial_q^\KK$, respectively. 

\begin{figure}
\centering
\begin{tikzcd}
\centering
     C_{q+1}^\KK\arrow{rr}{\partial_{q+1}^\KK}\arrow[hookrightarrow,dashed,gray]{dd} && C_q^\KK\arrow[rr,shift left=.75ex,blue,"\partial_q^\KK"]\arrow[hookrightarrow,dashed,gray]{dd}\arrow[dl,blue,"\left(\partial_{q+1}^{\LL, \KK}\right)^*" {yshift=-2pt, xshift=-7pt}] && C_{q-1}^\KK\arrow[hookrightarrow,dashed,gray]{dd}\arrow[ll,shift left=.75ex,blue,"\left(\partial_q^\KK\right)^*"]\\
      &C_{q+1}^{\LL, \KK}\arrow[hookrightarrow,dashed,gray]{dl}\arrow[ur,shift left=.75ex,blue,"\partial_{q+1}^{\LL, \KK}"]&& \,\,\,\,\,\,\,\,\,\, & \\
        C_{q+1}^\LL\arrow{rr}{\partial_{q+1}^\LL}  && C_q^\LL\arrow{rr}{\partial_q^\LL} && C_{q-1}^\LL
\end{tikzcd}
\caption{The relevant morphisms in the definition of the persistent Laplacian $\Delta_{q}^{\KK, \LL}$.}
\label{fig_dig_per_lap}
\end{figure}
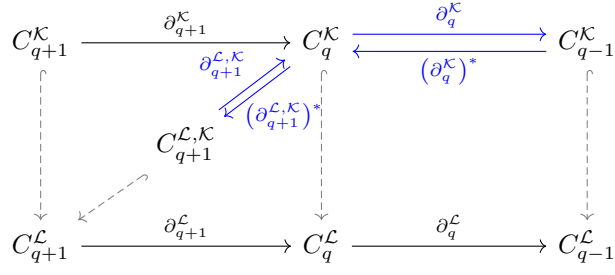

\begin{proposition}[{\cite[Theorem 3.1.]{pers_lap}}]\label{prop:pers_up_lap}
Assume that $\dim(C_{q+1}^{\LL, \KK})>0$.
Choose any basis of $C_{q+1}^{\LL, \KK}\subset C_{q+1}^{\LL}$ represented by a column matrix $Z$, 
and let $B_{q+1}^{\LL, \KK}$ be the matrix representation of $\partial_{q+1}^{\LL, \KK}$ with respect to this basis.
Then the matrix representation of persistent up Laplacian is expressed as follows:
\begin{equation}\label{eq:pers_up_lap}
\Delta_{q, \mathrm{up}}^{\KK, \LL}
=
B_{q+1}^{\LL, \KK}(Z^TZ)^{-1}(B_{q+1}^{\LL, \KK})^T.
\end{equation}
\end{proposition}

As in the standard combinatorial setting, the nonzero spectrum of the persistent Laplacian decomposes into its up and down parts:
\[
\spec_{\neq 0}\bigl(\Delta_q^{\KK,\LL}\bigr)
=
\spec_{\neq 0}\bigl(\Delta_{q,\mathrm{up}}^{\KK,\LL}\bigr)
\sqcup
\spec_{\neq 0}\bigl(\Delta_{q,\mathrm{down}}^\KK\bigr).
\]
Cheeger inequalities for the down part both exist in the literature and follow from our results by setting \(\KK=\LL\), using the fact that
\[
\spec_{\neq 0}\bigl(\Delta_{q,\mathrm{down}}^\LL\bigr)
=
\spec_{\neq 0}\bigl(\Delta_{q-1,\mathrm{up}}^\LL\bigr).
\]
Thus, Cheeger-type inequalities for the persistent up part, combined with the spectral decomposition above, yield Cheeger-type inequalities for the total persistent Laplacian. We therefore focus on the up part from here onward.

\subsection{Schur complements and Kron reductions}
\label{subsec.schur}

The authors in \cite{pers_lap} provided another equivalent expression for $\Delta_{q,\mathrm{up}}^{\KK,\LL}$ using Schur complements. Let $\LL$ be a simplicial complex and let $\KK\subseteq\LL$ be a subcomplex. We write
\[
\mathcal{B}:=S_q^\KK,\qquad \mathcal{I}:=S_q^\LL\setminus S_q^\KK,
\]
where $S_q^\KK$ and $S_q^\LL$ denote the sets of $q$-simplices of $\KK$ and $\LL$, respectively. Thus $\mathcal{B}$ indexes the $q$-simplices retained in $\KK$, while $\mathcal{I}$ indexes the remaining $q$-simplices of $\LL$. With respect to the decomposition of $C_q^\LL$ into coordinates indexed by $\mathcal{B}$ and $\mathcal{I}$, write
\[
\Delta_{q,\mathrm{up}}^\LL=\begin{bNiceMatrix}[first-row,code-for-first-row=\scriptstyle,first-col,code-for-first-col=\scriptstyle]
& \mathcal{B} & \mathcal{I} \\
\mathcal{B} & L_{\mathcal{B}\mathcal{B}} & L_{\mathcal{B}\mathcal{I}} \\
\mathcal{I} & L_{\mathcal{I}\mathcal{B}} & L_{\mathcal{I}\mathcal{I}}
\end{bNiceMatrix}.
\]
Then the persistent up Laplacian is the Schur complement of the interior block:
\[
\Delta_{q,\mathrm{up}}^{\KK,\LL}=L_{\mathcal{B}\mathcal{B}}-L_{\mathcal{B}\mathcal{I}}L_{\mathcal{I}\mathcal{I}}^\dagger L_{\mathcal{I}\mathcal{B}},
\]
where $L_{\mathcal{I}\mathcal{I}}^\dagger$ denotes the Moore--Penrose pseudoinverse of $L_{\mathcal{I}\mathcal{I}}$. This observation is important for computational purposes, and it also provides a concrete way to recognize the persistent Laplacian of graphs as a Kron reduction.

\begin{proposition}[{\cite[Lemma 2.1]{MR3017573}, \cite[Proposition 4.10]{pers_lap}}]
\label{lem.kron}
Let $H\hookrightarrow G$ be an inclusion of graphs. Then $\Delta_{0,\mathrm{up}}^{H,G}$ equals the graph Laplacian of an edge-weighted graph with the same vertex set as $H$.
\end{proposition}

While interpreting the eigenvalues of the persistent Laplacian is difficult in general, this result shows that the graph case is more transparent. This reduction is important in electrical engineering \cite{MR3017573}. For $q>0$, there is no Kron reduction in general: one cannot always realize the persistent up Laplacian of an inclusion $\KK\hookrightarrow\LL$ as the simplicial up Laplacian of a single simplicial complex \cite{pers_lap}. We shall see below, however, that such a reduction does exist under additional structural assumptions on $\LL$.

\subsection{Energy-minimizing extensions and Dirichlet-to-Neumann}

In this section, we explain how the persistent up Laplacian arises from a Dirichlet-to-Neumann construction and how its Schur-complement form encodes an energy-minimizing extension principle. 
The
 Schur-complement and Kron reduction/effective conductance viewpoint was developed in \cite{pers_lap}.
 Here we formulate the closely related Dirichlet-to-Neumann interpretation explicitly and use it in the Cheeger-type arguments.
The boundary/interior notation $\mathcal{B},\mathcal{I}$ is the one introduced in \cref{subsec.schur}; it is analogous to the graph notation used in \cref{sec.dir}.

\begin{lemma}\label{lemma:schur_optimization}
For every $u_{\mathcal{B}}\in C_q^\KK$, one has
\[
u_{\mathcal{B}}^T\Delta_{q,\mathrm{up}}^{\KK,\LL}u_{\mathcal{B}}=\inf_{u_{\mathcal{I}}}\begin{bmatrix}u_{\mathcal{B}}\\u_{\mathcal{I}}\end{bmatrix}^T\Delta_{q,\mathrm{up}}^\LL\begin{bmatrix}u_{\mathcal{B}}\\u_{\mathcal{I}}\end{bmatrix},
\]
where the infimum is taken over all interior $q$-chain values indexed by $\mathcal{I}$.
\end{lemma}

\begin{proof}
This is the quadratic-minimization characterization of the Schur complement, applied to the block decomposition in \cref{subsec.schur}; see \cite[Appendix A.5.5]{MR2061575}.
\end{proof}

Spelled out using the adjoint property,
\[
\begin{bmatrix}u_{\mathcal{B}}\\u_{\mathcal{I}}\end{bmatrix}^T\Delta_{q,\mathrm{up}}^\LL\begin{bmatrix}u_{\mathcal{B}}\\u_{\mathcal{I}}\end{bmatrix}=\left\|(\partial_{q+1}^\LL)^*\begin{bmatrix}u_{\mathcal{B}}\\u_{\mathcal{I}}\end{bmatrix}\right\|_2^2,
\]
which measures how the $q$-chain values vary across $(q+1)$-simplices, with orientation taken into account. Hence, once the boundary values are prescribed, Schur-complement elimination determines the energy-minimizing extension to $\LL$ (see \cite[Appendix A.5.5]{MR2061575}):
\[
u_{\mathcal I}^{\min}=-L_{\mathcal I\mathcal I}^{\dagger}L_{\mathcal I\mathcal B}u_{\mathcal B}+r,\qquad r\in\ker L_{\mathcal I\mathcal I},
\]
and the quadratic form of the persistent up Laplacian gives the corresponding minimal energy. For a graph, this reduces to extending function values from boundary vertices to interior vertices so that the total variation along edges is minimized. The lemma above extends this same principle to simplicial complexes. More generally, the same argument applies to any symmetric positive semidefinite operator on a finite-dimensional vector space.

\subsubsection{Dirichlet-to-Neumann}

This subsection is not needed for the main arguments of the paper and may be skipped. It is included because the Dirichlet-to-Neumann perspective on persistent Hodge theory may be of independent interest.

The \emph{Dirichlet-to-Neumann map} on $q$-chains is defined by
\[
\Lambda:C_q^\KK\longrightarrow C_q^\KK,\qquad u_{\mathcal{B}}\longmapsto\left(\Delta_{q,\mathrm{up}}^\LL u\right)\big|_{\mathcal{B}},
\]
where $u\in C_q^\LL$ solves
\[
\Delta_{q,\mathrm{up}}^\LL u(\sigma)=0\quad\text{for }\sigma\in\mathcal{I},\qquad u(\sigma)=u_{\mathcal{B}}(\sigma)\quad\text{for }\sigma\in\mathcal{B}.
\]
Thus the interior $q$-chain values are \emph{harmonic}, i.e. trivial under the up Laplacian, while the boundary values are prescribed. With respect to the block decomposition from \cref{subsec.schur}, the Dirichlet problem is equivalently described by
\begin{equation}\label{eq:mtx_form_DtN}
\begin{bmatrix}L_{\mathcal{B}\mathcal{B}}&L_{\mathcal{B}\mathcal{I}}\\L_{\mathcal{I}\mathcal{B}}&L_{\mathcal{I}\mathcal{I}}\end{bmatrix}\begin{bmatrix}u_{\mathcal{B}}\\u_{\mathcal{I}}\end{bmatrix}=\begin{bmatrix}z\\0\end{bmatrix},
\end{equation}
where $z\in C_q^\KK$ is the resulting boundary flux, so $\Lambda u_{\mathcal{B}}=z$. From \cref{lemma:schur_optimization} and the ensuing discussion,
\[
\inf_{u_{\mathcal I}}\begin{bmatrix}u_{\mathcal B}\\u_{\mathcal I}\end{bmatrix}^T\Delta_{q,\mathrm{up}}^\LL\begin{bmatrix}u_{\mathcal B}\\u_{\mathcal I}\end{bmatrix}
=
u_{\mathcal B}^T\Delta_{q,\mathrm{up}}^{\KK,\LL}u_{\mathcal B},
\]
with minimizing interior values
\[
u_{\mathcal I}^{\min}=-L_{\mathcal I\mathcal I}^{\dagger}L_{\mathcal I\mathcal B}u_{\mathcal B}+r,\qquad r\in\ker L_{\mathcal I\mathcal I}.
\]
This minimizing extension is harmonic on the interior. Indeed, since $\Delta_{q,\mathrm{up}}^\LL$ is positive semidefinite, \cite[Theorem~1]{MR245582} gives $\ker L_{\mathcal I\mathcal I}\subseteq\ker L_{\mathcal B\mathcal I}$. If $h\in\ker L_{\mathcal I\mathcal I}$, then $L_{\mathcal B\mathcal I}h=0$, and hence
\[
h^TL_{\mathcal I\mathcal B}u_{\mathcal B}=u_{\mathcal B}^TL_{\mathcal B\mathcal I}h=0.
\]
Thus $L_{\mathcal I\mathcal B}u_{\mathcal B}\in(\ker L_{\mathcal I\mathcal I})^\perp=\operatorname{im}L_{\mathcal I\mathcal I}$, and therefore
\[
L_{\mathcal I\mathcal I}L_{\mathcal I\mathcal I}^{\dagger}L_{\mathcal I\mathcal B}u_{\mathcal B}=L_{\mathcal I\mathcal B}u_{\mathcal B}.
\]
It follows that
\[
L_{\mathcal I\mathcal B}u_{\mathcal B}+L_{\mathcal I\mathcal I}u_{\mathcal I}^{\min}
=
L_{\mathcal I\mathcal B}u_{\mathcal B}-L_{\mathcal I\mathcal I}L_{\mathcal I\mathcal I}^{\dagger}L_{\mathcal I\mathcal B}u_{\mathcal B}+L_{\mathcal I\mathcal I}r=0.
\]
Multiplying out the boundary row gives
\[
\begin{aligned}
z&=L_{\mathcal B\mathcal B}u_{\mathcal B}+L_{\mathcal B\mathcal I}u_{\mathcal I}^{\min}\\
&=L_{\mathcal B\mathcal B}u_{\mathcal B}-L_{\mathcal B\mathcal I}L_{\mathcal I\mathcal I}^{\dagger}L_{\mathcal I\mathcal B}u_{\mathcal B}+L_{\mathcal B\mathcal I}r\\
&=\left(L_{\mathcal B\mathcal B}-L_{\mathcal B\mathcal I}L_{\mathcal I\mathcal I}^{\dagger}L_{\mathcal I\mathcal B}\right)u_{\mathcal B}
=\Delta_{q,\mathrm{up}}^{\KK,\LL}u_{\mathcal B},
\end{aligned}
\]
where $L_{\mathcal B\mathcal I}r=0$ follows from $r\in\ker L_{\mathcal I\mathcal I}$ and $\ker L_{\mathcal I\mathcal I}\subseteq\ker L_{\mathcal B\mathcal I}$. We conclude the following.

\begin{proposition}
Let $\KK\hookrightarrow\LL$ be an inclusion of simplicial complexes. Then the $q$-th persistent up Laplacian coincides with the Dirichlet-to-Neumann map; that is,
\[
\Lambda=\Delta_{q,\mathrm{up}}^{\KK,\LL}.
\]
\end{proposition}

Dirichlet-to-Neumann maps and their eigenvalue problems have also been studied for graphs; see \cite{MR3722072,MR4445686}. The study of persistent Laplacians on simplicial complexes is thus a natural generalization of that line of work.

\begin{remark}
In the continuous setting, where $\Omega\subset\mathbb{R}^n$ is a bounded domain with sufficiently regular boundary, the continuous Dirichlet-to-Neumann operator \cite{MR3616276} solves
\[
\Delta u_f=0\quad \text{in }\Omega,\qquad u_f=f\quad \text{on }\partial\Omega.
\]The associated eigenvalue problem is the \emph{Steklov eigenvalue problem}, which encodes geometric and topological information about the domain; see \cite{MR3137253} for a historical survey. In physical terms, if $f$ is the voltage on $\partial\Omega$, then $\Lambda f$ is the resulting boundary current flux. This is the mathematical foundation of \emph{Electrical Impedance Tomography} and the \emph{Calderón inverse problem} \cite{MR590275,MR3460047}. 
\end{remark}

\section{A Cheeger inequality for the persistent Laplacian}
\label{sec:cheeger_perslap}

In this section, we prove a Cheeger-type inequality for the persistent up Laplacian associated with an inclusion
\(\KK\hookrightarrow \LL\) of simplicial complexes. The result is essentially an extension of \cite[Proposition 2.1]{jost2024cheeger}. One difference is that \cite{jost2024cheeger} focuses on the first non-trivial eigenvalue, whereas we restrict our attention to the first nonzero eigenvalue. Nevertheless, a Cheeger-type inequality for the persistent Laplacian in that setting can be obtained by essentially the same argument; see \cref{remark.connection.to.jost}.

The key observation is that one can give a well-defined definition of the \(p\)-persistent Laplacian in the persistent setting. We then introduce a persistent Cheeger constant and prove \cref{prop_cheeger_ineq}.

We believe that the introduction of a \(p\)-Laplacian in the persistent setting is of independent interest and will be considered in subsequent work.

\subsection{A persistent \(p\)-Laplacian}\label{sec:pers_p_lap}

Following \cite{jost2024cheeger}, the \(q\)-th up \(p\)-Laplacian on \(\LL\), defined on chains, is
\[
\Delta_{q, \text{up}}^{\LL}
\coloneqq
B_{q+1}^{\LL}\circ \alpha_p\circ
(B_{q+1}^\LL)^T ,
\]
where
\[
\alpha_p:
(t_1,t_2,\ldots)
\mapsto
\begin{cases}
\bigl(|t_1|^{p-2}t_1, |t_2|^{p-2}t_2,\ldots\bigr), & p>1,\\[2mm]
\{(\xi_1,\xi_2,\ldots): \xi_i\in \sgn(t_i)\ \text{for all }i\}, & p=1,
\end{cases}
\]
and \(\sgn(\cdot)\) is defined by
\[
\sgn(t)=
\begin{cases}
\{1\}, & t>0,\\
[-1,1], & t=0,\\
\{-1\}, & t<0.
\end{cases}
\]

This definition is given with respect to the standard bases of \(q\)- and \((q+1)\)-simplices. Since the \(p\)-Laplacian is nonlinear, its action is not invariant under an arbitrary change of basis. Moreover, the persistent Laplacian is closely tied to the choice of a basis of chains, namely the columns of \(Z\). Thus, one might expect that any definition of a persistent \(p\)-Laplacian depends on the computation of \(Z\). We now show that this is not the case.

Let \(S_q^{\KK}\) and \(S_q^{\LL}\) be the ordered sets of \(q\)-simplices in \(\KK\) and \(\LL\), respectively. We assume that the ordering of the \(q\)-simplices
\[
S_q^{\LL}=\{[\sigma_i]\}_{i=1}^{n_q^\LL}
\]
is chosen so that
\[
S_q^{\KK}=\{[\sigma_i]\}_{i=1}^{n_q^\KK}.
\]
When \(n_q^\KK<n_q^\LL\), we can write the matrix \(B_{q+1}^{\LL}\) in block form as
\[
B_{q+1}^{\LL}
=
\begin{bmatrix}
B\\
D
\end{bmatrix},
\]
where the row indices of \(B\) and \(D\) correspond to \(S_q^{\KK}\) and
\(S_q^{\LL}\setminus S_q^{\KK}\), respectively.

When \(\dim(C_{q+1}^{\LL,\KK})>0\), \cite[Lemma 3.4]{pers_lap} gives
\[
B_{q+1}^{\LL,\KK}=BZ.
\]
Therefore, the persistent up Laplacian can be written as
\[
\Delta_{q,\mathrm{up}}^{\KK,\LL}
=
B_{q+1}^{\LL, \KK}(Z^TZ)^{-1}(B_{q+1}^{\LL, \KK})^T=
BZ(Z^T Z)^{-1}Z^T B^T.
\]
Let
\[
P=Z(Z^T Z)^{-1}Z^T.
\]
Then $P$ is a linear transformation $P\colon C_{q+1}^\LL\to C_{q+1}^\LL$ satisfying
\[
P=P^T= P^*\qquad \qquad  \text{and} \qquad \qquad P=P^2.
\]
Hence
\begin{equation}\label{eq:P_pers_up_lap}
\Delta_{q,\mathrm{up}}^{\KK,\LL}
=
BPB^T
=
(BP)(PB^T).
\end{equation}
Moreover, \(P\) is the orthogonal projection from \(C_{q+1}^{\LL}\) onto the subspace
\(C_{q+1}^{\LL,\KK}\). It is therefore independent of the choice of basis for
\(C_{q+1}^{\LL,\KK}\), and we thus have obtained a basis-free definition of the persistent Laplacian. With respect to the standard basis of $q$-simplices we get the following definition.

\begin{definition}\label{def:pers_lap}
Let \(\KK\hookrightarrow \LL\) be an inclusion of simplicial complexes. The
\(q\)-th up persistent \(p\)-Laplacian associated with this inclusion is
defined for \(x\in C_q^{\KK}\) by
\[
\Delta_{q,p,\mathrm{up}}^{\KK,\LL}(x)
=
BP\alpha_p(PB^T x).
\]
\end{definition}

As in the case of the \(2\)-Laplacian, one can use Rayleigh quotients (\cref{thm.rayleigh}) to
define variational eigenvalues. Concretely, set
\begin{equation}\label{eq:var_eigenval}
\lambda_i\bigl(\Delta_{q,p,\mathrm{up}}^{\KK,\LL}\bigr)
=
\inf_{\gamma(S)\geq i}
\sup_{x\in S}
\frac{\|PB^T x\|_p^p}{\|x\|_p^p},
\qquad
i=1,2,\ldots, |S_q^\KK|.
\end{equation}
Here
\[
\gamma(S)=
\begin{cases}
\min \left\{
k\in \mathbb Z^+ :
\text{there exists an odd continuous map }
\varphi:S\to \mathbb S^{k-1}
\right\}, & S\neq \varnothing,\\[2mm]
0, & S=\varnothing,
\end{cases}
\]
denotes the Krasnoselskii genus of a centrally symmetric set
\(S\subset \mathbb R^n\setminus\{0\}\).

For \(p=2\), the variational eigenvalues defined in
\cref{eq:var_eigenval} are precisely the eigenvalues of the persistent up
Laplacian. For \(p\neq 2\), these variational eigenvalues form only a subset
of the eigenvalues satisfying the generalized eigenvalue problem
\[
\Delta_{q,p,\mathrm{up}}^{\KK,\LL}(x)
=
\lambda \alpha_p(x),
\qquad p>1.
\]
(For $p=1$, the eigenvalue problem involves the Clarke subdifferential \cite{clarke1990optimization}; see \cite{jost2024cheeger}.)
Determining whether the total number of such generalized eigenvalues is
finite remains an open problem in graph theory; see, for example,
\cite{deidda2025nonlinear} for a recent overview.

\subsection{A persistent Cheeger constant}\label{sec:JZ_cheeger_cnst}
For the purposes of this paper, we focus on the \(p\)-Laplacian only in the cases \(p=1\) and \(p=2\). Our goal is to define the persistent Cheeger constant in terms of the smallest non-zero variational eigenvalue for \(p=1\). This definition then leads naturally to a Cheeger-type inequality. Moreover, in the case of graphs, the resulting Cheeger constant agrees with the Cheeger constant appearing in \cref{def:graph_asymm_cheeger}.

First, observe that
\[
\ker\bigl(\Delta_{q,p,\mathrm{up}}^{\KK,\LL}\bigr)
=
\ker(PB^T).
\]
Since \(PB^T\) vanishes on \(\ker(PB^T)\), and since
\[
\gamma(\ker(PB^T)\cap S_p)=\dim \ker(PB^T),
\qquad
S_p:=\{x:\|x\|_p=1\},
\]
we have
\[
\lambda_i\bigl(\Delta_{q,p,\mathrm{up}}^{\KK,\LL}\bigr)=0
\qquad
\text{for all } i\leq \dim \ker(PB^T).
\]
In fact, the smallest non-zero variational eigenvalue is
\[
\lambda_{\dim \ker(PB^T)+1}
\bigl(\Delta_{q,p,\mathrm{up}}^{\KK,\LL}\bigr).
\]
This result is well-known; see, for instance \cite[Section 2]{tudisco2022nonlinear}.

The next result shows that, in the case \(p=1\), this first non-zero
variational eigenvalue admits an equivalent formulation in terms of
\(\ell^1\)-minimal representatives. 

\begin{lemma}
\label{lem:p1cheeger_equiv}
Let
\[
i=\dim\ker(PB^T)+1.
\]
Then
\[
\lambda_i\bigl(\Delta_{q,1,\mathrm{up}}^{\KK,\LL}\bigr)
=
\inf_{\gamma(S)\geq i}
\sup_{x\in S}
\frac{\|PB^T x\|_1}{\|x\|_1}
=
\min_{\substack{x\neq 0\\ x\perp^1\ker(PB^T)}}
\frac{\|PB^T x\|_1}{\|x\|_1}.
\]
Here \(x\perp^1 Y\) means that
\[
\|x+y\|_1\geq \|x\|_1
\qquad
\text{for every } y\in Y.
\]
Equivalently, \(x\) is an \(\ell^1\)-minimal representative of the affine
space \(x+Y\).
\end{lemma}

\begin{proof}
Set
\[
T:=PB^T,
\qquad
K:=\ker T.
\]
We apply \cite[Theorem 2.1]{jost2021discrete} with
\[
F(x)=\|Tx\|_1,
\qquad
G(x)=\|x\|_1,
\qquad
\Pi=K.
\]
The functions \(F\) and \(G\) are even, nonnegative, and \(1\)-homogeneous.
Moreover,
\[
F(x+y)=F(x)
\qquad
\text{for every } y\in K,
\]
and
\[
F(x)=0
\quad\Longleftrightarrow\quad
Tx=0
\quad\Longleftrightarrow\quad
x\in K.
\]
Thus the zero space appearing in the min--max theorem is precisely
\(\Pi=K\). Hence \cite[Theorem 2.1]{jost2021discrete} gives
\[
\inf_{\gamma(S)\geq \dim K+1}
\sup_{x\in S}
\frac{\|Tx\|_1}{\|x\|_1}
=
\min_{\substack{x\neq 0\\
\partial\|x\|_1\cap K^\perp\neq\varnothing}}
\frac{\|Tx\|_1}{\|x\|_1}.
\]

It remains to identify the subdifferential condition with the stated
\(\ell^1\)-orthogonality condition. By
\cite[Proposition 2.3]{jost2021discrete}, the condition
\[
\partial\|x\|_1\cap K^\perp\neq\varnothing
\]
is equivalent to \(x\) being an \(\ell^1\)-minimal representative of its
coset modulo \(K\), that is,
\[
\|x+y\|_1\geq \|x\|_1
\qquad
\text{for every } y\in K.
\]
This is precisely the condition \(x\perp^1 K\). Since \(T=PB^T\), the claim
follows.
\end{proof}

This motivates the following definition.

\begin{definition}\label{def:nonzero_cheeger}
Let \(\KK\hookrightarrow \LL\) be an inclusion of simplicial complexes. The
persistent Cheeger constant of this inclusion is
\[
\varphi_q^{\KK,\LL}
=
\min_{\substack{x\neq 0\\ x\perp^1\ker(PB^T)}}
\frac{\|PB^T x\|_1}{\|x\|_1}.
\]
\end{definition}

When \(\KK=\LL\), we have \(P=I\) (the identity matrix), and therefore
\begin{equation}\label{eq:non_zero_cheeger}
\varphi_q^{\LL}
=
\min_{\substack{x\neq 0\\
x\perp^1\ker\left((B_{q+1}^{\LL})^T\right)}}
\frac{\|(B_{q+1}^{\LL})^T x\|_1}{\|x\|_1}.
\end{equation}
In the case of graphs, this expression agrees with the graph Cheeger constant
appearing in \cref{eq:graph_asymm_cheeger}
(see {\cite[Theorem 2.12.]{MR1421568}}).

Below, we study this Cheeger constant for oriented pseudomanifolds
(\cref{sec:cheeger_phi_psdmfld}) and graphs
(\cref{sec.graph.comparisons}) in the setting of the persistent Laplacian.

\subsection{A persistent Cheeger-type inequality}\label{sec:JZ_cheeger_ineq}

We now prove a Cheeger-type inequality for the smallest non-zero variational
eigenvalue of the persistent up Laplacian.

\begin{proposition}\label{prop_cheeger_ineq}
Let \(\KK\hookrightarrow \LL\) be an inclusion of simplicial complexes, and
let \(|S_q^\KK|\) and \(|S_{q+1}^\LL|\) denote the numbers of \(q\)-simplices
in \(\KK\) and \((q+1)\)-simplices in \(\LL\), respectively. Let
\(\lambda_{\min}^+\) denote the smallest non-zero eigenvalue of
the persistent up Laplacian \(\Delta_{q,\mathrm{up}}^{\KK,\LL}\). Then
\begin{equation}\label{ineq:persistent_cheeger}
\frac{\bigl(\varphi_q^{\KK,\LL}\bigr)^2}{|S_{q+1}^{\LL}|}
\leq
\lambda_{\min}^+
\leq
|S_q^\KK|\bigl(\varphi_q^{\KK,\LL}\bigr)^2 .
\end{equation}
\end{proposition}

\begin{proof}
Let \(x\in C_q^\KK\) be nonzero. We use the elementary norm inequalities
\[
1\leq \frac{\|x\|_1^2}{\|x\|_2^2}\leq |S_q^\KK|,
\qquad
1\leq
\frac{\|PB^T x\|_1^2}{\|PB^T x\|_2^2}
\leq |S_{q+1}^{\LL}|,
\]
whenever \(PB^T x\neq 0\). Hence
\[
\frac{1}{|S_{q+1}^{\LL}|}
\frac{\|PB^T x\|_1^2}{\|x\|_1^2}
\leq
\frac{\|PB^T x\|_2^2}{\|x\|_2^2}
\leq
|S_q^\KK|
\frac{\|PB^T x\|_1^2}{\|x\|_1^2}.
\]

Set
\[
i=\dim\ker(PB^T)+1.
\]
By the discussion above, this is precisely the index of the smallest non-zero variational eigenvalue.By \cref{lem:p1cheeger_equiv} and \cref{def:nonzero_cheeger},
\[
\varphi_q^{\KK,\LL}
=
\lambda_i\bigl(\Delta_{q,1,\mathrm{up}}^{\KK,\LL}\bigr)
=
\inf_{\gamma(S)\geq i}
\sup_{x\in S}
\frac{\|PB^T x\|_1}{\|x\|_1}.
\]
On the other hand,
\[
\lambda_i\bigl(\Delta_{q,\mathrm{up}}^{\KK,\LL}\bigr)
=
\inf_{\gamma(S)\geq i}
\sup_{x\in S}
\frac{\|PB^T x\|_2^2}{\|x\|_2^2}.
\]
Applying the preceding norm comparison to the Rayleigh quotients gives
\[
\frac{1}{|S_{q+1}^{\LL}|}
\left(
\inf_{\gamma(S)\geq i}
\sup_{x\in S}
\frac{\|PB^T x\|_1}{\|x\|_1}
\right)^2
\leq
\lambda_i\bigl(\Delta_{q,\mathrm{up}}^{\KK,\LL}\bigr)
\leq
|S_q^\KK|
\left(
\inf_{\gamma(S)\geq i}
\sup_{x\in S}
\frac{\|PB^T x\|_1}{\|x\|_1}
\right)^2.
\]
Using the identity above with \(\varphi_q^{\KK,\LL}\), we obtain
\[
\frac{\bigl(\varphi_q^{\KK,\LL}\bigr)^2}{|S_{q+1}^{\LL}|}
\leq
\lambda_i\bigl(\Delta_{q,\mathrm{up}}^{\KK,\LL}\bigr)
\leq
|S_q^\KK|\bigl(\varphi_q^{\KK,\LL}\bigr)^2,
\]
as claimed.
\end{proof}

When \(\KK=\LL\), we obtain the following corollary.

\begin{corollary}\label{cor:hodge_cheeger_ineq}
Let \(\lambda_{\min}^+\) be the smallest non-zero variational eigenvalue of
the up Laplacian \(\Delta_{q,\mathrm{up}}^{\LL}\). Then
\begin{equation}\label{ineq:hodge_cheeger}
\frac{\bigl(\varphi_q^{\LL}\bigr)^2}{|S_{q+1}^{\LL}|}
\leq
\lambda_{\min}^+
\leq
|S_q^\LL|\bigl(\varphi_q^{\LL}\bigr)^2 .
\end{equation}
\end{corollary}

\begin{remark}
\label{remark.connection.to.jost}
Following the definition of Cheeger constant in \cite{jost2024cheeger},
we obtain the following bound of the smallest \emph{non-trivial} persistent Cheeger constant:
\[
\psi_{q}^{\KK, \LL}
=
\min_{0\neq x\perp^1\im\big((B_q^\KK)^T\big)}\frac{\|PB^T x\|_1}{\|x\|_1}.
\]
Let 
\[\lambda_{+}^{\KK, \LL}
:=
\min_{0 \neq x\in \ker(B_q^\KK)}\frac{\langle x, \Delta_{q, \text{up}}^{\KK, \LL}x\rangle}{\|x\|^2}
\]
denote the smallest \emph{non-trivial} eigenvalue of $\Delta_{q, \text{up}}^{\KK, \LL}$,
using exactly the same argument as 
the proof of \cref{prop_cheeger_ineq},
we obtain 
\[
\frac{\bigl(\psi_q^{\KK,\LL}\bigr)^2}{|S_{q+1}^{\LL}|}
\leq
\lambda_{+}^{\KK, \LL}
\leq
|S_q^\KK|\bigl(\psi_q^{\KK,\LL}\bigr)^2 .
\]
See also \cref{rem.pseudo.jz}
\end{remark}

\section{Locally complete skeleta}\label{sec:loc_cmplt_skl}
One of the first instances of a Cheeger inequality for simplicial complexes was given in \cite{MR3516884} under the assumption that the simplicial complex has a complete $q$-skeleton. Concretely, they considered $\Delta^{\LL}_{q,\mathrm{up}}$ under the assumption that all possible $q$-dimensional simplices were included in $\LL$. For $q=0$, this assumption is void and applies to any graph. For $q\geq 1$, however, this is a fairly strong assumption.

In this section, we extend this result to inclusions $\KK\hookrightarrow \LL$ of simplicial complexes, under the assumption that $\KK$ every connected component of $\KK$ has a complete $q$-skeleton. This allows one to, e.g., consider $\KK$ to be a collection of triangles ($1$-cycles), and then the inclusion into $\LL$ measures how these cycles are connected in the ambient space. 
\begin{remark}
    Subsequent to \cite{MR3516884}, Gundert and Szedl\'ak \cite{MR3382297} extended this framework to general simplicial complexes without the requirement of a complete skeleton. We conjecture that their result naturally extends to the persistent setting; however, due to the technical nature of such an extension, we limit our discussion to the simpler setting.
\end{remark}

\subsection{The Cheeger constant}

Let $\KK \subseteq \LL$ be locally $q$-complete. That is,
\[
\KK = \coprod_{\alpha=1}^N \KK^\alpha,
\]
where each $\KK^\alpha$ is a connected component of $\KK$ with complete
$q$-skeleton. We denote the vertex sets of $\LL$ and $\KK^\alpha$ by
$V^\LL$ and $V^{\KK^\alpha}$, respectively.

We will consider labelings of the vertices of $\LL$ by the labels
$0,\ldots,q+1$. For a subcomplex $\LL'\subseteq \LL$, let
$V_i^{\LL'}$ denote the set of vertices of $\LL'$ with label $i$. The following condition ensures that each connected component of $\KK$
contains every label.

\begin{definition}\label{def:effective}
A partition
\[
\mathcal{P} = \coprod_{i=0}^{q+1} V_i^\LL
\]
of $V^\LL$ is called \emph{effective} if
\[
V_i^\LL \cap V^{\KK^\alpha} \neq \varnothing
\]
for every $0 \leq i \leq q+1$ and every $1 \leq \alpha \leq N$.
\end{definition}

In particular, \cref{def:effective} implies that each component
$\KK^\alpha$ has at least $q+2$ vertices.

\begin{example}\label{ex:complete_skeleton}
We illustrate the setup in \cref{fig:disk}. In this example, $\LL$ is a
$2$-dimensional simplicial complex, and
\[
\KK = \KK^1 \coprod \KK^2
\]
has two connected components, each with complete $1$-skeleton. These components are colored red.

We partition the vertices of $\LL$ into three color classes, colored
blue, yellow, and green. This partition is effective because each
component $\KK^1$ and $\KK^2$ contains vertices of all three colors, as
shown in \cref{fig:disk}.

\begin{figure}[H]
    \centering

\begin{tikzpicture}[scale=2.5]
  \colorlet{fillcolor}{blue!10}
  \colorlet{edgecolor}{blue!60!black}
  \fill[fillcolor] (0.3380,-0.0645) -- (0.8413,0.5406) -- (0.3585,0.6075) -- cycle;
  \fill[fillcolor] (-0.6549,-0.7557) -- (-0.1423,-0.9898) -- (-0.2552,-0.0452) -- cycle;
  \fill[fillcolor] (0.0468,0.5008) -- (0.3380,-0.0645) -- (0.3585,0.6075) -- cycle;
  \fill[fillcolor] (0.0468,0.5008) -- (0.3585,0.6075) -- (-0.1423,0.9898) -- cycle;
  \fill[fillcolor] (-0.9595,0.2817) -- (-0.5394,0.0970) -- (-0.6549,0.7557) -- cycle;
  \fill[fillcolor] (-0.6549,-0.7557) -- (-0.5732,-0.0958) -- (-0.9595,-0.2817) -- cycle;
  \fill[fillcolor] (-0.5732,-0.0958) -- (-0.6549,-0.7557) -- (-0.2552,-0.0452) -- cycle;
  \fill[fillcolor] (0.3585,0.6075) -- (0.4154,0.9096) -- (-0.1423,0.9898) -- cycle;
  \fill[fillcolor] (0.8413,0.5406) -- (0.4154,0.9096) -- (0.3585,0.6075) -- cycle;
  \fill[fillcolor] (0.5373,-0.3720) -- (0.8413,-0.5406) -- (1.0000,0.0000) -- cycle;
  \fill[fillcolor] (0.3380,-0.0645) -- (0.5373,-0.3720) -- (1.0000,0.0000) -- cycle;
  \fill[fillcolor] (0.8413,-0.5406) -- (0.5373,-0.3720) -- (0.4154,-0.9096) -- cycle;
  \fill[fillcolor] (0.5373,-0.3720) -- (0.3380,-0.0645) -- (-0.2552,-0.0452) -- cycle;
  \fill[fillcolor] (-0.1423,-0.9898) -- (0.5373,-0.3720) -- (-0.2552,-0.0452) -- cycle;
  \fill[fillcolor] (-0.6549,0.7557) -- (-0.2682,0.6758) -- (-0.1423,0.9898) -- cycle;
  \fill[fillcolor] (-0.2682,0.6758) -- (0.0468,0.5008) -- (-0.1423,0.9898) -- cycle;
  \fill[fillcolor] (-0.5394,0.0970) -- (-0.4631,-0.0002) -- (-0.2552,-0.0452) -- cycle;
  \fill[fillcolor] (-0.4631,-0.0002) -- (-0.5732,-0.0958) -- (-0.2552,-0.0452) -- cycle;
  \fill[fillcolor] (-0.9595,0.2817) -- (-0.5903,-0.0185) -- (-0.5394,0.0970) -- cycle;
  \fill[fillcolor] (-0.5903,-0.0185) -- (-0.4631,-0.0002) -- (-0.5394,0.0970) -- cycle;
  \fill[fillcolor] (-0.4631,-0.0002) -- (-0.5903,-0.0185) -- (-0.5732,-0.0958) -- cycle;
  \fill[fillcolor] (-0.5903,-0.0185) -- (-0.9595,0.2817) -- (-0.9595,-0.2817) -- cycle;
  \fill[fillcolor] (-0.5732,-0.0958) -- (-0.5903,-0.0185) -- (-0.9595,-0.2817) -- cycle;
  \fill[red!30, opacity=0.5] (-0.2682,0.6758) -- (-0.1232,0.4469) -- (0.0468,0.5008) -- cycle;
  \fill[red!30, opacity=0.5] (-0.1232,0.4469) -- (-0.5394,0.0970) -- (-0.2552,-0.0452) -- cycle;
  \fill[fillcolor] (0.3380,-0.0645) -- (-0.1232,0.4469) -- (-0.2552,-0.0452) -- cycle;
  \fill[fillcolor] (0.0468,0.5008) -- (-0.1232,0.4469) -- (0.3380,-0.0645) -- cycle;
  \fill[fillcolor] (0.8413,0.5406) -- (0.3380,-0.0645) -- (1.0000,0.0000) -- cycle;
  \fill[fillcolor] (0.5373,-0.3720) -- (-0.1423,-0.9898) -- (0.4154,-0.9096) -- cycle;
  \fill[red!30, opacity=0.5] (-0.5394,0.0970) -- (-0.2682,0.6758) -- (-0.6549,0.7557) -- cycle;
  \fill[fillcolor] (-0.1232,0.4469) -- (-0.2682,0.6758) -- (-0.5394,0.0970) -- cycle;
  \fill[red!30, opacity=0.5] (-0.2552,-0.0452) -- (-0.1423,-0.9898) -- (0.5373,-0.3720) -- cycle;
  \fill[red!30, opacity=0.5] (0.8413,-0.5406) -- (0.5373,-0.3720) -- (0.4154,-0.9096) -- cycle;
  \draw[edgecolor, thick] (0.8413,0.5406) -- (0.3380,-0.0645) -- (1.0000,0.0000) -- cycle;
  \draw[edgecolor, thick] (0.3380,-0.0645) -- (0.8413,0.5406) -- (0.3585,0.6075) -- cycle;
  \draw[edgecolor, thick] (-0.6549,-0.7557) -- (-0.1423,-0.9898) -- (-0.2552,-0.0452) -- cycle;
  \draw[edgecolor, thick] (0.0468,0.5008) -- (0.3380,-0.0645) -- (0.3585,0.6075) -- cycle;
  \draw[edgecolor, thick] (0.0468,0.5008) -- (0.3585,0.6075) -- (-0.1423,0.9898) -- cycle;
  \draw[edgecolor, thick] (-0.9595,0.2817) -- (-0.5394,0.0970) -- (-0.6549,0.7557) -- cycle;
  \draw[edgecolor, thick] (-0.6549,-0.7557) -- (-0.5732,-0.0958) -- (-0.9595,-0.2817) -- cycle;
  \draw[edgecolor, thick] (-0.5732,-0.0958) -- (-0.6549,-0.7557) -- (-0.2552,-0.0452) -- cycle;
  \draw[edgecolor, thick] (0.3585,0.6075) -- (0.4154,0.9096) -- (-0.1423,0.9898) -- cycle;
  \draw[edgecolor, thick] (0.8413,0.5406) -- (0.4154,0.9096) -- (0.3585,0.6075) -- cycle;
  \draw[edgecolor, thick] (0.5373,-0.3720) -- (0.8413,-0.5406) -- (1.0000,0.0000) -- cycle;
  \draw[edgecolor, thick] (0.3380,-0.0645) -- (0.5373,-0.3720) -- (1.0000,0.0000) -- cycle;
  \draw[edgecolor, thick] (0.8413,-0.5406) -- (0.5373,-0.3720) -- (0.4154,-0.9096) -- cycle;
  \draw[edgecolor, thick] (0.5373,-0.3720) -- (-0.1423,-0.9898) -- (0.4154,-0.9096) -- cycle;
  \draw[edgecolor, thick] (0.5373,-0.3720) -- (0.3380,-0.0645) -- (-0.2552,-0.0452) -- cycle;
  \draw[edgecolor, thick] (-0.1423,-0.9898) -- (0.5373,-0.3720) -- (-0.2552,-0.0452) -- cycle;
  \draw[edgecolor, thick] (-0.6549,0.7557) -- (-0.2682,0.6758) -- (-0.1423,0.9898) -- cycle;
  \draw[edgecolor, thick] (-0.2682,0.6758) -- (0.0468,0.5008) -- (-0.1423,0.9898) -- cycle;
  \draw[edgecolor, thick] (-0.5394,0.0970) -- (-0.2682,0.6758) -- (-0.6549,0.7557) -- cycle;
  \draw[edgecolor, thick] (-0.5394,0.0970) -- (-0.4631,-0.0002) -- (-0.2552,-0.0452) -- cycle;
  \draw[edgecolor, thick] (-0.4631,-0.0002) -- (-0.5732,-0.0958) -- (-0.2552,-0.0452) -- cycle;
  \draw[edgecolor, thick] (-0.9595,0.2817) -- (-0.5903,-0.0185) -- (-0.5394,0.0970) -- cycle;
  \draw[edgecolor, thick] (-0.5903,-0.0185) -- (-0.4631,-0.0002) -- (-0.5394,0.0970) -- cycle;
  \draw[edgecolor, thick] (-0.4631,-0.0002) -- (-0.5903,-0.0185) -- (-0.5732,-0.0958) -- cycle;
  \draw[edgecolor, thick] (-0.5903,-0.0185) -- (-0.9595,0.2817) -- (-0.9595,-0.2817) -- cycle;
  \draw[edgecolor, thick] (-0.5732,-0.0958) -- (-0.5903,-0.0185) -- (-0.9595,-0.2817) -- cycle;
  \draw[edgecolor, thick] (-0.1232,0.4469) -- (-0.2682,0.6758) -- (-0.5394,0.0970) -- cycle;
  \draw[edgecolor, thick] (-0.2682,0.6758) -- (-0.1232,0.4469) -- (0.0468,0.5008) -- cycle;
  \draw[edgecolor, thick] (-0.1232,0.4469) -- (-0.5394,0.0970) -- (-0.2552,-0.0452) -- cycle;
  \draw[edgecolor, thick] (0.3380,-0.0645) -- (-0.1232,0.4469) -- (-0.2552,-0.0452) -- cycle;
  \draw[edgecolor, thick] (0.0468,0.5008) -- (-0.1232,0.4469) -- (0.3380,-0.0645) -- cycle;

  \fill[green!30, opacity=0.5] (-0.1423,-0.9898) -- (0.4154,-0.9096) -- (0.5373,-0.3720) -- cycle;
  \fill[green!30, opacity=0.5] (-0.5394,0.0970) -- (-0.2682,0.6758) -- (-0.1232,0.4469) -- cycle;

  \draw[red, very thick, line cap=round] (-0.1423,-0.9898) -- (0.4154,-0.9096);
  \draw[red, very thick, line cap=round] (0.4154,-0.9096) -- (0.5373,-0.3720);
  \draw[red, very thick, line cap=round] (0.5373,-0.3720) -- (-0.1423,-0.9898);
  \draw[red, very thick, line cap=round] (-0.5394,0.0970) -- (-0.2682,0.6758);
  \draw[red, very thick, line cap=round] (-0.2682,0.6758) -- (-0.1232,0.4469);
  \draw[red, very thick, line cap=round] (-0.1232,0.4469) -- (-0.5394,0.0970);
  \node[red, font=\small] at (0.25,-0.8) {$\mathcal{K}^1$};
  \node[red, font=\small] at (-0.277,0.4119) {$\mathcal{K}^2$};
  \fill[yellow!90!orange] (1.0000,0.0000) circle (0.7pt);
  \draw[orange!80!black, line width=0.4pt] (1.0000,0.0000) circle (0.7pt);
  \fill[cyan!60!blue] (0.8413,0.5406) circle (0.7pt);
  \draw[blue!80!black, line width=0.4pt] (0.8413,0.5406) circle (0.7pt);
  \fill[yellow!90!orange] (0.4154,0.9096) circle (0.7pt);
  \draw[orange!80!black, line width=0.4pt] (0.4154,0.9096) circle (0.7pt);
  \fill[green!70!black] (-0.1423,0.9898) circle (0.7pt);
  \draw[green!40!black, line width=0.4pt] (-0.1423,0.9898) circle (0.7pt);
  \fill[yellow!90!orange] (-0.6549,0.7557) circle (0.7pt);
  \draw[orange!80!black, line width=0.4pt] (-0.6549,0.7557) circle (0.7pt);
  \fill[yellow!90!orange] (-0.9595,0.2817) circle (0.7pt);
  \draw[orange!80!black, line width=0.4pt] (-0.9595,0.2817) circle (0.7pt);
  \fill[yellow!90!orange] (-0.9595,-0.2817) circle (0.7pt);
  \draw[orange!80!black, line width=0.4pt] (-0.9595,-0.2817) circle (0.7pt);
  \fill[cyan!60!blue] (-0.6549,-0.7557) circle (0.7pt);
  \draw[blue!80!black, line width=0.4pt] (-0.6549,-0.7557) circle (0.7pt);
  \fill[green!70!black] (-0.1423,-0.9898) circle (0.7pt);
  \draw[green!40!black, line width=0.4pt] (-0.1423,-0.9898) circle (0.7pt);
  \fill[yellow!90!orange] (0.4154,-0.9096) circle (0.7pt);
  \draw[orange!80!black, line width=0.4pt] (0.4154,-0.9096) circle (0.7pt);
  \fill[green!70!black] (0.8413,-0.5406) circle (0.7pt);
  \draw[green!40!black, line width=0.4pt] (0.8413,-0.5406) circle (0.7pt);
  \fill[green!70!black] (-0.1232,0.4469) circle (0.7pt);
  \draw[green!40!black, line width=0.4pt] (-0.1232,0.4469) circle (0.7pt);
  \fill[cyan!60!blue] (-0.4631,-0.0002) circle (0.7pt);
  \draw[blue!80!black, line width=0.4pt] (-0.4631,-0.0002) circle (0.7pt);
  \fill[yellow!90!orange] (0.3585,0.6075) circle (0.7pt);
  \draw[orange!80!black, line width=0.4pt] (0.3585,0.6075) circle (0.7pt);
  \fill[green!70!black] (-0.5732,-0.0958) circle (0.7pt);
  \draw[green!40!black, line width=0.4pt] (-0.5732,-0.0958) circle (0.7pt);
  \fill[yellow!90!orange] (-0.5394,0.0970) circle (0.7pt);
  \draw[orange!80!black, line width=0.4pt] (-0.5394,0.0970) circle (0.7pt);
  \fill[green!70!black] (0.0468,0.5008) circle (0.7pt);
  \draw[green!40!black, line width=0.4pt] (0.0468,0.5008) circle (0.7pt);
  \fill[cyan!60!blue] (0.3380,-0.0645) circle (0.7pt);
  \draw[blue!80!black, line width=0.4pt] (0.3380,-0.0645) circle (0.7pt);
  \fill[yellow!90!orange] (-0.5903,-0.0185) circle (0.7pt);
  \draw[orange!80!black, line width=0.4pt] (-0.5903,-0.0185) circle (0.7pt);
  \fill[yellow!90!orange] (-0.2552,-0.0452) circle (0.7pt);
  \draw[orange!80!black, line width=0.4pt] (-0.2552,-0.0452) circle (0.7pt);
  \fill[cyan!60!blue] (-0.2682,0.6758) circle (0.7pt);
  \draw[blue!80!black, line width=0.4pt] (-0.2682,0.6758) circle (0.7pt);
  \fill[cyan!60!blue] (0.5373,-0.3720) circle (0.7pt);
  \draw[blue!80!black, line width=0.4pt] (0.5373,-0.3720) circle (0.7pt);
\end{tikzpicture}

\caption{Example of locally $q$-complete $\KK\subset \LL$, with $\KK=\KK^1\coprod \KK^2$.}
\label{fig:disk}
\end{figure}
\end{example}
Before defining the Cheeger constant, we introduce some additional notation.
Since each component $\KK^\alpha$ has a complete $q$-skeleton, any set of
$q+1$ vertices in $V^{\KK^\alpha}$ spans a $q$-simplex of $\KK^\alpha$.
Thus, if a $(q+1)$-simplex $\tau\in S_{q+1}^{\LL}$ has exactly
$q+1$ vertices in $V^{\KK^\alpha}$, then these vertices span a unique
$q$-face of $\tau$ contained in $\KK^\alpha$.

We define the \emph{$(q+1)$-frontier of $\KK^\alpha$}, denoted
$\Ft(\KK^\alpha)$, to be the set of $(q+1)$-simplices of $\LL$ that have
exactly one vertex outside $V^{\KK^\alpha}$ and whose remaining $q+1$
vertices span a $q$-simplex in $\KK^\alpha$. Equivalently,
\[
\Ft(\KK^\alpha)
=
\left\{
\tau \in S_{q+1}^{\LL}
:
\begin{array}{l}
\exists v\in \tau \text{ such that }
\tau\setminus\{v\}\in S_q^{\KK^\alpha},\\
\text{and } v\notin V^{\KK^\alpha}
\end{array}
\right\}.
\]
These are the pink $2$-simplices in \cref{fig:disk}.

We then define the set of \emph{interior} $(q+1)$-simplices of $\LL$ by
\[
\LL^\circ
:=
S_{q+1}^{\LL}
\setminus
\left(
S_{q+1}^{\KK}
\cup
\bigcup_{\alpha=1}^N \Ft(\KK^\alpha)
\right).
\]
These correspond to the purple $2$-simplices in \cref{fig:disk}.

Let $\mathcal P$ be an effective partition. If
$\mathcal X\subseteq \LL$ is a subcomplex, we write
$S_{q+1}^{\mathcal X}$ for its set of $(q+1)$-simplices. If $\mathcal A\subseteq S_{q+1}^{\LL}$ is a set of $(q+1)$-simplices,
we define
\[
F_{\mathcal P}(\mathcal A)
=
\left\{
\tau\in \mathcal A :
\tau \text{ has } q+2 \text{ distinct labels with respect to }
\mathcal P
\right\}.
\]
For a subcomplex $\mathcal X\subseteq \LL$, we use the shorthand
\[
F_{\mathcal P}(\mathcal X):=F_{\mathcal P}(S_{q+1}^{\mathcal X}).
\]

Next, for each $0 \leq i \leq q+1$, let $G_{\mathcal{P}}(\Ft(\KK^\alpha),i)$ be the set of $(q+1)$-simplices in $\Ft(\KK^\alpha)$ whose vertices have $q+2$ distinct labels and whose unique vertex outside $\KK^\alpha$ has label $i$; that is,
\[
G_{\mathcal{P}}(\Ft(\KK^\alpha),i)
=
\left\{
\tau \in \Ft(\KK^\alpha)
:
\begin{array}{l}
\exists v \in \tau \text{ such that } 
v \in V_i^\LL \setminus V_i^{\KK^\alpha}, \\
\tau \setminus \{v\} \in S_q^{\KK^\alpha}
\text{ has } q+1 \text{ distinct labels, all different from } i
\end{array}
\right\}.
\]

We also let $H_{\mathcal{P}}(\Ft(\KK^\alpha),i)$ be the set of $(q+1)$-simplices in $\Ft(\KK^\alpha)$ whose $q$-face in $\KK^\alpha$ has $q+1$ distinct labels, all different from $i$, and whose unique vertex outside $\KK^\alpha$ has the same label as one of the vertices in this $q$-face; equivalently,
\[
H_{\mathcal{P}}(\Ft(\KK^\alpha),i)
=
\left\{
\tau \in \Ft(\KK^\alpha)
:
\begin{array}{l}
\exists v \in \tau \text{ such that }
\tau \setminus \{v\} \in S_q^{\KK^\alpha}, \\
\tau \setminus \{v\}
\text{ has } q+1 \text{ distinct labels, all different from } i, \\
v \in V^\LL \setminus V_i^\LL
\end{array}
\right\}.
\]

\begin{definition}\label{def:nt_cheeger}
Let $\KK\hookrightarrow \LL$ be an inclusion of simplicial complexes, and assume that $\KK$ is locally $q$-complete with decomposition
\[
\KK = \coprod_{\alpha=1}^N \KK^\alpha.
\]
Let
\[
m_i = \min_{1 \leq \alpha \leq N} |V_i^{\KK^\alpha}|,
\qquad
m = \sum_{i=0}^{q+1} m_i,
\]
and define
\[
\epsilon_i^\alpha = |V_i^{\KK^\alpha}| - m_i,
\qquad
s_i^\alpha = m + \epsilon_i^\alpha .
\]
For an effective partition $\mathcal{P}$, define the \emph{interior, boundary, and frontier contributions} by
\begin{align*}
I_{\mathcal{P}}
&=
m^2 |F_{\mathcal{P}}(\LL^\circ)|, \\
B_{\mathcal{P}}
&=
\sum_{\alpha=1}^N
|V^{\KK^\alpha}|^2
|F_{\mathcal{P}}(\KK^\alpha)|, \\
R_{\mathcal{P}}
&=
\sum_{i=0}^{q+1}
\sum_{\alpha=1}^N
\left(
(s_i^\alpha)^2
|G_{\mathcal{P}}(\Ft(\KK^\alpha),i)|
+
(\epsilon_i^\alpha)^2
|H_{\mathcal{P}}(\Ft(\KK^\alpha),i)|
\right).
\end{align*}
The \emph{persistent Cheeger constant} is defined by
\[
h_q^{\KK,\LL}
=
\min_{\mathcal{P}}
\frac{
I_{\mathcal{P}} + B_{\mathcal{P}} + R_{\mathcal{P}}
}{
\displaystyle
\sum_{\alpha=1}^N
|V^{\KK^\alpha}|
\prod_{i=0}^{q+1}
|V_i^{\KK^\alpha}|
},
\]
where the minimum is taken over all effective partitions $\mathcal{P}$.
\end{definition}

When $\KK$ contains only one component ($N=1$), then $\epsilon_i^1=0$ and $s_i^1=m=|V^\KK|$ for each $i\leq q+1$. Hence, the Cheeger constant reduces to
\[
h_{q}^{\KK, \LL}
=
\min_{\mathcal{P}}
\frac{
|V^\KK|^2\cdot \big|F_{\mathcal{P}}(\LL^\circ)\cup F_{\mathcal{P}}(\KK) \cup \big(\bigcup_{i=0}^{q+1}G_{\mathcal{P}}(\Ft(\KK), i)\big)\big|
}{|V^\KK|\cdot \prod_{i=0}^{q+1}|V^\KK_i|}.
\]
Observe that in the numerator the set $F_{\mathcal{P}}(\LL^\circ)\cup F_{\mathcal{P}}(\KK) \cup \big(\bigcup_{i=0}^{q+1}G_{\mathcal{P}}(\Ft(\KK), i)\big)=F_{\mathcal{P}}(\LL)$,
hence we obtain the following persistent Cheeger constant:
\begin{equation}\label{eq:nt_cheeger_inq_graph}
h_{q}^{\KK, \LL}=
\min_{\mathcal{P}}
\frac{|V^\KK|\cdot|F_{\mathcal{P}}(\LL)|}{\prod_{i=0}^{q+1}|V^\KK_i|}.
\end{equation}

If $\KK=\LL$, this recovers the Cheeger constant in \cite[Theorem 1.2]{MR3516884}.

Furthermore, note that this Cheeger constant can be zero; e.g., if $\KK=\LL$ is the $1$-skeleton of a tetrahedron with one added $2$-simplex. One can then partition the vertices so that the labels of the vertices of the $2$-simplex are not distinct.
 
\subsection{A (one-sided) Cheeger inequality}\label{sec:cmplt_cheeger_ineq}
In this section, we let $\lambda_{+}^{\KK, \LL}$ be the smallest non-trivial eigenvalue of $\Delta_{q, \text{up}}^{\KK, \LL}$,
i.e.,
\[
\lambda_{+}^{\KK, \LL}
:=
\min_{0 \neq x\in \ker(\partial_q^\KK)}\frac{\langle x, \Delta_{q, \text{up}}^{\KK, \LL}x\rangle}{\|x\|^2}.
\]
The additional constraint that $x\in \ker(\partial_q^\KK)$ allows us to deduce precisely when this eigenvalue is non-zero; see also \cref{remark.connection.to.jost}.

\begin{lemma}
Let $\KK \subseteq \LL$. Then $\lambda_{+}^{\KK,\LL}=0$ if and only if the map
\[
H_q(\KK;\R) \longrightarrow H_q(\LL;\R)
\]
induced by the inclusion is nonzero.
\end{lemma}

\begin{proof}
By the Hodge decomposition theorem, we have the orthogonal decomposition
\[
\ker(\Delta_{q,\mathrm{up}}^{\KK,\LL})
=
\ker\bigl((\partial_{q+1}^{\KK,\LL})^*\bigr)
=
\im\bigl((\partial_q^\KK)^*\bigr)
\oplus
\ker(\Delta_q^{\KK,\LL}).
\]
Since
\[
\im\bigl((\partial_q^\KK)^*\bigr)
=
\ker(\partial_q^\KK)^\perp,
\]
it follows that
\[
\ker(\Delta_{q,\mathrm{up}}^{\KK,\LL})
=
\ker(\partial_q^\KK)^\perp
\oplus
\ker(\Delta_q^{\KK,\LL}).
\]
Moreover, $\dim \ker(\Delta_q^{\KK,\LL})$ is equal to the rank of the map
\[
H_q(\KK;\R) \longrightarrow H_q(\LL;\R).
\]
Therefore, there exists a nonzero element
$x \in \ker(\partial_q^\KK)$ such that
\[
\Delta_{q,\mathrm{up}}^{\KK,\LL} x = 0
\]
if and only if
\[
\ker(\Delta_q^{\KK,\LL}) \neq 0.
\]
Hence, $\lambda_{\min}^{\KK,\LL}=0$ if and only if the induced map is nonzero.
\end{proof}

We now prove a one-sided persistent Cheeger-type inequality. In \cite[Section 4.2]{MR3516884}, the authors give an example using the M\"obius strip showing that a lower Cheeger inequality is not possible in general.

\begin{theorem}\label{prop:nt_cheeger_ineq_cmplx}
Let $\KK\hookrightarrow \LL$ be an inclusion of simplicial complexes, and assume that $\KK$ is locally $q$-complete with decomposition
\[
\KK = \coprod_{\alpha=1}^N \KK^\alpha.
\]
We have the following inequality:
\begin{equation}\label{eq:cheeger_ineq_cmplt_skl}
\lambda_{+}^{\KK, \LL}\leq h_{q}^{\KK, \LL}.
\end{equation}
\end{theorem}

\begin{proof}
Let $\mathcal{P}=\coprod_{i=0}^{q+1}V^\LL_i$ be an effective partition that realizes the minimum in $h_{q}^{\KK, \LL}$. For simplicity, we write $\mathrm{label}(v) = j$ if $v\in V_j^\LL$.  

We put a total order on the vertices of $\LL$ such that $v<w$ if $v\in V_i^\LL$ and $w\in V_j^\LL$ for $i<j$. Then, we orient the higher-order simplices according to this ordering. In particular, for a $q$-simplex $\sigma$ in $\LL$, we write $\sigma = [v_0v_1\cdots v_q]$ where $v_r < v_s$ for $r < s$.

For each such simplex, we define
\[
f([v_0v_1\cdots v_q])=
\begin{cases}
    (-1)^i|V^{\KK^\alpha}_{i}| & \text{if } \sigma \in \KK^\alpha \text{ and } \bigcup_{j=0}^q \mathrm{label}(v_j) = \{0, \ldots, q+1\}\setminus \{i\};\\
    (-1)^im_{i} & \text{if } \sigma \notin \KK \text{ and } \bigcup_{j=0}^q \mathrm{label}(v_j) = \{0, \ldots, q+1\}\setminus \{i\};\\
    0 & \text{otherwise},
\end{cases}
\]
and define the associated $q$-chains $f_\LL:=\sum_{\sigma\in S_q^\LL}f(\sigma)\sigma\in C_q^\LL$ and 
$f_\KK:=\sum_{\sigma\in S_q^\KK}f(\sigma)\sigma\in C_q^\KK$, respectively.

The $q$-chain $f_\KK$ will serve as our test chain to give an upper bound on the eigenvalue. To that end, we first need to show that $f_\KK\in \ker(\partial_q^\KK)$. This is immediate by noting that restricted to each $\KK^\alpha$, $f_\KK$ coincides with the $q$-chain (or $(q-1)$-chain in their language) constructed for a simplicial complex with a complete skeleton in \cite{MR3516884}.  

Furthermore, following the construction in the same paper, it is straightforward to verify that 
\[\sum_{\sigma\in S_q^{\KK^\alpha}} f_{\KK}(\sigma)^2 = |V^{\KK^\alpha}|\cdot \prod_{i=0}^{q+1}|V^{\KK^\alpha}_i|.\]
To see this: a $q$-simplex $\sigma$ contributes to the sum $\sum_{\sigma\in S_{q}^{\KK^\alpha}}f_\KK(\sigma)^2$ only if the vertices of $\sigma$ have $q+1$ different labels. That is, there is a label $i$ such that $V_i^{\KK^\alpha}\cap \sigma=\varnothing$, and thus $\sigma$ contributes $|V_i^{\KK^\alpha}|^2$ to the sum. Since $\KK^\alpha$ is $q$-complete, there are $\prod_{j \neq i} |V_j^{\KK^\alpha}|$ such $q$-simplices whose vertices have different labels and do not intersect with $V_i^{\KK^\alpha}$. Therefore,
\[
\sum_{\sigma\in S_{q}^{\KK^\alpha}}f_\KK(\sigma)^2=
\sum_{i=0}^{q+1}\left(\prod_{j\neq i}|V_j^{\KK^\alpha}|\right) |V_i^{\KK^\alpha}|^2=
|V^{\KK^\alpha}|\cdot \prod_{i=0}^{q+1}|V^{\KK^\alpha}_i|.
\]

Summing over each connected component of $\KK$, we obtain
\[
\| f_\KK\|^2=
\sum_{\sigma\in S_q^\KK}f_\KK(\sigma)^2=
\sum_{\alpha=1}^N \sum_{\sigma\in S_q^{\KK^\alpha}} f_\KK(\sigma)^2 =
\sum_{\alpha=1}^N \left(|V^{\KK^\alpha}|\cdot \prod_{i=0}^{q+1}|V^{\KK^\alpha}_i|\right).
\]

According to \cref{lemma:schur_optimization},
$\langle f_\KK, \Delta_{q, \text{up}}^{\KK, \LL}\, f_\KK \rangle \leq \langle f_\LL, \Delta_{q, \text{up}}^{\LL}\, f_\LL \rangle$,
hence
\[
\lambda_{+}^{\KK, \LL}
=
\min_{0 \neq x\in \ker(\partial_q^\KK)}\frac{\langle x, \Delta_{q, \text{up}}^{\KK, \LL}x\rangle}{\|x\|^2}\leq
\frac{\langle f_\KK, \Delta_{q, \text{up}}^{\KK, \LL}\, f_\KK \rangle}{\|f_\KK\|^2}\leq
\frac{\langle f_\LL, \Delta_{q, \text{up}}^{\LL}\, f_\LL \rangle}{\|f_\KK\|^2}.
\]
Hence, it suffices to show that
\[\frac{\langle f_\LL, \Delta_{q, \text{up}}^{\LL}\, f_\LL \rangle}{\|f_\KK\|^2}=
h_{q}^{\KK, \LL},\]
or, equivalently, 
\[
\langle f_\LL, \Delta_{q, \text{up}}^{\LL}\, f_\LL \rangle
=
I_{\mathcal{P}} + B_{\mathcal{P}} + R_{\mathcal{P}}.
\]

\paragraph{Evaluating the chain $f_\LL$.}
Recall that the coefficient of $(\partial_{q+1}^\LL)^*(f_\LL)$ at the simplex $\tau\in S^\LL_{q+1}$ is given by 
\[\langle(\partial_{q+1}^\LL)^*(f_\LL),\tau\rangle = \langle f_\LL, \partial_{q+1}^\LL(\tau)\rangle.\]
Hence, 
\[\langle f_\LL, \Delta_{q, \text{up}}^\LL f_\LL \rangle= \|(\partial_{q+1}^\LL)^*(f_\LL)\|^2 = \sum_{\tau\in S_{q+1}^\LL}|\langle f_\LL, \partial_{q+1}^\LL(\tau)\rangle|^2.\]
We compute this quantity through a case-by-case analysis.

\paragraph{Case 0a: repeated labels.}
If $\tau$ contains three vertices sharing the same label (i.e., $v_{i_1}, v_{i_2}, v_{i_3}\in V^\LL_i$), or two pairs of vertices sharing the same labels (i.e., $v_{i_1}, v_{i_2}\in V^\LL_i$ and $v_{j_1}, v_{j_2}\in V^\LL_j$), then removing any single vertex leaves a $q$-face that still contains duplicate labels. Thus, $\langle f_\LL, \partial_{q+1}^\LL(\tau)\rangle=0$.

\paragraph{Case 0b: repeated labels outside the frontier.}
Assume that $\tau\in \LL^\circ \cup \KK$, and that $\tau$ contains exactly one pair of vertices with the same label. That is, we can write $\tau = [v_0v_1\cdots v_{q+1}]$ where $\mathrm{label}(v_i) = \mathrm{label}(v_{i+1})$ for some index $i$. For $j\neq i,i+1$, the face $\tau\setminus \{v_j\}$ still contains duplicated labels, so $f(\tau \setminus \{v_j\}) = 0$. Hence
\[\langle f_\LL,\partial_{q+1}^\LL(\tau)\rangle = (-1)^i f(\tau\setminus \{v_i\}) + (-1)^{i+1} f(\tau\setminus \{v_{i+1}\}) = 0,\]
since $f$ only depends on the labels of the vertices, meaning $f(\tau \setminus \{v_i\}) = f(\tau \setminus \{v_{i+1}\})$. 

\paragraph{Case 1 (interior): $\tau \in F_{\mathcal{P}}(\LL^\circ)$.}
For $\tau \in F_{\mathcal{P}}(\LL^\circ)$, all $q+2$ labels are present. Removing the vertex $v_i$ with label $i$ yields a face that misses label $i$. Because $\tau$ is in the interior, this face evaluates to
\[
f(\tau\setminus \{v_i\})=(-1)^im_{i}. 
\]
Hence, 
\[
\langle f_\LL, \partial_{q+1}^\LL(\tau)\rangle = \sum_{i=0}^{q+1}(-1)^i f(\tau\setminus \{v_i\}) = \sum_{i=0}^{q+1}(-1)^i(-1)^i m_i = m,
\]
and therefore
\[
\sum_{\tau\in F_{\mathcal{P}}(\LL^\circ)}|\langle f_\LL, \partial_{q+1}^\LL(\tau)\rangle|^2=
m^2 |F_{\mathcal{P}}(\LL^\circ)| = I_\mathcal{P}.
\]

\paragraph{Case 2 (boundary): $\tau \in F_{\mathcal{P}}(\KK^\alpha)$.}
Similarly, for $\tau \in F_{\mathcal{P}}(\KK^\alpha)$, 
\[
f(\tau\setminus \{v_i\})=(-1)^i|V^{\KK^\alpha}_{i}|. 
\]
Following the argument from the previous case, 
\[
\langle f_\LL, \partial_{q+1}^\LL(\tau)\rangle = \sum_{i=0}^{q+1} (-1)^i(-1)^i|V^{\KK^\alpha}_{i}| = |V^{\KK^\alpha}|.
\]
Summing over all such simplices yields
\[
\sum_{\tau\in F_{\mathcal{P}}(\KK^\alpha)}|\langle f_\LL, \partial_{q+1}^\LL(\tau)\rangle|^2 = |V^{\KK^\alpha}|^2 |F_{\mathcal{P}}(\KK^\alpha)|,
\]
and thus summing over all components gives $\sum_{\alpha=1}^N |V^{\KK^\alpha}|^2 |F_{\mathcal{P}}(\KK^\alpha)| = B_\mathcal{P}$.

\paragraph{Case 3 (frontier with distinct labels): $\tau\in G_{\mathcal{P}}(\Ft(\KK^\alpha), i)$.}
Recall that $i$ denotes the label of the unique vertex in $\tau$ not contained in $\KK$. If we remove this vertex $v_i$, the remaining face is in $\KK^\alpha$ and misses label $i$, giving
\[f(\tau\setminus \{v_i\}) = (-1)^i|V_i^{\KK^\alpha}|.\]
If we remove any other vertex $v_j$ (where $j\neq i$), the remaining face is not in $\KK$ and misses label $j$, yielding
\[f(\tau\setminus \{v_j\}) = (-1)^j m_j.\]
Therefore,
\[\langle f_\LL, \partial_{q+1}^\LL(\tau)\rangle =
\sum_{j=0}^{q+1}(-1)^j f(\tau\setminus \{v_j\})=
|V^{\KK^\alpha}_i| + \sum_{j\neq i}m_j
= s_i^\alpha,
\]
and 
\[
\sum_{\tau\in G_{\mathcal{P}}(\Ft(\KK^\alpha), i)}\big|\langle f_\LL, \partial_{q+1}^\LL(\tau)\rangle|^2=
(s_i^\alpha)^2 |G_{\mathcal{P}}(\Ft(\KK^\alpha), i)|.
\]

\paragraph{Case 4 (frontier with distinct labels in $\KK^\alpha$): $\tau\in H_{\mathcal{P}}(\Ft(\KK^\alpha), i)$.}
By definition, the $q$-face of $\tau$ in $\KK^\alpha$ misses label $i$. The unique vertex outside $\KK^\alpha$, say $v_{out}$, has the same label as exactly one vertex $v_{in} \in \KK^\alpha$. Let this duplicated label be $k$ (where $k \neq i$).

Because $v_{out}$ and $v_{in}$ share the same label, they are adjacent in our ordered vertices of $\tau$. Let their indices be $a$ and $a+1$. If we remove any vertex other than $v_{out}$ or $v_{in}$, the remaining $q$-face still contains duplicate labels, meaning $f$ evaluates to zero. We thus only consider removing $v_{out}$ or $v_{in}$:
\begin{itemize}
    \item Removing $v_{in}$ yields a face not in $\KK$ that misses label $i$. Thus, $f(\tau \setminus \{v_{in}\}) = (-1)^i m_i$.
    \item Removing $v_{out}$ yields a face in $\KK^\alpha$ that misses label $i$. Thus, $f(\tau \setminus \{v_{out}\}) = (-1)^i |V^{\KK^\alpha}_i|$.
\end{itemize}
Therefore, the inner product is
\begin{align*}
\langle f_\LL, \partial_{q+1}^\LL(\tau)\rangle 
&= (-1)^a f(\tau \setminus \{v_{in}\}) + (-1)^{a+1} f(\tau \setminus \{v_{out}\}) \\
&= (-1)^{a+i} m_i - (-1)^{a+i} |V^{\KK^\alpha}_i| \\
&= (-1)^{a+i+1} (|V^{\KK^\alpha}_i| - m_i) \\
&= \pm \epsilon_i^\alpha.
\end{align*}
Squaring this yields
\[
\sum_{\tau\in H_{\mathcal{P}}(\Ft(\KK^\alpha), i)}\big|\langle f_\LL, \partial_{q+1}^\LL(\tau)\rangle|^2
= \sum_{\tau\in H_\mathcal{P}(\Ft(\KK^\alpha),i)} (\epsilon_i^\alpha)^2 
= (\epsilon_i^\alpha)^2 |H_{\mathcal{P}}(\Ft(\KK^\alpha), i)|.
\]

\paragraph{Case 5 (frontier with duplicate labels in $\KK^\alpha$).}
The remaining case is that $\tau\in \Ft(\KK^\alpha)$, where the $q+1$ vertices in $\KK^\alpha$ span at most $q$ labels. This means there are two vertices in $\KK^\alpha$ sharing a label. Following the precise argument from Case 0b, we obtain $\langle f_\LL,\partial_{q+1}^\LL(\tau)\rangle = 0$.

\paragraph{Bringing it all together.}
Combining Cases 3 and 4 over all choices of $i$ and $\alpha$, we get
\[
\sum_{i=0}^{q+1}\sum_{\alpha=1}^N
\left(
\sum_{\tau\in G_{\mathcal{P}}(\Ft(\KK^\alpha),i)}
\big|\langle f_\LL,\partial_{q+1}^\LL(\tau)\rangle\big|^2
+
\sum_{\tau\in H_{\mathcal{P}}(\Ft(\KK^\alpha),i)}
\big|\langle f_\LL,\partial_{q+1}^\LL(\tau)\rangle\big|^2
\right)
=
R_{\mathcal{P}}.
\]
Combining this with Cases 1 and 2, we thus obtain
\begin{align*}
\langle f_\LL, \Delta_{q, \text{up}}^\LL f_\LL \rangle 
&= \sum_{\tau\in S_{q+1}^\LL}|\langle f_\LL, \partial_{q+1}^\LL(\tau)\rangle|^2\\ 
&= \sum_{\tau\in F_{\mathcal{P}}(\LL^\circ)}|\langle f_\LL, \partial_{q+1}^\LL(\tau)\rangle|^2 + \sum_{\tau\in F_{\mathcal{P}}(\KK)}|\langle f_\LL, \partial_{q+1}^\LL(\tau)\rangle|^2  \\
&\quad +\sum_{i=0}^{q+1}\sum_{\alpha=1}^N
\left(
\sum_{\tau\in G_{\mathcal{P}}(\Ft(\KK^\alpha),i)}
\big|\langle f_\LL,\partial_{q+1}^\LL(\tau)\rangle\big|^2
+
\sum_{\tau\in H_{\mathcal{P}}(\Ft(\KK^\alpha),i)}
\big|\langle f_\LL,\partial_{q+1}^\LL(\tau)\rangle\big|^2
\right)
\\
&= I_\mathcal{P} + B_\mathcal{P} + R_\mathcal{P}.
\end{align*}
\end{proof}

\subsection{Showing tightness}
In the following example, we show that the bound from \cref{prop:nt_cheeger_ineq_cmplx} is tight.
\begin{example}
\label{ex.cheeger.complete.tight}
\begin{figure}
    \centering

\begin{tikzpicture}[scale=2, line join=round]

  \coordinate (V1) at (0, 1.5);      
  \coordinate (V2) at (-0.9, -0.6);  
  \coordinate (V3) at (1.2, -0.3);   
  \coordinate (V4) at (0.8, 0.6);    
  \coordinate (V5) at (-1.2, 0.3);   
  \coordinate (V6) at (0, -1.5);     

  \fill[blue, fill opacity=0.1]  (V1) -- (V5) -- (V4) -- cycle;
  \fill[blue, fill opacity=0.1]  (V1) -- (V4) -- (V3) -- cycle;
  \fill[blue, fill opacity=0.3] (V6) -- (V5) -- (V4) -- cycle; 
  \fill[blue, fill opacity=0.1]  (V6) -- (V4) -- (V3) -- cycle;

  \draw[dashed, thick, black!60] (V6) -- (V4);
  \draw[dashed, thick, black!60] (V5) -- (V4);

  \fill[green, fill opacity=0.3]  (V1) -- (V5) -- (V2) -- cycle;
  \fill[blue, fill opacity=0.2] (V1) -- (V2) -- (V3) -- cycle; 
  \fill[blue, fill opacity=0.2]  (V6) -- (V5) -- (V2) -- cycle;
  \fill[blue, fill opacity=0.2]  (V6) -- (V2) -- (V3) -- cycle;

  \fill[green, fill opacity=0.3]  (V3) -- (V4) -- (V6) -- cycle;

  \draw[thick] (V1) -- (V5);
  \draw[thick] (V1) -- (V2);
  \draw[thick] (V1) -- (V3);
  \draw[dashed, thick, black!60] (V6) -- (V5);
  \draw[thick] (V6) -- (V2);
  \draw[thick] (V6) -- (V3);
  \draw[thick] (V5) -- (V2);
  \draw[thick] (V2) -- (V3);

  \draw[thick] (V1) -- (V4);
  \draw[thick] (V4) -- (V3);

  \node[circle, fill=red, draw=black, inner sep=2.5pt, label=above:\textbf{1}] at (V1) {};
  \node[circle, fill=yellow, draw=black, inner sep=2.5pt, label=left:\textbf{2}] at (V2) {};
  \node[circle, fill=red, draw=black, inner sep=2.5pt, label=right:\textbf{5}] at (V3) {};
  
  \node[circle, fill=blue, draw=black, inner sep=2.5pt, label=above right:\textbf{4}] at (V4) {};
  \node[circle, fill=blue, draw=black, inner sep=2.5pt, label=above left:\textbf{3}] at (V5) {};
  \node[circle, fill=yellow, draw=black, inner sep=2.5pt, label=below:\textbf{6}] at (V6) {};
\end{tikzpicture}
    \caption{The octahedron from \cref{ex.cheeger.complete.tight}.}
    \label{fig:octahedron}
\end{figure}
We take $\LL$ to be an octahedron and $\KK^1$ and $\KK^2$ to be the faces $[123]$ and $[456]$, respectively.\footnote{More precisely, $\KK^1$ and $\KK^2$ are the closures of the $2$-simplices $[123]$ and $[456]$ (i.e., the smallest simplicial complexes containing them).} Let $\KK=\KK^1\coprod \KK^2$ (see the green triangles in \cref{fig:octahedron}). 

First, we notice that the map $H_1(\KK;\R) \to H_1(\LL;\R)$ induced by the inclusion is zero, hence the smallest non-trivial eigenvalue satisfies $\lambda_{+}^{\KK, \LL}\neq 0$. 

We also observe that in this example $\LL^\circ=\varnothing$, and $|F_{\mathcal{P}}(\KK^1)|=|F_{\mathcal{P}}(\KK^2)|=1$ for any effective partition $\mathcal{P}$ (since an effective partition must assign all three distinct labels to the three vertices of each component). 

We then label the vertices of $[123]$ and $[456]$ with red, yellow, and blue, as shown in \cref{fig:octahedron}. We claim that this specific partition achieves the minimum for the persistent Cheeger constant $h_{1}^{\KK, \LL}$.

Indeed, in each component $\KK^\alpha$, there is exactly one vertex for each label. Thus, for each $0 \leq i \leq 2$ and $\alpha \in \{1, 2\}$, we have $|V_i^{\KK^\alpha}|=1$, $m_i=1$, and $\epsilon_i^\alpha=0$. 
Consequently, the frontier contributions from $H_{\mathcal{P}}$ vanish, and we only need to check $G_{\mathcal{P}}(\Ft(\KK^\alpha), i)$. As displayed in \cref{fig:octahedron}, for each $2$-simplex $\tau\in \Ft(\KK^\alpha)$, the vertices of $\tau$ realize only two different labels. Thus, $G_{\mathcal{P}}(\Ft(\KK^\alpha), i)=\varnothing$ for each $i$ and $\alpha$, meaning $R_{\mathcal{P}} = 0$.

Therefore, the Cheeger constant evaluates to
\[
h_1^{\KK, \LL}
=
\frac{I_{\mathcal{P}} + B_{\mathcal{P}} + R_{\mathcal{P}}}{\displaystyle \sum_{\alpha=1}^2 |V^{\KK^\alpha}|\cdot \prod_{i=0}^2|V^{\KK^\alpha}_i|}
=
\frac{0 + (3^2 \cdot 1 + 3^2 \cdot 1) + 0}{3 \cdot (1 \cdot 1 \cdot 1) + 3 \cdot (1 \cdot 1 \cdot 1)}
=
\frac{18}{6}=3.
\]

On the other hand, by a direct computation, one can deduce that $\lambda_{+}^{\KK, \LL}=3$ as well. Hence, the Cheeger-type inequality \cref{eq:cheeger_ineq_cmplt_skl} holds with equality.
\end{example}

\section{Pseudomanifolds}
\label{sec:psd_mfld}

We assume throughout that $\iota\colon \KK \hookrightarrow \LL$ is an inclusion of simplicial complexes, and that $\LL$ is a $(q+1)$-dimensional \emph{orientable pseudomanifold} (\emph{non-branching}): every $q$-simplex is in the boundary of at most two $(q+1)$-simplices, and the $(q+1)$-simplices can be oriented such that they induce opposite orientations on shared $q$-dimensional faces. Equivalently, 
the boundary matrix 
$\partial_{q+1}^\LL \colon C_{q+1}^\LL \longrightarrow C_{q}^\LL$
with respect to the oriented simplex bases has at most two nonzero entries in each row, and the product of two non-zero elements ($\pm 1)$ in the same row is $-1$. Note that the assumption on $\LL$ is satisfied for any $(q+1)$-dimensional complex embedded in $\R^{q+1}$, but not for e.g., projective spaces. 

The \emph{dual graph} of $\LL$ is the graph with vertex set of $(q+1)$-simplices, and an edge connecting two $(q+1)$-simplices with a shared face. 

The goal of this section is twofold. First, we show that for orientable pseudomanifolds the persistent up Laplacian admits a Kron-type reduction to a graph Laplacian, possibly with Dirichlet boundary conditions. Second, we use this reduction and \cref{thm:weighted_graph_cheeger_dirichlet} to obtain a two-sided Cheeger inequality.

We then specialize \cref{prop_cheeger_ineq} to the setting of pseudomanifolds (but orientable and not) and show that the resulting persistent Cheeger constant has a concrete geometric interpretation on the dual graph.

\begin{remark}
\label{rem.schaub}
Spectra for (non-persistent) Laplacians of planar pseudomanifolds
appear in \cite{grande2024disentangling}.  We believe that the Dirichlet
viewpoint in this section, together with the standard graph Laplacian interpretation for the $1$-down Laplacian, offers a transparent and natural explanation for their results concerning the first Laplacian; see, in particular, \cref{ex:circle.plane}.
\end{remark}

\subsection{A Kron reduction for pseudomanifolds}\label{sec:kron}
As shown in~\cite{MR5053447}, if $\LL$ is a pseudomanifold (non-branching in the terminology of that paper), we may choose a basis 
$c_1,\dots,c_m$ of $C_{q+1}^{\LL, \KK}$ with an appropriate ordering such that  
(i) each $c_i$ is a $0/1$ indicator chain of $(q+1)$-simplices,  
(ii) $\mathrm{supp}(c_i)\cap\mathrm{supp}(c_j)=\varnothing$ for $i\neq j$ ($\mathrm{supp}$ being the simplices on which the chain has non-zero coefficients), and  
(iii) $\partial_{q+1}^\LL c_i\in C_{q}^\KK$.  
For simplicity, we use the standard inner product on chains. Define
\[
V_i:=\langle c_i,c_i\rangle_{q+1},\qquad
\mathrm{sf}_{ij}:=\#\{\text{$q$-faces shared by $c_i,c_j$}\}
\]
\[
\mathrm{free}_i:=\#\{\text{$q$-faces of $c_i$ not shared with any $c_j$}\},
\]
\[
A_i:=\mathrm{free}_i+\sum_{j\neq i}\mathrm{sf}_{ij},\qquad
W:= \mathrm{diag}(V_1,\dots,V_m).
\]
Here $V_i$ counts $(q+1)$-simplices in $c_i$, $\mathrm{sf}_{ij}$ counts shared $q$-faces,
$\mathrm{free}_i$ counts boundary faces not shared with another $c_j$, and $A_i$ is the total
$q$-boundary ``area'' of $c_i$, and $W$ is the Gram matrix (\cref{sec.adjprop}) of the basis for $C_{q+1}^{\LL, \KK}.$ In particular, letting $Z$ have columns the chains $c_i$, we get $Z^TZ = W$ the following expression for the persistent up Laplacian:
\[\Delta_{q, \mathrm{up}}^{\KK, \LL}
=
B_{q+1}^{\LL, \KK}(Z^TZ)^{-1}(B_{q+1}^{\LL, \KK})^T 
= B_{q+1}^{\LL, \KK}W^{-1}(B_{q+1}^{\LL, \KK})^T,
\]
where $B^{\LL,\KK}_{q+1}$ be the boundary matrix of $\partial_{q+1}^\LL$ restricted to $\mathrm{span}\{c_i\}$.
\subsubsection{The Kron reduction}
\begin{proposition}
\label{kron.shared}
We have,
\[
\bigl((B^{\LL,\KK}_{q+1})^{T} B^{\LL,\KK}_{q+1}\bigr)_{ii}=A_i,
\qquad
\bigl((B^{\LL,\KK}_{q+1})^{T} B^{\LL,\KK}_{q+1}\bigr)_{ij}=-\,\mathrm{sf}_{ij}.
\]
\end{proposition}

\begin{proof}
Each $q$-face in $\partial_{q+1}^\LL c_i$ appears with coefficient $\pm1$.  
By non-branching, a face is either free (contributing $+1$ to the diagonal) or shared with exactly one $c_j$ (contributing $-1$ to the $(i,j)$ entry since orientations are opposite). Summing these contributions gives the stated formulas.
\end{proof}

Hence, from \cref{prop.dualgraph}, we have that $(B^{\LL,\KK}_{q+1})^{T} B^{\LL,\KK}_{q+1}$ is the graph Laplacian of a graph with Dirichlet boundary conditions. We illustrate the need for phantom vertices in the non-persistent setting.

\begin{example}
The $2$-down Laplacian of a single $2$-simplex is represented by the matrix $[3]$, while its dual graph consists of a single vertex. By padding each of the three edges of the $2$-simplex with simplices (in red, below), we can adjoin three phantom vertices to the dual graph. The $2$-down Laplacian then arises as the restriction of the Laplacian of this extended graph to the row and column corresponding to the original $2$-simplex. In particular, it is precisely the Dirichlet graph Laplacian on the extended graph with the boundary vertices given by the vertices $\sigma_i$ for $i>0$.
\end{example}
\begin{center}
\begin{tikzpicture}[scale=2]

\coordinate (A) at (90:1);
\coordinate (B) at (210:1);
\coordinate (C) at (-30:1);

\filldraw[fill=gray!20,draw=black] (A) -- (B) -- (C) -- cycle;

\coordinate (D) at ($(A)!0.5!(B) + (-0.5,+0.5)$);
\filldraw[fill=red!30,draw=black] (A) -- (B) -- (D) -- cycle;

\coordinate (E) at ($(B)!0.5!(C) + (0,-0.5)$);
\filldraw[fill=red!30,draw=black] (B) -- (C) -- (E) -- cycle;

\coordinate (F) at ($(C)!0.5!(A) + (0.5,0.5)$);
\filldraw[fill=red!30,draw=black] (C) -- (A) -- (F) -- cycle;

\begin{scope}[xshift=0cm]

\node[circle,fill=black,inner sep=2pt,label=above:{\(\sigma_0\)}] (s0) at (0,0) {};
\node[circle,fill=black,inner sep=2pt,label=below:{\(\sigma_1\)}] (s1) at (-0.75,0.5) {};
\node[circle,fill=black,inner sep=2pt,label=below:{\(\sigma_2\)}] (s2) at (0.75,0.5) {};
\node[circle,fill=black,inner sep=2pt,label=right:{\(\sigma_3\)}] (s3) at (0,-0.75) {};

\draw[dashed] (s0) -- (s1);
\draw[dashed] (s0) -- (s2);
\draw[dashed] (s0) -- (s3);
\end{scope}
\end{tikzpicture}
\end{center}
In general, the diagonal of $\Delta_{q+1, \mathrm{down}}^\LL$ will be constant and equal to $q+2$, see, e.g. \cite{MR3077874}. The link between $(q+1)$-down Laplacians and graph Laplacians of dual graphs with boundary conditions was made in \cite[Fig. 5]{MR3194207} for the combinatorial Laplacian. 
The following is immediate from the induced block structure on $B_{q+1}^\LL$. 
\begin{lemma}
    Let $C_1, \ldots, C_k$ denote the connected components of the dual graph of $\LL$, and let $\LL_i$ denote the simplicial complex generated by the $(q+1)$-simplices corresponding to the vertices in $C_i$. Then, \[
C_{q+1}^{\LL,\KK}=\bigoplus_{i=1}^k C_{q+1}^{\LL_i,\KK\cap\LL_i},
\qquad
\Delta_{q,\mathrm{up}}^{\KK,\LL}
=
\bigoplus_{i=1}^k
\Delta_{q,\mathrm{up}}^{\KK\cap\LL_i,\LL_i}.
\]
\end{lemma}

From this lemma, we can consider the components of $(q+1)$-simplices in the dual graph of $\widetilde{\LL}$ separately. Moreover, we shall say that $\LL$ is \emph{boundaryless} if every $q$-simplex is the face of no or two $(q+1)$-simplices.  We now arrive at the \emph{Kron reduction}.

\begin{theorem}
\label{thm.spectum.graph}
Let $\iota\colon \KK \hookrightarrow \LL$ be an inclusion of non-empty simplicial complexes, and assume that $\LL$ is an orientable pseudomanifold of dimension $(q+1)$ with connected dual graph. Let \(G\) be the weighted graph associated to the matrix
\[
X=(B^{\LL,\KK}_{q+1})^T B^{\LL,\KK}_{q+1}
\]
as in \cref{prop.dualgraph}, and let \(W=\mathrm{diag}(V_1,\ldots,V_m)\). Then the non-zero spectrum of $\Delta^{\KK,\LL}_{q,\mathrm{up}}$ coincides with:
\begin{enumerate}
    \item[(i)] the non-zero spectrum of $L_G$, if ${\LL}$ is boundaryless;
    \item[(ii)] the non-zero spectrum of $L_G^{\mathrm{dir}}$, otherwise.
\end{enumerate}
\end{theorem}

\begin{proof}
By \cref{prop:laplaceusefulLA}, the non-zero spectrum of $\Delta^{\KK,\LL}_{q,\mathrm{up}}$ coincides with the non-zero spectrum of
$
W^{-1} (B^{\LL, \KK}_{q+1})^T B^{\LL,\KK}_{q+1}.
$
This matrix is similar to the symmetric matrix
\[
M := W^{-1/2} (B^{\LL, \KK}_{q+1})^T B^{\LL, \KK}_{q+1} W^{-1/2} = W^{-1/2} X W^{-1/2},
\]
which is precisely the graph Laplacian matrix (with Dirichlet boundary conditions when appropriate) constructed in \cref{prop.dualgraph}.

By \cref{kron.shared}, the row sum of the \(i\)-th row of \(X\) is precisely \(\mathrm{free}_i\). Thus all row sums vanish if and only if \(\mathrm{free}_i=0\) for every \(i\). Equivalently, every \(q\)-face appearing in the boundary of a \((q+1)\)-simplex in the support of the \(c_i\)'s is shared with another such \((q+1)\)-simplex. Since the dual graph is connected, this is equivalent to \(\LL\) being boundaryless. In this case no Dirichlet boundary vertices are added, so \(L_G^{\mathrm{dir}}=L_G\). Otherwise some \(\mathrm{free}_i>0\), and the corresponding positive row sum gives a Dirichlet contribution.

\end{proof}

\subsubsection{A persistent Cheeger inequality}\label{sec:psd_mfld_cheeger_ineq}
Given the reduction to a graph with Dirichlet boundary conditions, a Cheeger inequality for the pseudomanifold case is near-immediate from \cref{thm:weighted_graph_cheeger_dirichlet}. However, for the sake of clarity, we recall the definitions in the notation used here.

For any nonempty subset \(S\subseteq \{1,\ldots,m\}\), where \(\{1,\ldots,m\}\) indexes the chains \(c_i\), set
\[
\operatorname{Vol}(S):=\sum_{i\in S}V_i,\qquad 
\operatorname{Cut}(S):=\sum_{\substack{i\in S\\ j\notin S}}\mathrm{sf}_{ij},\qquad 
\operatorname{Bdy}(S):=\sum_{i\in S}\mathrm{free}_i.
\]
We also write \(V_{\mathrm{tot}}:=\operatorname{Vol}(\{1,\ldots,m\})\).

\begin{definition}\label{def.cheeger.const}
Define
\[
h_{\mathrm{vol}}:= 
\min_{\substack{\varnothing\neq S\subseteq \{1,\ldots,m\}\\ \operatorname{Vol}(S)\leq V_{\mathrm{tot}}/2}} 
\frac{\operatorname{Cut}(S)}{\operatorname{Vol}(S)},
\qquad
h^{\mathrm{Bdy}}_{\mathrm{vol}}:= 
\min_{\varnothing\neq S\subseteq \{1,\ldots,m\}} 
\frac{\operatorname{Cut}(S)+\operatorname{Bdy}(S)}{\operatorname{Vol}(S)}.
\]
\end{definition}

\begin{theorem}
\label{thm.pers.cheeger}
Let \(\Delta:=\max_{1\leq i\leq m} A_i/V_i\), and let \(\KK\hookrightarrow\LL\), where \(\LL\) is a \((q+1)\)-dimensional orientable pseudomanifold whose dual graph is connected. Let \(\lambda_{\min}^+\) denote the smallest non-zero eigenvalue of
the persistent up Laplacian \(\Delta_{q,\mathrm{up}}^{\KK,\LL}\). 
\begin{enumerate}
    \item If \(\LL\) is boundaryless, then
    \[
    \frac{h_{\mathrm{vol}}^2}{2\Delta}\leq \lambda_{\min}^+\leq 2h_{\mathrm{vol}}.
    \]
    \item Otherwise,
    \[
    \frac{(h^{\mathrm{Bdy}}_{\mathrm{vol}})^2}{2\Delta}\leq \lambda_{\min}^+\leq h^{\mathrm{Bdy}}_{\mathrm{vol}}.
    \]
\end{enumerate}
\end{theorem}
\begin{proof}
From \cref{thm.spectum.graph}, the smallest nonzero eigenvalue of \(\Delta_{q,\mathrm{up}}^{\KK,\LL}\) agrees with the smallest relevant eigenvalue of
\[
M:=W^{-1/2}(B_{q+1}^{\LL,\KK})^TB_{q+1}^{\LL,\KK}W^{-1/2},
\]
where \(W=\operatorname{diag}(V_1,\ldots,V_m)\). The result now follows from \cref{thm:weighted_graph_cheeger_dirichlet} by noting that for $X:=(B_{q+1}^{\LL,\KK})^TB_{q+1}^{\LL,\KK}$,
\[
X_{ii}=A_i,\qquad X_{ij}=-\mathrm{sf}_{ij}\quad (i\neq j).
\]
and
\[
b_i=\sum_jX_{ij}=A_i-\sum_{j\neq i}\mathrm{sf}_{ij}=\mathrm{free}_i.
\]
Moreover, the vertex weights in the graph theorem are \(\mu_i=V_i\), and hence
\[
\operatorname{Vol}(S)=\sum_{i\in S}V_i,\qquad
\operatorname{Cut}(S)=\sum_{\substack{i\in S\\j\notin S}}\mathrm{sf}_{ij},\qquad
\operatorname{Bdy}(S)=\sum_{i\in S}\mathrm{free}_i,
\]
which are exactly the quantities in \cref{def.cheeger.const}. Also,
\[
\Delta=\max_i\frac{X_{ii}}{\mu_i}=\max_i\frac{A_i}{V_i}.
\]

The interior graph associated with \(X\) is the dual graph on the chains \(c_i\), with edge weights \(\mathrm{sf}_{ij}\), and is connected by assumption. Therefore \cref{thm:weighted_graph_cheeger_dirichlet} applies. If \(\LL\) is boundaryless, then \(\mathrm{free}_i=0\) for all \(i\), so all row sums of \(X\) vanish. Hence
\[
\frac{h_{\mathrm{vol}}^2}{2\Delta}\leq \lambda_{\min}^+\leq 2h_{\mathrm{vol}}.
\]
If \(\LL\) has boundary, then at least one \(\mathrm{free}_i\) is positive, so at least one row sum of \(X\) is positive. In this case \cref{thm:weighted_graph_cheeger_dirichlet} gives
\[
\frac{(h_{\mathrm{vol}}^{\mathrm{Bdy}})^2}{2\Delta}\leq \lambda_{\min}^+\leq h_{\mathrm{vol}}^{\mathrm{Bdy}}.
\]
\end{proof}
We emphasize that, unlike in the setting of
\cite{MR3194207}, our result (in the non-persistent case) is
fundamentally combinatorial: the Cheeger constants are non-zero and have a
direct geometric interpretation.  This transparency, however, comes at a cost:
the lower bound requires orientability; see \cref{sec.klein.failure} and
compare with \cite[Theorem~2.7\,(2)]{MR3194207}.

The following example shows how the Kron reduction provides a transparent understanding of the $1$-st persistent up Laplacian for planar simplicial complexes.
\begin{example}
\label{ex:circle.plane}
Consider the following filtration of two-dimensional simplicial complexes in $\R^2$,
\[
\KK_1 \hookrightarrow \KK_2 \hookrightarrow \cdots \hookrightarrow \KK_9 = \LL.
\]
The complexes $\KK_i$ for $i \in \{2,4,7,9\}$ are shown in \cref{fig.filtration}. 
From \cref{prop:laplaceusefulLA} we know that 
$
\lambda_{1, {\min}}^{\KK_i,\LL}
$, the smallest non-zero eigenvalue of $\Delta_{1, \mathrm{up}}^{\KK_i, \LL}$
coincides with the smallest non-zero eigenvalue of 
\[
M := W^{-1/2} (B^{\KK_i,\LL}_{2})^T B^{\KK_i,\LL}_{2}\, W^{-1/2},
\]
which is a Dirichlet Laplacian operator on the $2$-chains $\{c_1,\ldots,c_m\}$ by \cref{thm.spectum.graph}. In \cref{fig.2simp.evalues} we show, for each $i\in\{2,4,7,9\}$, the relative magnitude of the entries 
$v_j$ of an eigenvector $v$ corresponding to the smallest non-zero eigenvalue of $M$, plotted on the 
corresponding chain $c_j$.  
In this example, the chains $c_j$ are precisely the $2$-simplices of $\KK_i$, together with the single chain 
filling the inner hole.  
Plots for the eigenvector of the second-smallest non-zero eigenvalue are shown in 
\cref{fig.2simp.evalues.second}. For comparison, \cref{fig.contlap} shows a numerical solution of the eigenvalue problem 
$\nabla^2 u = -\lambda u$ on the unit disk with Dirichlet boundary condition $u|_{\partial D}=0$.  
The similarity to the persistent Laplacian eigenvectors is obvious.

Finally, \cref{fig.filtration.circle.cheeger} shows how $\lambda^{\KK_i,\LL}_{1, {\min}+}$ evolves along the 
filtration, together with the corresponding upper and lower bounds on the Cheeger constant.  
The red and green curves indicate the Cheeger cut costs for removing the hole-filling chain and for 
cutting off  $\KK_i$ (the outer part), respectively. 

\end{example}

\begin{figure}[h!]
    \centering
    \begin{subfigure}{0.24\textwidth}
        \centering
        \includegraphics[width=\textwidth]{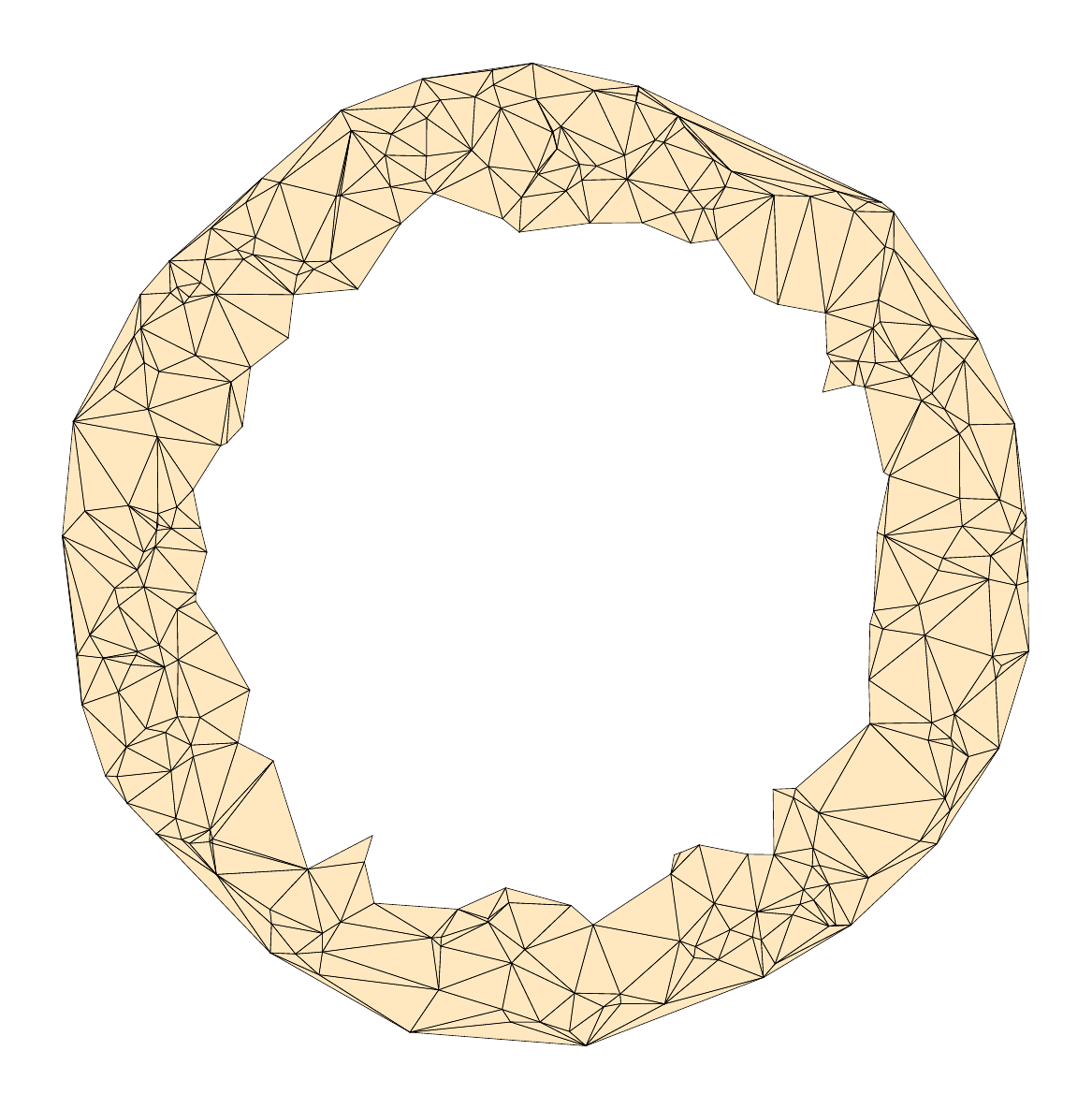}
    \end{subfigure}
        \begin{subfigure}{0.24\textwidth}
        \centering
        \includegraphics[width=\textwidth]{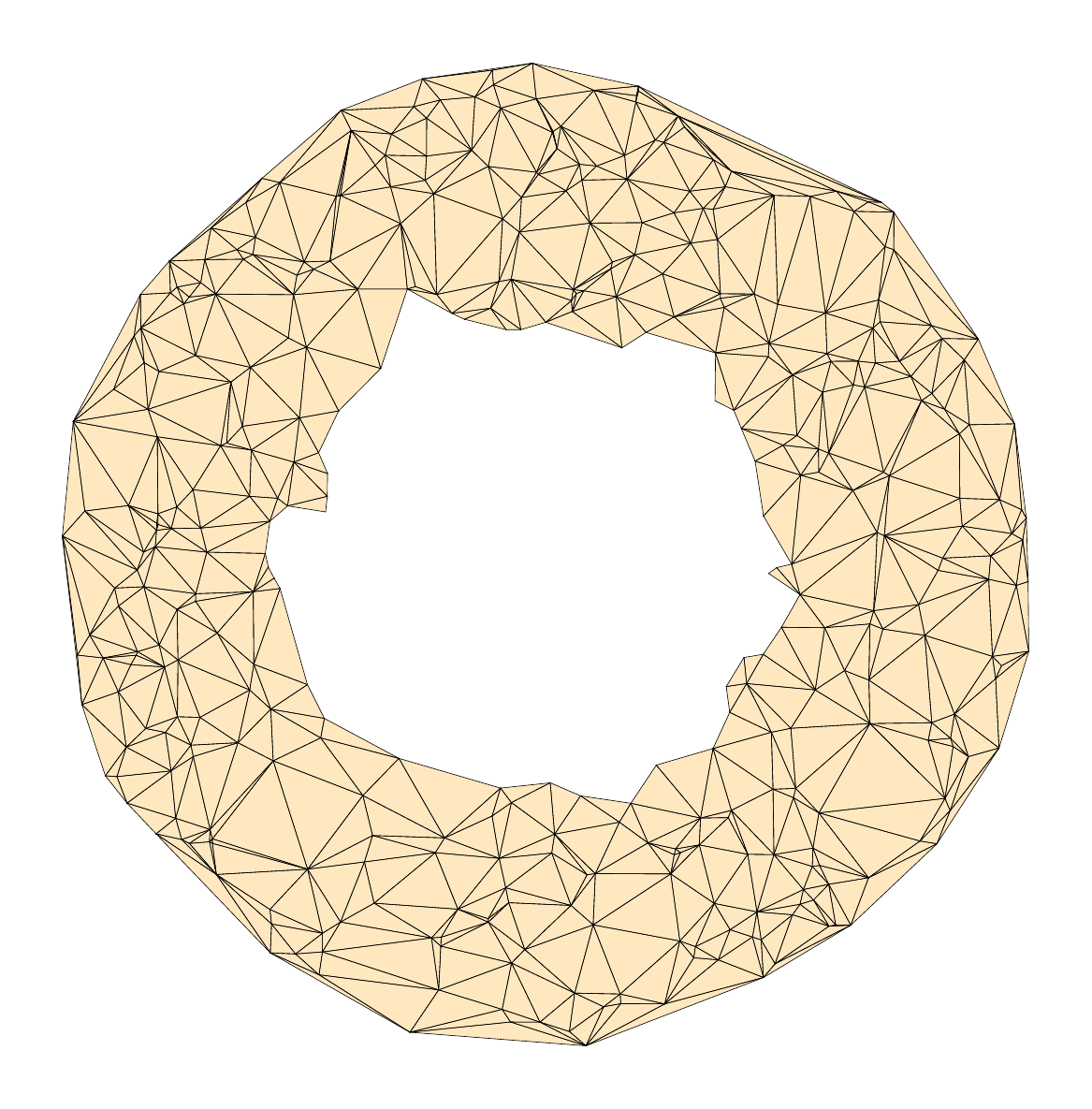}
    \end{subfigure}
    \begin{subfigure}{0.24\textwidth}
        \centering
        \includegraphics[width=\textwidth]{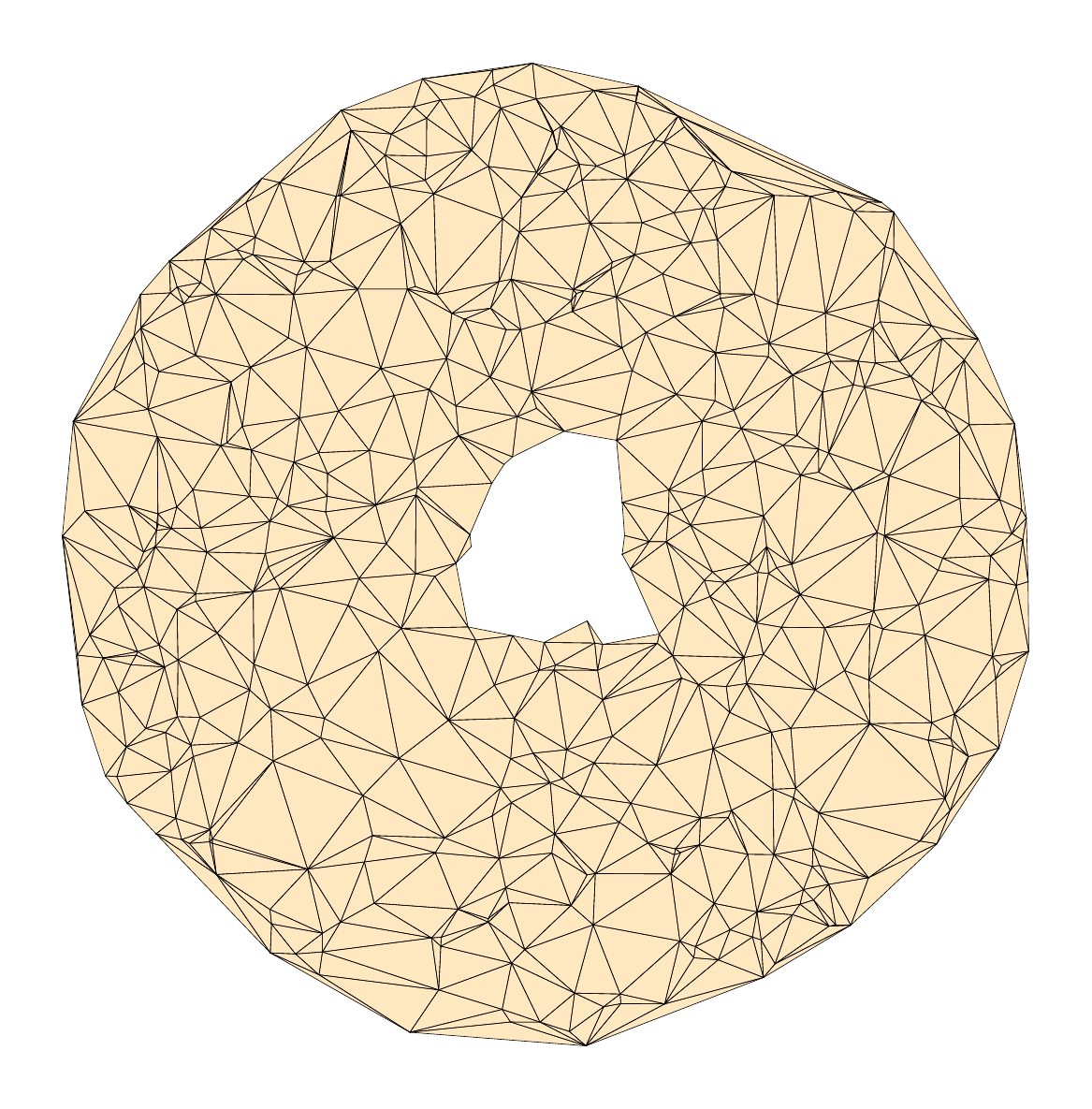}
    \end{subfigure}
    \begin{subfigure}{0.24\textwidth}
        \centering
        \includegraphics[width=\textwidth]{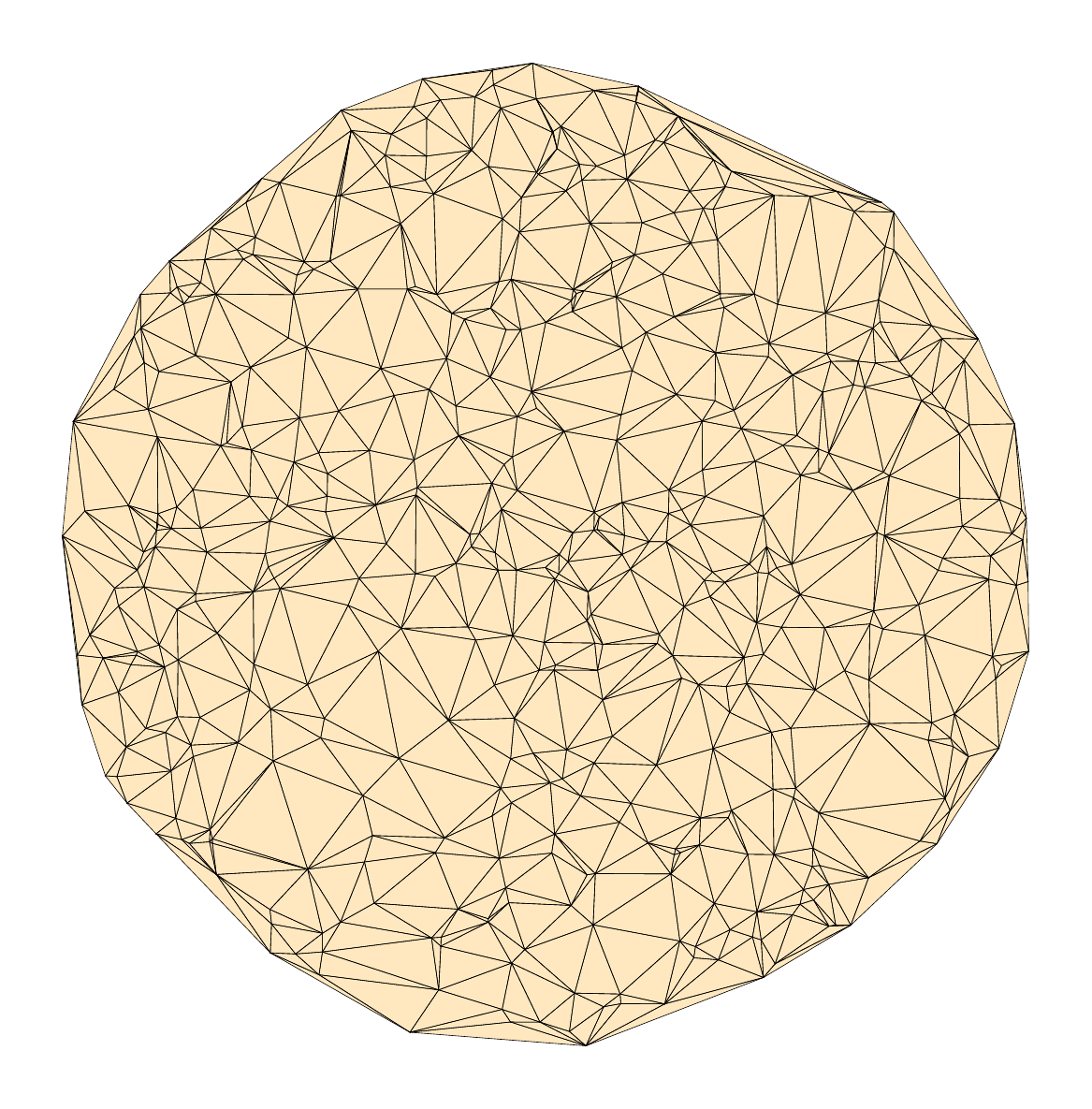}
    \end{subfigure}
    \caption{A filtration of a planar simplicial complex; see \cref{ex:circle.plane}.}
    \label{fig.filtration}
\end{figure}

\begin{figure}[h!]
    \centering
    \begin{subfigure}{0.39\textwidth}
        \centering
        \includegraphics[width=\textwidth]{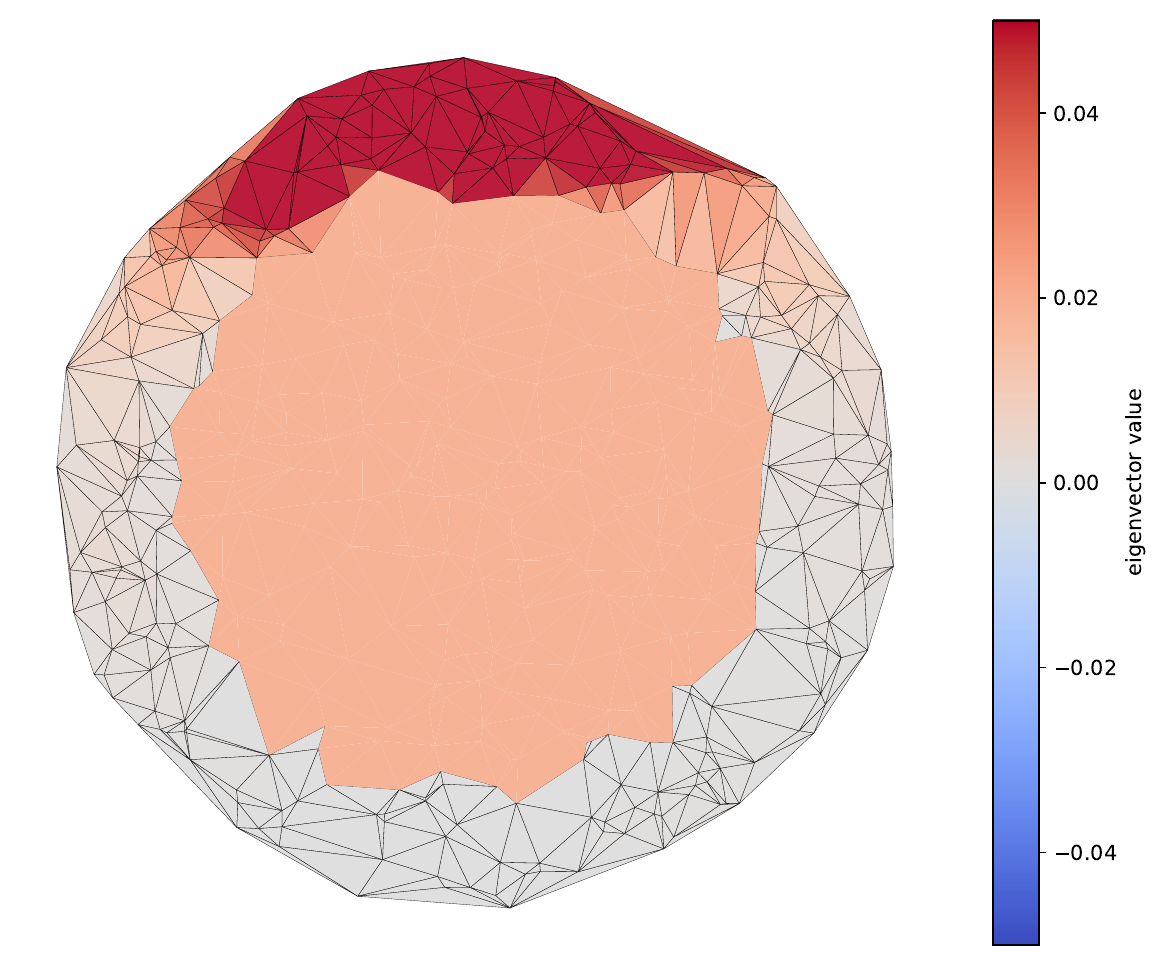}
    \end{subfigure}
    \begin{subfigure}{0.39\textwidth}
        \centering
        \includegraphics[width=\textwidth]{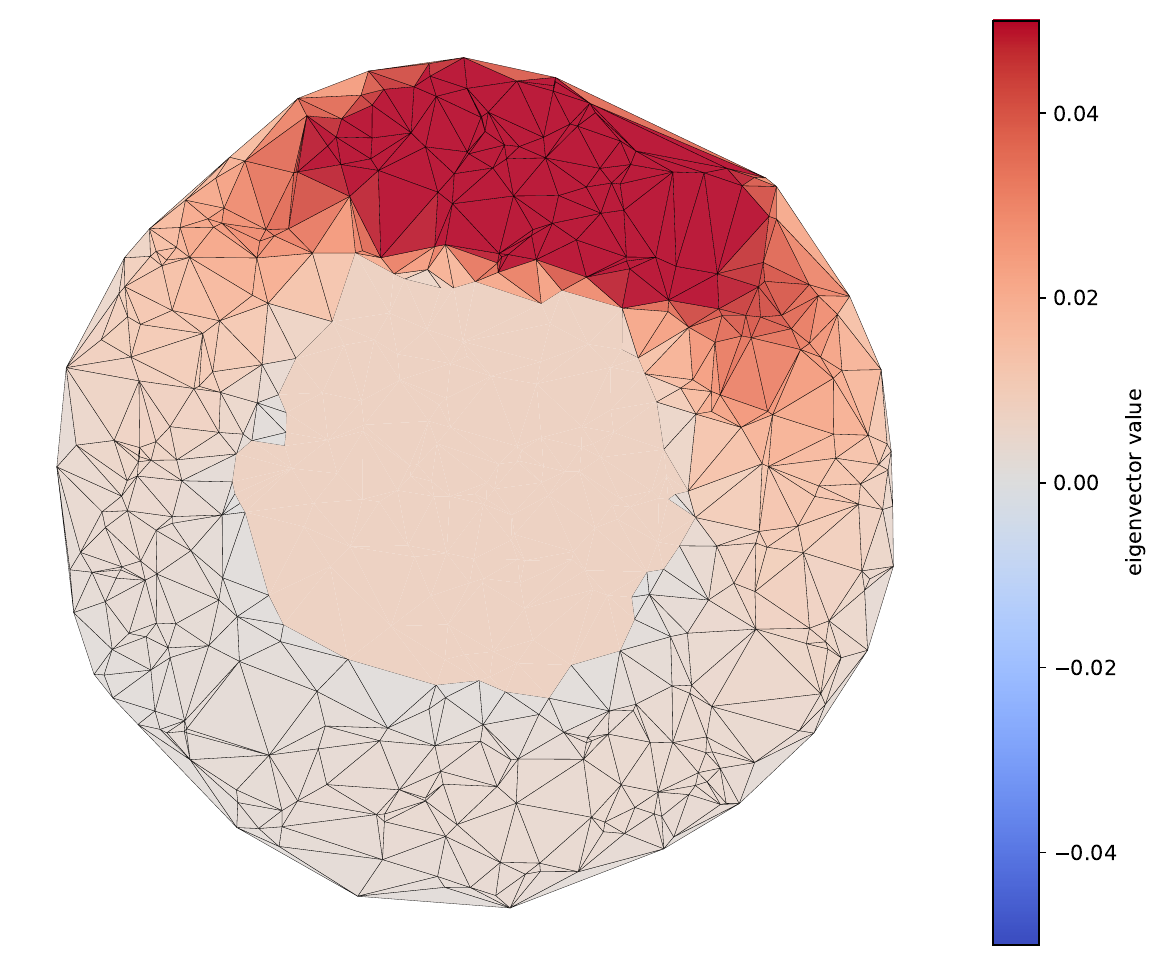}
    \end{subfigure}\\
    \begin{subfigure}{0.39\textwidth}
        \centering
        \includegraphics[width=\textwidth]{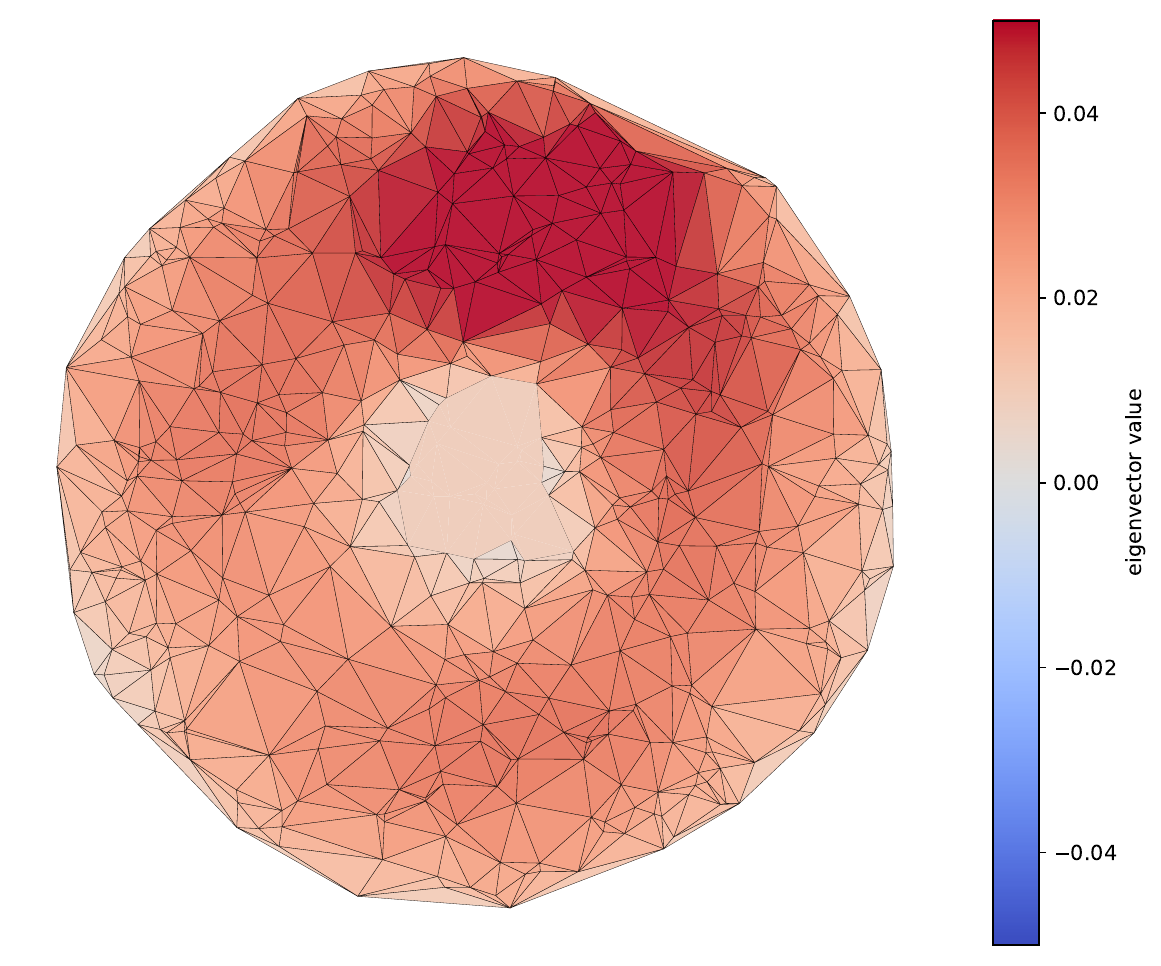}
    \end{subfigure}
        \begin{subfigure}{0.39\textwidth}
        \centering
        \includegraphics[width=\textwidth]{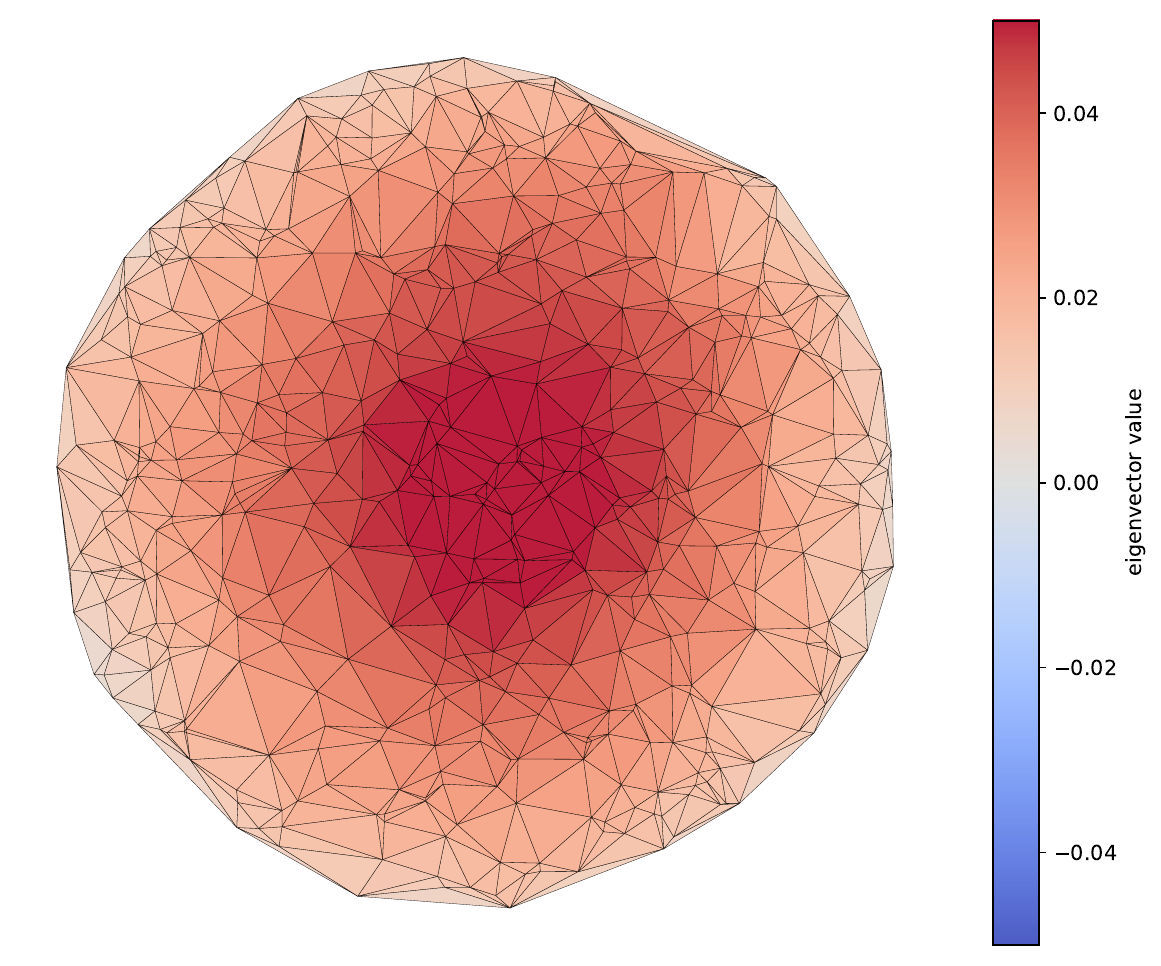}
    \end{subfigure}
    \caption{Eigenvector entries on the $2$-simplices for the smallest non-zero eigenvalue of $M$, for $\KK_i \hookrightarrow \LL$ with $i \in \{2,4,7,9\}$; see \cref{ex:circle.plane}.}
    \label{fig.2simp.evalues}
\end{figure}

\begin{figure}[h!]
    \centering
    \begin{subfigure}{0.39\textwidth}
        \centering
        \includegraphics[width=\textwidth]{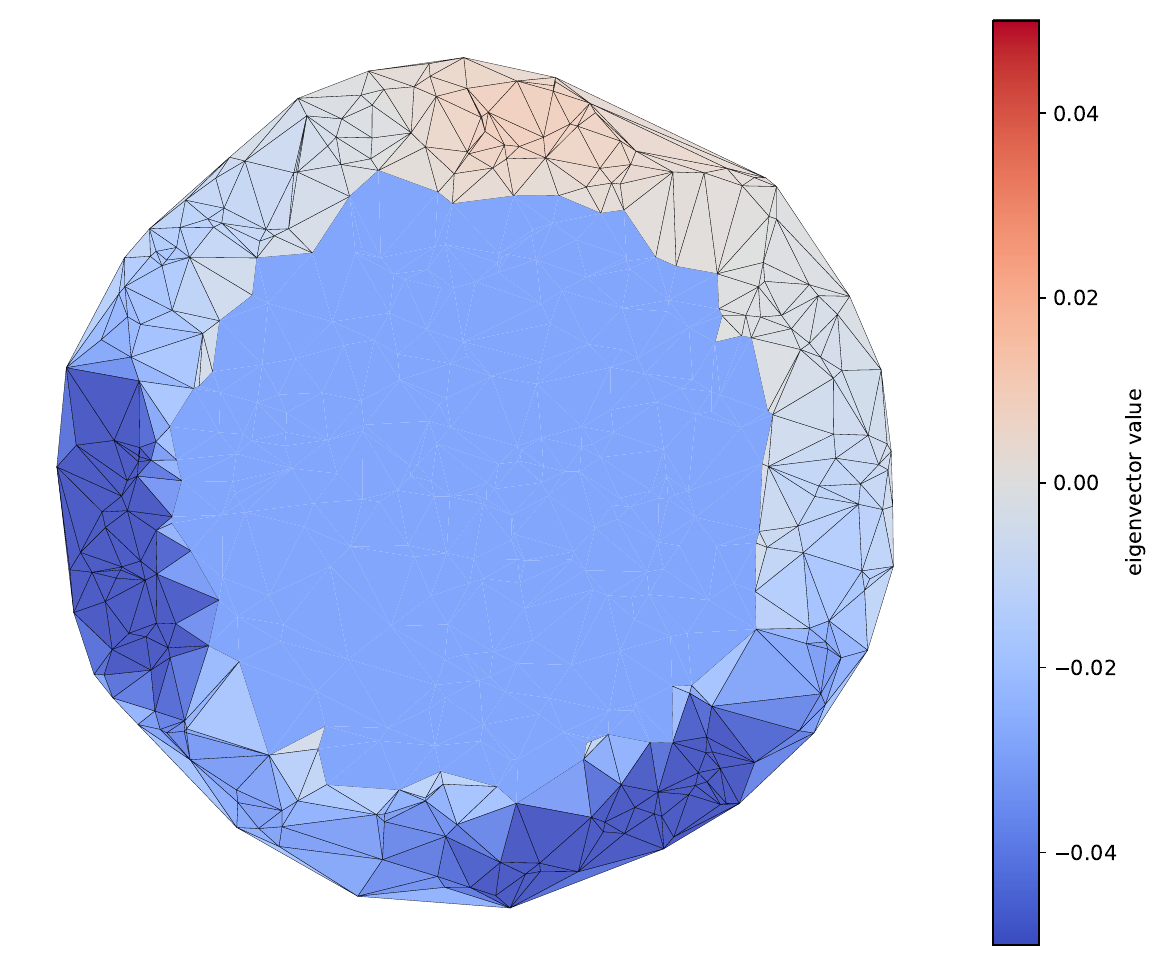}
    \end{subfigure}
    \begin{subfigure}{0.39\textwidth}
        \centering
        \includegraphics[width=\textwidth]{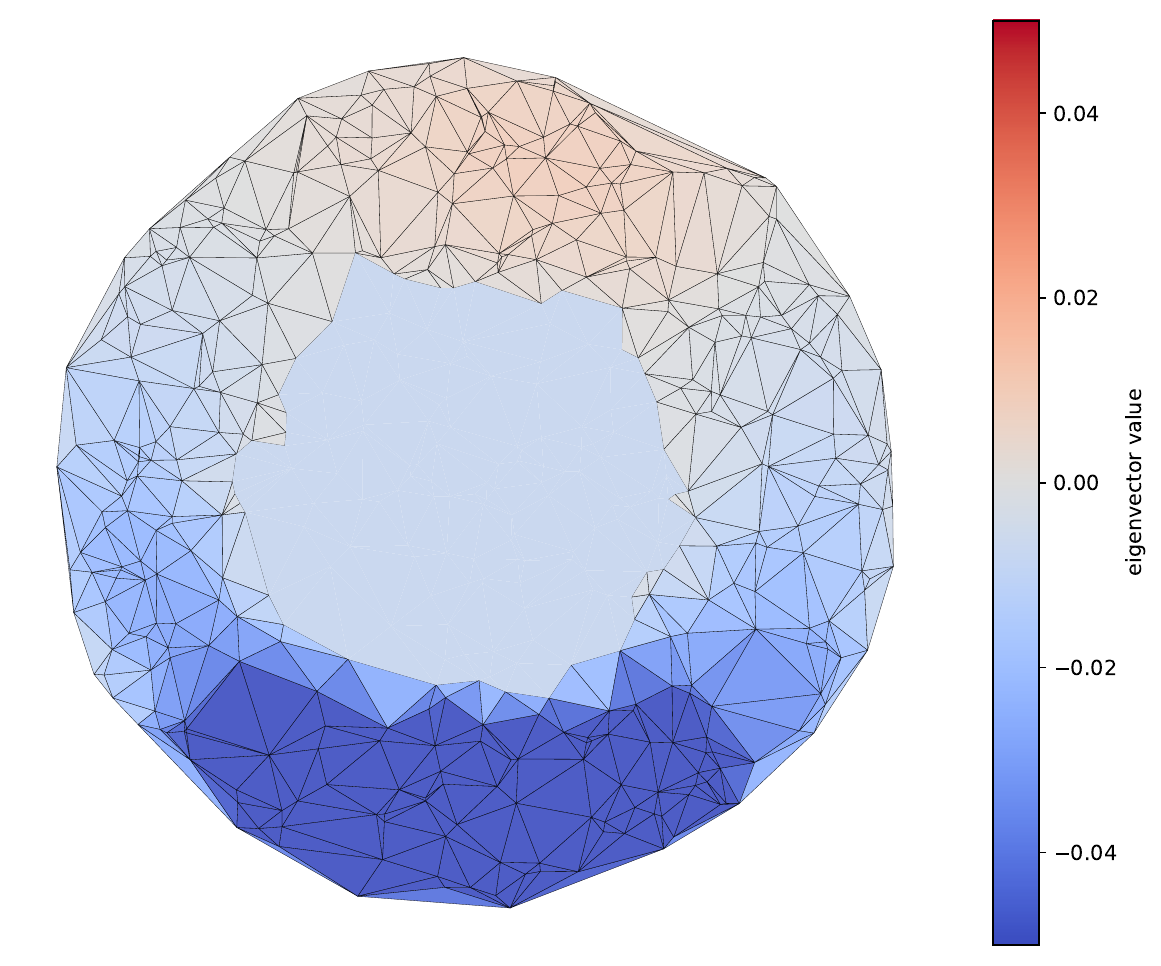}
    \end{subfigure}\\
    \begin{subfigure}{0.39\textwidth}
        \centering
        \includegraphics[width=\textwidth]{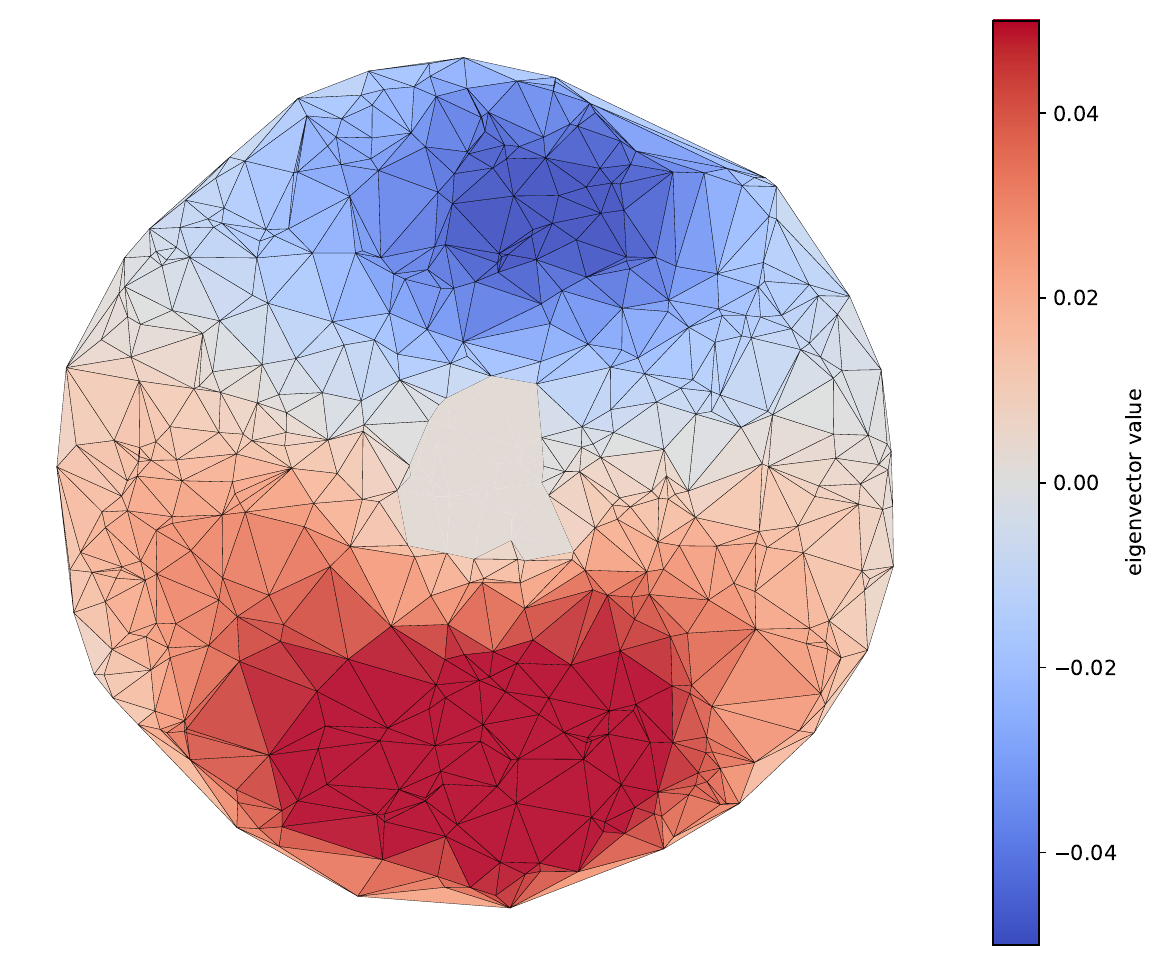}
    \end{subfigure}
        \begin{subfigure}{0.39\textwidth}
        \centering
        \includegraphics[width=\textwidth]{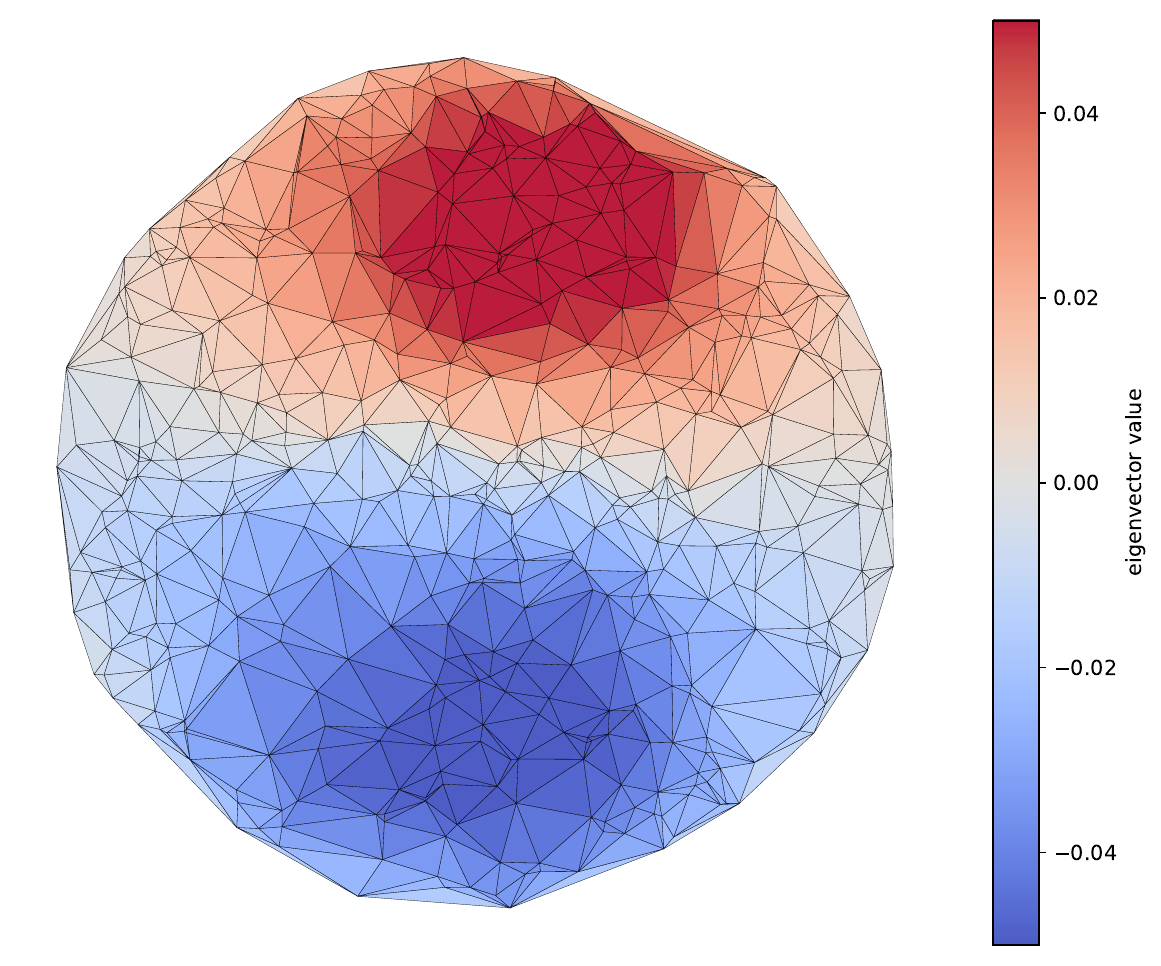}
    \end{subfigure}
    \caption{Eigenvector entries on the $2$-simplices for the second-smallest non-zero eigenvalue of $M$, for $\KK_i \hookrightarrow \LL$ with $i \in \{2,4,7,9\}$; see \cref{ex:circle.plane}.}

    \label{fig.2simp.evalues.second}
\end{figure}

\begin{figure}
    \includegraphics[width=0.7\textwidth]{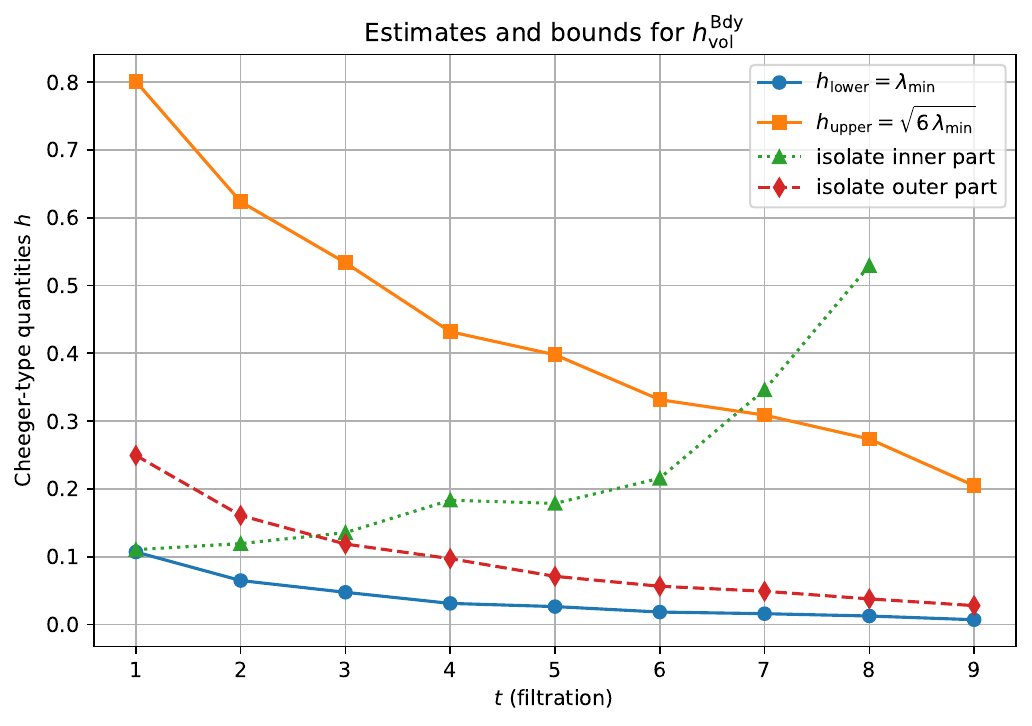}
    \label{fig:disk-graphs}
    \caption{Evolution of the smallest nonzero eigenvalue of the persistent Laplacian for \(\KK_t\hookrightarrow \LL\), together with the corresponding Cheeger bounds for a filtration of the unit disk; see \cref{ex:circle.plane}. In particular, \(h_{\mathrm{vol}}^{\mathrm{Bdy}}\) lies between the two lower curves at each time step. For small \(t\), the chain filling the inner disk has large volume relative to its boundary area, and it either realizes or closely approximates the Cheeger constant. For large \(t\), the situation is nearly symmetric, where it becomes more efficient to isolate the outer region formed by the chains contained in \(\KK_t\).}
    \label{fig.filtration.circle.cheeger}
\end{figure}

\begin{figure}[h!]
    \centering
    \begin{subfigure}{0.39\textwidth}
        \centering
        \includegraphics[width=\textwidth]{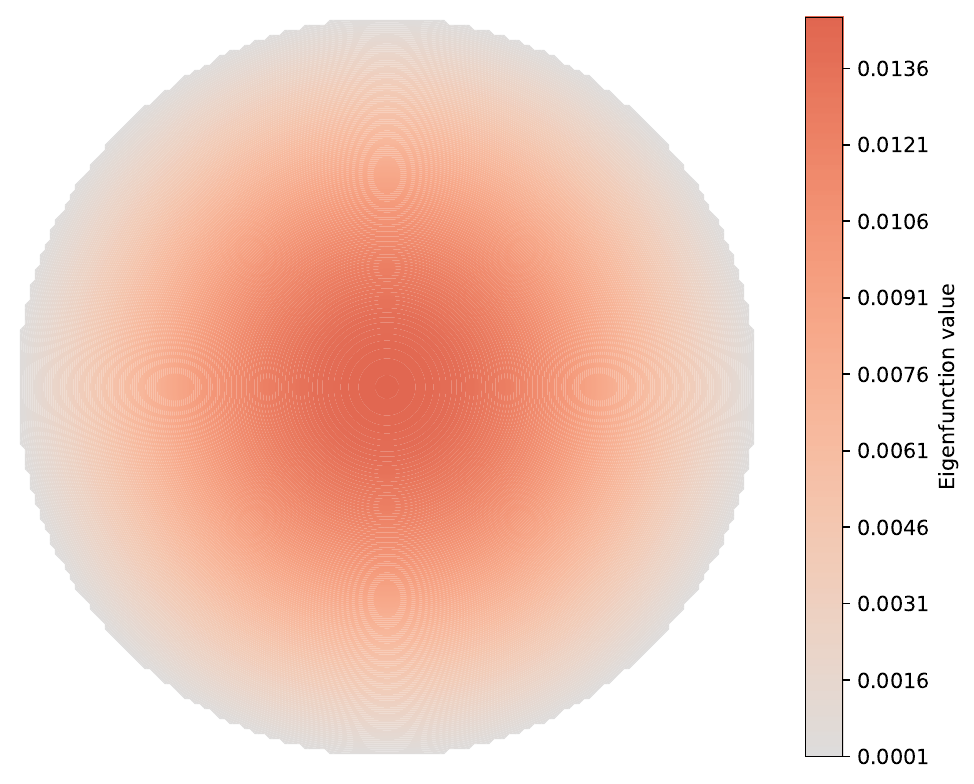}
    \end{subfigure}
    \begin{subfigure}{0.39\textwidth}
        \centering
        \includegraphics[width=\textwidth]{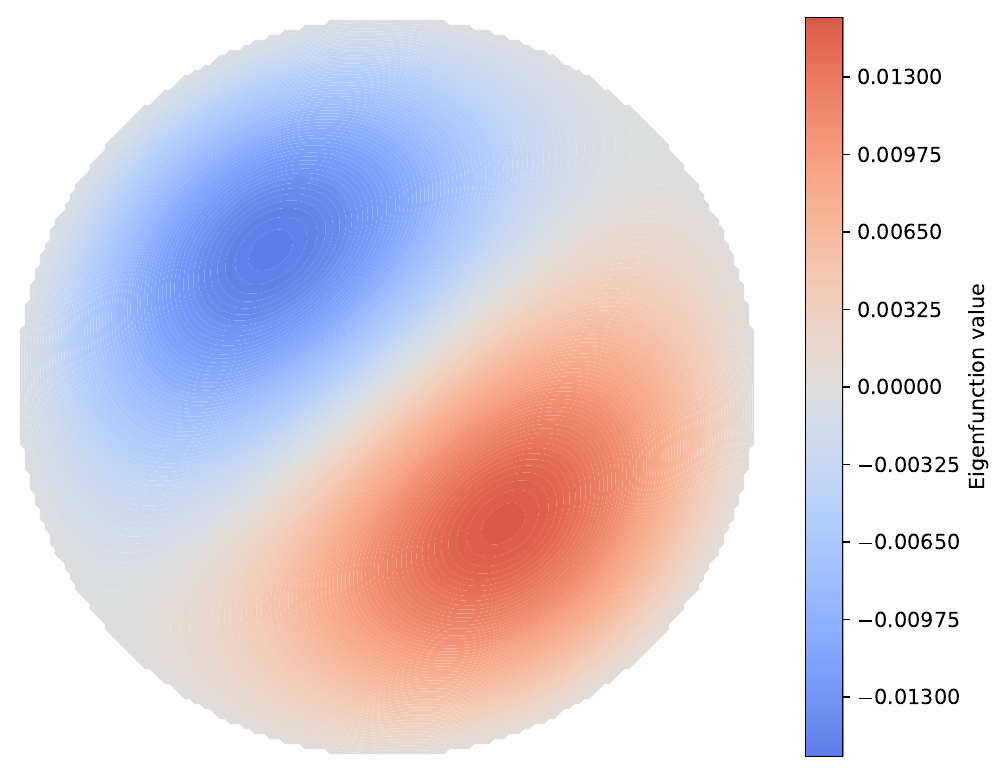}
    \end{subfigure}
   \caption{A numerical solution of the eigenvalue problem 
$\nabla^2 u = -\lambda u$ on the unit disk with Dirichlet boundary condition $u|_{\partial D}=0$; see \cref{ex:circle.plane}.}

    \label{fig.contlap}
\end{figure}

\subsubsection{Failure in the non-orientable setting}

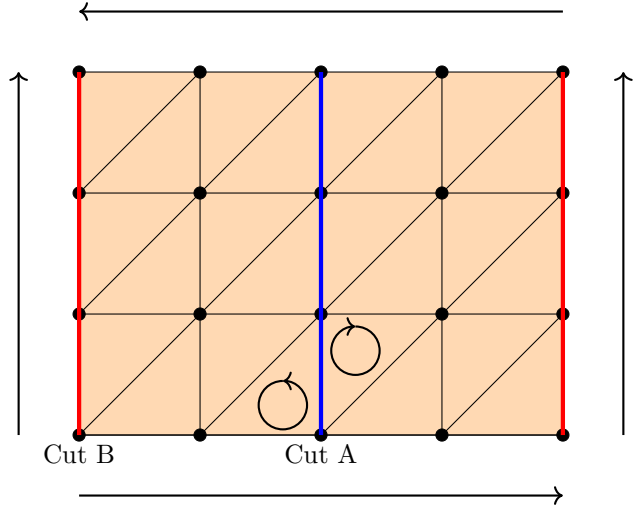
\begin{figure}
\centering
\begin{tikzpicture}[scale=1.6,
    v/.style={circle,fill=black,inner sep=1.5pt}
]

\draw[fill=orange!30] (0,0) -- (4,0) -- (4,3) -- (0,3) -- cycle;

\foreach \x/\y in {0/0, 1/0, 2/0, 0/1, 1/1, 2/1, 0/2, 1/2, 2/2, 3/0, 3/1, 3/2}
{\draw (\x, \y) -- (\x+1, \y+1); \draw (\x, \y) -- (\x+1, \y); \draw (\x, \y)-- (\x, \y+1);}

\foreach \x/\y in {0/0, 1/0, 2/0, 0/1, 1/1, 2/1, 0/2, 1/2, 2/2, 0/3, 1/3, 2/3, 3/3, 3/0, 3/1, 3/2, 4/0, 4/1, 4/2, 4/3}
{\draw[color=black, fill=black] (\x,\y) circle (.05);}

\draw[thick, ->] (4.5, 0) -- (4.5,3);
\draw[thick, ->] (-0.5, 0) -- (-0.5,3);
\draw[thick, ->] (0, -0.5) -- (4,-0.5);
\draw[thick, <-] (0, 3.5) -- (4,3.5);

\draw[ultra thick, blue] (2,0) -- (2,3);
\draw[ultra thick, red] (0,0) -- (0,3);
\draw[ultra thick, red] (4,0) -- (4,3);
\node[below] at (0,0) {Cut B};
\node[below] at (2,0) {Cut A};

\draw[->, thick] (1.65,0.25) +(80:0.2) arc(-270:90:0.2);
\draw[->, thick] (2.25,0.7) +(80:0.2) arc(90:-270:0.2);

\end{tikzpicture}
\caption{A triangulated Klein bottle.}
\label{fig:klein}
\end{figure}

\label{sec.klein.failure}
We show that no analogue of the Cheeger-type lower bound in \cref{thm.pers.cheeger} can hold for non-orientable closed surfaces. Our focus is on \(\Delta_{q+1,\mathrm{down}}^\LL\), since the
non-zero spectra of \(\Delta_{q,\mathrm{up}}^\LL\) and
\(\Delta_{q+1,\mathrm{down}}^\LL\) agree for every simplicial complex
\(\LL\).

We shall use the following result on \emph{expander graphs}, i.e., families
of bounded-degree graphs whose Cheeger constants are uniformly bounded
below by a positive constant.

\begin{theorem}[\cite{ivrissimtzis2019trivalent}]
There exist a constant \(c>0\) and a sequence \((\LL_k)_{k\geq 2}\) of
finite triangulations of closed orientable surfaces such that the number
of two-dimensional simplices of \(\LL_k\) tends to infinity, and
\[
    \phi(G_k) \geq c
\]
for all \(k\), where \(\phi\) is given in
\cref{def:graph_asymm_cheeger}, and \(G_k\) denotes the dual graph of \(\LL_k\).
\label{thm.dualgraph.expander}
\end{theorem}

Note that in the non-persistent setting, partitioning the \(2\)-simplices of \(\LL_k\)
is precisely the same as partitioning the vertices of the dual graph
\(G_k\). Hence
\[
    h_{\mathrm{vol}}(\LL_k)=\phi(G_k).
\]

For the purposes of our example, we slightly modify the construction.
Let \(\mathbb U_k\) be the non-orientable surface obtained by taking the
simplicial connected sum of \(\LL_k\) with the triangulation of the
Klein bottle \(\mathbb K\) shown in \cref{fig:klein}. That is, we remove
one \(2\)-simplex from each complex and then pairwise identify the
remaining boundary edges.

Let \(\lambda_{\min}^k\) denote the smallest eigenvalue of \(\Delta_{2,\mathrm{down}}^{\mathbb{U}_k}\).
This eigenvalue is non-zero, since $\Delta_{2}^{\mathbb{U}_k}=\Delta_{2,\mathrm{down}}^{\mathbb{U}_k}$ and 
\[
   \ker \Delta_{2}^{\mathbb{U}_k} \cong H_2(\mathbb U_k;\mathbb R)=0,
\]
as \(\mathbb U_k\) is a closed non-orientable surface.

\begin{proposition}
There exists a constant \(\epsilon>0\) such that
\[
    h_{\mathrm{vol}}(\mathbb U_k)>\epsilon
\]
for all \(k\), while
\[
    \lim_{k\to\infty}\lambda_{\min}^k=0.
\]
\end{proposition}

\begin{proof}
Let \(H_k\) be the dual graph of \(\mathbb U_k\). Then \(H_k\) is
obtained from \(G_k\) by removing one vertex from \(G_k\), removing one
vertex from the fixed dual graph of the triangulated Klein bottle
\(\mathbb K\), and joining the three resulting degree-two vertices on
each side so that the resulting graph is again \(3\)-regular. Thus
\((H_k)\) is obtained from \((G_k)\) by a modification involving only a
uniformly bounded number of vertices and edges. In particular, the
families \((H_k)\) and \((G_k)\) are uniformly \emph{quasi-isometric}, and therefore the lower bound on \(h_{\mathrm{vol}}\) follows
from \cref{thm.dualgraph.expander}; see, for instance,
\cite[Lemma~4.4]{sawicki2020super}. 

It remains to show that \(\lambda_{\min}^k\to 0\). Consider the
triangulation of the Klein bottle in \cref{fig:klein}. Let the vertical
line \(x=0\) be Cut~B, and let the line \(x=1/2\) be Cut~A. Orient all
triangles to the left of Cut~A counterclockwise, as seen in
\(\mathbb R^2\), and all triangles to the right of Cut~A clockwise. Let
\(S\) be the set of \(2\)-simplices whose barycenter \((a,b)\) satisfies
\(a<1/2\), and let \(S^c\) be its complement.

Remove a triangle from \(S\) which does not intersect either cut, and
perform the connected sum with \(\LL_k\). Choose the orientation on the
triangles of \(\LL_k\) so that, along the connected-sum boundary, the
induced orientations cancel with those of the adjacent Klein-bottle
triangles. In particular, the connected-sum seam does not contribute to
the boundary of the chain defined below.

Define \(x_k\in C_2(\mathbb U_k)\) by
\[
(x_k)_\sigma =
\begin{cases}
+1, & \sigma\in S^c,\\
-1, & \text{otherwise}.
\end{cases}
\]
Here the second case includes all \(2\)-simplices in the \(\LL_k\)
summand. By the choice of orientations, the contributions along Cut~B
cancel, whereas the contributions along Cut~A do not cancel. Moreover,
there are no contributions along the connected-sum seam. Hence
\(\partial_{2}^{\mathbb{U}_k}x_k\) is supported on the finitely many edges of Cut~A.

Therefore
\[
    x_k^T\Delta_{2,\mathrm{down}}^{{\mathbb U}_k}x_k
    =
    \|\partial_{2}^{\mathbb{U}_k}x_k\|^2 = 12
\]

Hence the Rayleigh quotient gives
\[
    0<\lambda_{\min}^k
    \leq
    \frac{x_k^T\Delta_{2,\mathrm{down}}^{\mathbb U_k}x_k}
         {\|x_k\|^2}
    \longrightarrow 0.
\]
\end{proof}

Consequently, there cannot exist constants \(C>0\) and \(m>0\) such that
\[
    \lambda_{\min}\geq C\cdot h_{\mathrm{vol}}^m
\]
holds uniformly over all closed non-orientable triangulated surfaces.

\subsection{The Cheeger constant \texorpdfstring{$\varphi_q^{\KK,\LL}$}{phi(q) K,L}}
\label{sec:cheeger_phi_psdmfld}
Let $\KK\hookrightarrow \LL$ be an inclusion where $\LL$ is a $(q+1)$-dimensional orientable pseudomanifold. Since \(\LL\) is an orientable non-branching pseudomanifold, the transpose of \(B_{q+1}^{\LL,\KK}\) can be identified, after choosing orientations, with the signed incidence matrix of an associated dual graph. Thus \(B_{q+1}^{\LL,\KK}\) is totally unimodular; see \cref{sec:TU}.

Leveraging this observation, we obtain a transparent description of the nonzero persistent Cheeger constant $\varphi_{q}^{\KK, \LL}$ (see \cref{def:nonzero_cheeger}) in this setting. At the end of the section, we shall also consider the case where $\LL$ has boundary.

\subsubsection{Some results from graph theory}
We first recall that the distance between two vertices $u, v\in V$ in a graph $G=(V, E)$, denoted by $\dist(u, v)$, is defined as the length of the shortest path joining $u$ and $v$. The diameter of $G$, denoted by $\diam(G)$, is defined as the maximum distance over all pairs of vertices in $G$ (see \cite[Chapter 3]{MR1421568}).

A well-known result bridging linear programming and spectral graph theory demonstrates that the diameter of a connected graph $G$ can be deduced via the incidence matrix $I$ (see \cref{sec:introduction}).
We present the result as follows.

\begin{lemma}\label{lemma:ungrd_graph_diam}
We have the equality
\begin{equation}\label{eq:ungrd_graph_diam}
\min_{x\perp^1\ker(I)}\frac{\|Ix\|_1}{\|x\|_1}=
\frac{2}{\diam(G)}.
\end{equation}
\end{lemma}

Since the graph $G=(V, E)$ is connected by assumption,
therefore,
the incidence matrix $I$ has rank $|V|-1$, 
and if we remove one row from $I$, 
the remaining rows of this matrix have the full row rank.
We refer to the node corresponding to the removed row as the \emph{ground vertex} and this remaining matrix as the \emph{reduced} incidence matrix of the graph $G$.
Similarly to \cref{lemma:ungrd_graph_diam},
the reduced incidence matrix gives rise to the maximum distance between the ground vertex and other vertices.  

\begin{lemma}\label{lemma:grd_graph_diam}
Let $u\in V$ be the ground vertex and $R$ be the reduced incidence matrix.
Then
\begin{equation}\label{eq:grd_graph_diam}
\min_{x\perp^1\ker(R)}\frac{\|Rx\|_1}{\|x\|_1}=
\min_{i\in V\setminus\{u\}}\frac{1}{\dist(i,u)}.
\end{equation}
\end{lemma}

We were unable to find good references for these results in the literature, and thus we provide the proofs of \cref{lemma:ungrd_graph_diam} and \cref{lemma:grd_graph_diam}, along with further details in \cref{sec:missing_proof}.
\begin{remark}
The salient point is that the (reduced) incidence matrix of a graph is totally unimodular (an observation made by Poincar\'{e} \cite{poincare}) and thus by the framework of linear programming,
the flow vector that reaches the minimum is integral, corresponding precisely to a discrete shortest path, giving rise to $\diam(G)$ or $\dist(i, u)$.
\end{remark}

\subsubsection{Pseudomanifolds without boundary}
Recall the block form 
\[
B_{q+1}^{\LL}
=
\begin{bmatrix}
B\\
D
\end{bmatrix},
\]
where the row indices of \(B\) and \(D\) correspond to \(S_q^{\KK}\) and
\(S_q^{\LL}\setminus S_q^{\KK}\), respectively, as described in \cref{sec:cheeger_perslap}. We choose the basis $\{c_1,\cdots c_m\}$ of $C_{q+1}^{\LL, \KK}\subset C_{q+1}^\LL$ as described in \cref{sec:kron}.
Let $Z$ be the column matrix representation of this basis,
i.e.,
$Z=\begin{bmatrix}
    c_1 & \cdots & c_m
\end{bmatrix}$ and thus $P=Z(Z^TZ)^{-1}Z^T$ is a projection matrix. 
In the following lemma, we do not assume that $\LL$ is without boundary.
\begin{lemma}\label{lemma:PZ_l_1_norm}
We have the equality $\|PB^Tx\|_1=\|Z^TB^Tx\|_1$ for an arbitrary $x\in C_q^\KK$.
\end{lemma}

\begin{proof}
Recall from \cref{sec:kron} that each $c_i$ is a $0/1$ indicator vector, 
and $\mathrm{supp}(c_i)\cap\mathrm{supp}(c_j)=\varnothing$ for $i\neq j$. A well-known result in linear algebra states that the $\ell^1$-norm of a matrix is equal to the maximum $\ell^1$-norm of all its column vectors,
hence $\|Z(Z^TZ)^{-1}\|_1=1$.

Thus we have the following inequality
\[
\|PB^Tx\|_1=\|Z(Z^TZ)^{-1}Z^TB^Tx\|_1
\leq \|Z(Z^TZ)^{-1}\|_1\cdot\|Z^TB^Tx\|_1=\|Z^TB^Tx\|_1.
\]
On the other hand,
the equality $Z^TPB^T=Z^TB^T$ implies that $\|Z^TB^Tx\|_1\leq \|Z^T\|_1\cdot\|PB^Tx\|_1$.
Since $\|Z^T\|_1=1$,
we have 
\[\|Z^TB^Tx\|_1\leq \|PB^Tx\|_1.
\]
Therefore
\[
\|PB^Tx\|_1=\|Z^TB^Tx\|_1.
\]
\end{proof}

We can then give a description of the nonzero persistent Cheeger constant $\varphi_{q}^{\KK, \LL}$ via the corresponding dual graph as below.

\begin{proposition}\label{prop:pers_cheeger_nbrch_clsd}
Let $\KK\hookrightarrow \LL$ where $\LL$ is a $(q+1)$-dimensional orientable pseudomanifold without boundary. Let $G$ be the dual graph of $\LL$, i.e., the graph with incidence matrix given by the transpose of $B_{q+1}^{\LL, \KK}$. Then the Cheeger constant from \cref{def:nonzero_cheeger} takes the form 
\begin{equation}\label{eq:cheeger_clsd_non_branching}
\varphi_{q}^{\KK, \LL}
=\frac{2}{\diam(G)}.
\end{equation}
\end{proposition}

\begin{proof}
From \cref{lemma:PZ_l_1_norm},
 \[
 \varphi_{q}^{\KK, \LL}
    =
    \min_{x\perp^1\ker(PB^T)}\frac{\|PB^Tx\|_1}{\|x\|_1}
    =
    \min_{x\perp^1\ker(PB^T)}\frac{\|Z^TB^Tx\|_1}{\|x\|_1}.
    \]

Since $B_{q+1}^{\LL,\KK} = BZ$ \cite[Lemma 3.4]{pers_lap}, we have that $(BZ)^T$ is the incidence matrix of the dual graph associated to $\KK\hookrightarrow \LL$. The result follows from \cref{lemma:ungrd_graph_diam} by noting that \cref{lemma:PZ_l_1_norm} implies
\[
    \ker(PB^T)=\ker(Z^TB^T),
\]
since \(\|PB^Tx\|_1=0\) if and only if
\(\|Z^TB^Tx\|_1=0\). 
\end{proof}

\begin{remark}
\label{rem.pseudo.jz}
When $\KK=\LL$,
\cref{eq:cheeger_clsd_non_branching} reduces to 
\begin{equation}\label{eq:cheeger_const_psdmfld}
\varphi_q^\LL=\frac{2}{\diam(G)},
\end{equation}
where $G$ is the dual graph of $\LL$.
This is compatible with the claim in the proof of \cite[Theorem 2.3]{jost2024cheeger}.
We recall that in \cite[Definition 2.3]{jost2024cheeger} Jost and Zhang defined a \emph{non-trivial} Cheeger constant as
\begin{equation}\label{eq:JZ_Cheeger_const}
\varphi_{q}^{\operatorname{JZ}}(\LL)=
\min_{x\perp^1\im(B_q^T)}\frac{\|(B_{q+1}^\LL)^Tx\|_1}{\|x\|_{1, \operatorname{deg}}}.
\end{equation}
See also \cref{remark.connection.to.jost}.

There are two key differences between $\varphi_{q}^{\operatorname{JZ}}(\LL)$ in \cref{eq:JZ_Cheeger_const} and $\varphi_{q}^\LL$ in \cref{eq:non_zero_cheeger}.
\begin{itemize}
\item $\varphi_q^\LL$ is non-zero,
while $\varphi_{q}^{\operatorname{JZ}}(\LL)$ is equal to $0$ if and only if $H_q(\LL, \mathbb{R})\neq 0$:
\item In the definition of $\varphi_q^\LL$,
the denominator is the $\ell^1$-norm $\|x\|_1$,
while in the definition of $\varphi_{q}^{\operatorname{JZ}}(\LL)$ Jost and Zhang used the degree-weighted $\ell^1$-norm $\|x\|_{1, \operatorname{deg}}$,
we recall that $\|x\|_{q, \operatorname{deg}}:=\sum |x_i|\cdot\deg(\sigma_i)$ for a $q$-chain $x=\sum x_i\sigma_i$,
here $\deg(\sigma)=\#\{\tau\in S_{q+1}^\LL: \sigma\subset \tau\}$.
\end{itemize}
We recall a claim in \cite{jost2024cheeger} below:

\noindent \textbf{Claim: } Let $\LL$ be an orientable, closed $(q+1)$-dimensional pseudomanifold.
Let $H_1(\LL, \mathbb{R})=0$.
Then 
\[\varphi_q^{\operatorname{JZ}}(\LL)=\frac{1}{\diam(G)},\]
here $G$ is the dual graph of $\LL$.

This is compatible with \cref{eq:cheeger_const_psdmfld}.
When $H_q(\LL, \mathbb{R})=0$,
$\im ((B_q^\LL)^T)=\ker((B_{q+1}^\LL)^T)$,
and since $\LL$ has no boundary,
each $q$-simplex in $\LL$ has degree $2$,
thus,
for an arbitrary $x\in C_q^\LL$,
$\|x\|_1=\frac{1}{2}\|x\|_{1, \deg}$,
therefore,
\[\varphi_q^\LL=2\cdot \varphi_q^{\operatorname{JZ}}(\LL).\]
\end{remark}
\subsubsection{Pseudomanifolds with boundary}
When the simplicial complex $\LL$ has boundary,
i.e.,
there exists at least one $q$-simplex in $\LL$ that is not shared by two different $(q+1)$-simplices in common, the matrix $(BZ)^T$ can be regarded as the reduced incidence matrix of some graph $G$, with a ground vertex $u$. That is, one can add another column to $(BZ)^T$ representing the additional vertex, such that the resulting matrix is an incidence matrix of a graph. Observe that in this setting, the ground vertex $u$ in the dual graph $G$ corresponds to the missing polyhedra whose boundary contains all the $q$-simplices that are not the facets of two different $(q+1)$-simplices.

As in the previous section, we refer to this graph $G$ as the dual graph of the inclusion $\KK\hookrightarrow \LL$. 

\begin{proposition}\label{prop:pers_cheeger_nobrch_bdry}
Let $\KK\hookrightarrow\LL$ where $\LL$ is a $(q+1)$-dimensional orientable pseudomanifold with boundary. Let $G=(V, E)$ be the dual graph of $\KK\subset \LL$, and $u$ be the ground vertex.
Then 
\begin{equation}\label{eq:cheeger_bdry_non_branching}
\varphi_{q}^{\KK, \LL}
=\min_{i\in V}\frac{1}{\dist(i, u)}.
\end{equation}
\end{proposition}

\begin{proof}
Precisely as in the proof of \cref{prop:pers_cheeger_nbrch_clsd},
we have the equality
\[
 \varphi_{q}^{\KK, \LL}
    =
    \min_{x\perp^1\ker(PB^T)}\frac{\|PB^Tx\|_1}{\|x\|_1}
    =
    \min_{x\perp^1\ker(PB^T)}\frac{\|Z^TB^Tx\|_1}{\|x\|_1}.
    \]
Since $B_{q+1}^{\LL,\KK} = BZ$ \cite[Lemma 3.4]{pers_lap}, we have that $(BZ)^T$ is the reduced incidence matrix of the dual graph $G$. Thus the result follows from \cref{lemma:grd_graph_diam}.
\end{proof}

\subsubsection{Examples}
We give two examples to explain how to apply \cref{prop:pers_cheeger_nbrch_clsd} to compute the Cheeger constant $\varphi_{q}^{\KK, \LL}$.
In \cref{ex:ungrounded} the matrix $Z^TB^T$ is identical to the incidence matrix of an ungrounded graph,
and in \cref{ex:grounded} the matrix $Z^TB^T$ is identical to the incidence matrix of a grounded graph.
\begin{example}\label{ex:ungrounded}
Consider $\KK\hookrightarrow \LL$,
where where \(\LL\) is the boundary of the \(3\)-simplex $[1234]$, and
\(\KK\) is the \(2\)-simplex \([123]\) together with its faces. In this
case \(\LL\) has no boundary; see \cref{fig:combined}
\begin{figure}[htbp]
    \centering
    
    \begin{subfigure}[b]{0.45\textwidth}
        \centering
        \begin{tikzpicture}[line join = round, line cap = round]
            \def\side{2.5}
            \coordinate [label=above:4] (4) at (0,{sqrt(2) * \side},0);
            \coordinate [label=left:3] (3) at ({-.5*sqrt(3) * \side},0,-.5);
            \coordinate [label=below:2] (2) at (0,0,\side);
            \coordinate [label=right:1] (1) at ({.5*sqrt(3) * \side},0,-.5*\side);

            \begin{scope}[decoration={markings,mark=at position 0.5 with {\arrow{to}}}]
                \draw[densely dotted,postaction={decorate}] (1)--(3);
                \draw[fill=lightgray,fill opacity=.5] (2)--(1)--(4)--cycle;
                \draw[fill=gray,fill opacity=.5] (3)--(2)--(4)--cycle;
                \draw[fill=red,fill opacity=.15] (3)--(2)--(1)--cycle;
                \draw[postaction={decorate}] (2)--(1);
                \draw[postaction={decorate}] (2)--(3);
                \draw[postaction={decorate}] (3)--(4);
                \draw[postaction={decorate}] (2)--(4);
                \draw[postaction={decorate}] (1)--(4);
            \end{scope}
        \end{tikzpicture}
        \caption{Simplicial complexes $\KK\hookrightarrow \LL$.}%
        \label{fig:tetrahedron}
    \end{subfigure}
    \hfill
    \begin{subfigure}[b]{0.45\textwidth}
        \centering
        \begin{tikzpicture}[
            node distance=3cm,
            main node/.style={circle, draw, font=\sffamily\Large\bfseries, minimum size=0.8cm,  fill=blue!10, thick}
        ]
          \node[main node] (1) {1};
          \node[main node] (2) [right=of 1] {2};

          \path[thick]
            (1) edge [bend left=40] node[above] {$e_1$} (2)
            (2) edge [pos=0.5]      node[above] {$e_2$} (1)
            (1) edge [bend right=40] node[below] {$e_3$} (2);
        \end{tikzpicture}
        
        \vspace{0.8cm} 
        
        \caption{The dual graph $G$.}
        \label{fig:graph}
    \end{subfigure}
    
    \caption{Figures in \cref{ex:ungrounded}}
    \label{fig:combined}
\end{figure}.
The boundary matrix $B_2^\LL$ is 
\[
B_2^\LL
=
\begin{bNiceMatrix}[
  first-row,code-for-first-row=\scriptstyle,
  first-col,code-for-first-col=\scriptstyle,
]
& [123] & [124] & [134] & [234]\\
[12] &  1 & 1 & 0 & 0\\
[13] &  -1 & 0 & 1 & 0\\
[23] &  1 & 0 & 0 & 1\\
[14] &  0 & -1 & -1 & 0\\
[24] &  0 & 1 & 0 & -1\\
[34] &  0 & 0 & 1 & 1
\end{bNiceMatrix},
\]
and
\[
B=
\begin{bmatrix}
 1 & 1 & 0 & 0\\
 -1 & 0 & 1 & 0\\
 1 & 0 & 0 & 1\\
\end{bmatrix},
\quad
Z=
\begin{bmatrix}
    1 & 0\\
    0 & -1\\
    0 & 1\\
    0 & -1\\
\end{bmatrix},
\quad \text{ and thus }
(BZ)^T=
\begin{bmatrix}
 1 & -1 & 1\\
 -1 & 1 & -1\\
\end{bmatrix}.
\]
Consider $(BZ)^T$ as the incidence matrix of a dual graph $G$ containing two vertices and three edges. Since $\diam(G)=1$,
\[\varphi_{1}^{\KK, \LL}=\frac{2}{1}=2.\]
\end{example}

\begin{example}\label{ex:grounded}
Consider $\KK\subset \LL$ as in \cref{fig:grounded_graph}. 
Here $\LL$ has a boundary and $\KK$ is the $1$-dimensional subcomplex given by the outer four edges, and the diagonal, all colored in blue.
\begin{figure}[h]
    \centering
    
    \begin{subfigure}[b]{0.55\textwidth}
        \centering
        \begin{tikzpicture}
            \def\side{3.5}

            \fill[fill=yellow!40, fill opacity=0.5] (0, 0) -- (\side, 0) -- (1/2 * \side, \side * 0.8660254) -- cycle;
            \fill[fill=yellow!40, fill opacity=0.5] (\side, 0) -- (1/2 * \side, \side * 0.8660254) -- (\side * 3/2, \side * 1.732051 / 2) -- cycle;

            \draw (0,0) -- (30: 0.577*\side); 
            \draw (1/2*\side, \side * 0.8660254) -- (1/2*\side, 0.289*\side); 
            \draw (1/2*\side, 0.289*\side) -- (\side, 0); 
            \draw (\side, 0.577*\side) -- (\side, 0); 
            \draw (\side, 0.577*\side) -- (1/2*\side, 0.866*\side); 
            \draw (\side, 0.577*\side) -- (3/2*\side, \side * 1.732051 / 2); 

            \draw[blue, very thick] (0, 0) -- (\side, 0); 
            \draw[blue, very thick] (\side, 0) -- (1/2 * \side, \side * 0.8660254); 
            \draw[blue, very thick] (1/2 * \side, \side * 0.8660254) -- (0,0); 
            \draw[blue, very thick] (1/2 * \side, \side * 0.8660254) -- (3/2 * \side, \side * 1.732051 / 2); 
            \draw[blue, very thick] (3/2 * \side, \side * 1.732051 / 2) -- (\side, 0); 
            
            \node at (1/2 * \side, \side * 0.8 + 0.5) {1};
            \node at (3/2 * \side, \side * 0.8 + 0.5) {4};
            \node at (0, -0.4) {2};
            \node at (\side, -0.4) {3};
            \node at (0.4*\side, 0.35*\side) {5};
            \node at (0.92*\side, 0.5*\side) {6};

            \fill (0, 0) circle (2pt);  
            \fill (1/2 * \side, \side * 0.8660254) circle (2pt); 
            \fill (3/2 * \side, \side * 1.732051 / 2) circle (2pt); 
            \fill (\side, 0) circle (2pt); 
            \fill (1/2*\side, 0.289*\side) circle (2pt); 
            \fill (\side, 0.577*\side) circle (2pt); 
        \end{tikzpicture}
        \caption{The pair of simplicial complexes $\KK\subset \LL$.}
        \label{fig:triangle_k_l}
    \end{subfigure}
    \hfill 
    \begin{subfigure}[b]{0.4\textwidth}
        \centering
        \begin{tikzpicture}[
            main node/.style={circle, draw, font=\sffamily\Large\bfseries, minimum size=0.8cm, fill=blue!10, thick},
            ground node/.style={circle, draw, fill=gray!20, font=\sffamily\Large\bfseries, minimum size=0.8cm, thick}
        ]

          \node[main node]   (1) at (0, 3) {1};
          \node[main node]   (2) at (4, 3) {2};
          \node[ground node] (G) at (2, 0) {0}; 

          \path[thick]
            
            (1) edge [bend right=20] node[left=2pt]  {} (G)
            (1) edge [bend left=20]  node[right=2pt] {} (G)
            
            (2) edge                 node[above]     {} (1)
            
            (G) edge [bend left=20]  node[left=2pt]  {} (2)
            (2) edge [bend left=20]  node[right=2pt] {} (G);
        \end{tikzpicture}
        \caption{The associated grounded dual graph.}
        \label{fig:grounded_graph}
    \end{subfigure}
    \caption{The pair of simplicial complexes and the associated dual graph in \cref{ex:grounded}.}
    \label{fig:combined_figure}
\end{figure}
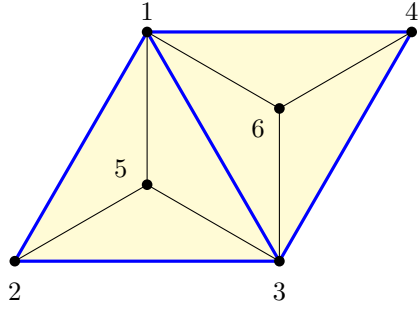
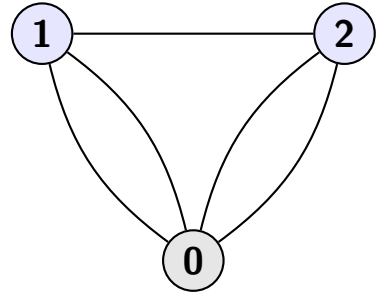
We get 
\[
B_2^\LL 
=
\begin{bNiceMatrix}[
  first-row,code-for-first-row=\scriptstyle,
  first-col,code-for-first-col=\scriptstyle,
]
& [125] & [153] & [235] & [136] & [164] & [346]\\
[12] & 1 & 0 & 0 & 0 & 0 & 0\\
[23] & 0 & 0 & 1 & 0 & 0 & 0\\
[13] & 0 & -1 & 0 & 1 & 0 & 0\\
[14] & 0 & 0 & 0 & 0 & -1 & 0\\
[34] & 0 & 0 & 0 & 0 & 0 & 1\\
[15] & -1 & 1 & 0 & 0 & 0 & 0\\
[25] & 1 & 0 & -1 & 0 & 0 & 0\\
[35] & 0 & -1 & 1 & 0 & 0 & 0\\
[16] & 0 & 0 & 0 & -1 & 1 & 0\\
[36] & 0 & 0 & 0 & 1 & 0 & -1\\
[46] & 0 & 0 & 0 & 0 & -1 & 1\\
\end{bNiceMatrix},
\]
\[
B = 
\begin{bmatrix}
1 & 0 & 0 & 0 & 0 & 0\\
0 & 0 & 1 & 0 & 0 & 0\\
0 & -1 & 0 & 1 & 0 & 0\\
0 & 0 & 0 & 0 & -1 & 0\\
0 & 0 & 0 & 0 & 0 & 1\\
\end{bmatrix},
\quad
Z=
\begin{bmatrix}
    1 & 0\\
    1 & 0\\
    1 & 0\\
    0 & 1\\
    0 & 1\\
    0 & 1\\
\end{bmatrix},
\quad \text{ and }
(BZ)^T=
\begin{bmatrix}
1 & 1 & -1 & 0 & 0\\
0 & 0 & 1 & -1 & 1\\
\end{bmatrix},
\]
and the matrix $(BZ)^T$ can be identified as the reduced incidence matrix of a graph $G$ as displayed in \cref{fig:grounded_graph}. Here $0$-vertex is the ground vertex, and the vertices \([1]\) and \([2]\) correspond to the polyhedra \([123]\) and \([134]\), respectively. It is obvious to see that $\dist(1, 0)=\dist(2, 0)=1$,
thus the Cheeger constant is 
\[\varphi_{1}^{\KK, \LL}=1/1=1.\]
\end{example}

When $\KK=\LL$,
\cref{prop:pers_cheeger_nbrch_clsd} and \cref{prop:pers_cheeger_nobrch_bdry} reduce to the (non-persistent) Cheeger constant.

\begin{example}\label{ex:cheeger_L}
Let $\LL$ be a $2$-dimensional simplicial complex as displayed in \cref{fig:cheeger_L}. 
    \begin{figure}[h]
\centering

\begin{subfigure}[b]{0.55\textwidth}
\centering
\begin{tikzpicture}
    \def\side{3.5}


    \filldraw[fill=yellow!40] (0, 0) -- (\side, 0) -- (1/2 * \side, \side * 0.8660254) -- cycle;
    \filldraw[fill=yellow!40] (\side, 0) -- (1/2 * \side, \side * 0.8660254) -- (\side * 3/2, \side * 1.732051 / 2) -- cycle;

    \draw (0,0) -- (30: 0.577*\side);
    \draw (1/2*\side, 0.866*\side) -- (1/2*\side, 0.289*\side);
    \draw (1/2*\side, 0.289*\side) -- (\side, 0);
    \draw (\side, 0.577*\side) -- (\side, 0);
    \draw (\side, 0.577*\side) -- (1/2*\side, 0.866 * \side);
    \draw (\side, 0.577*\side) -- (3/2*\side, 0.866 * \side);


    \node at (1/2 * \side,  \side * 0.8 + 0.5) {1};
    \node at (3/2 * \side,  \side * 0.8 + 0.5) {4};
    \node at (0, -0.3) {2};
    \node at (\side, -0.3) {3};
    \node at (0.4*\side, 0.35*\side) {5};
    \node at (0.92*\side, 0.5*\side) {6};

    \fill (0, 0) circle (2pt);  
    \fill (1/2 * \side, \side * 0.8660254) circle (2pt); 
    \fill (3/2 * \side, \side * 0.8660254) circle (2pt);
    \fill (\side, 0) circle (2pt);
    \fill (1/2*\side, 0.289*\side) circle (2pt);
    \fill (\side, 0.577*\side) circle (2pt);
\end{tikzpicture}
\caption{The simplicial complex $\LL$.}
\label{fig:cheeger_L}
\end{subfigure}
\hfill
\begin{subfigure}[b]{0.4\textwidth}
\centering
        \begin{tikzpicture}[scale=0.5,
    vertex/.style={circle, draw=black, font=\sffamily\Large\bfseries, thick, minimum size=8mm, fill=blue!10},
    ground/.style={circle, draw=black, 
    font=\sffamily\Large\bfseries, thick, minimum size=8mm, fill=gray!20},
    edge/.style={thick}
]

    \node[vertex] (1) at (-3, 2.5) {1};
    \node[vertex] (2) at (-3, 0)   {2};
    \node[vertex] (3) at (-1, 1.25) {3};
    
    \node[vertex] (4) at (1, 1.25) {4};
    \node[vertex] (5) at (3, 0)    {5};
    \node[vertex] (6) at (3, 2.5)  {6};

    \node[ground] (0) at (0, -2)   {0};

    \draw[edge] (0) to[bend left=95] node[midway, left] {} (1);
    \draw[edge] (2) -- node[midway, left] {} (1);
    \draw[edge] (1) -- node[midway, above right] {} (3);
    \draw[edge] (0) -- node[midway, below left] {} (2);
    \draw[edge] (3) -- node[midway, below right] {} (2);
    
    \draw[edge] (3) -- node[midway, above] {} (4);
    
    \draw[edge] (6) -- node[midway, above left] {} (4);
    \draw[edge] (4) -- node[midway, below left] {} (5);
    \draw[edge] (5) -- node[midway, below right] {} (0);
    \draw[edge] (5) -- node[midway, right] {} (6);
    \draw[edge] (0) to[bend right=95] node[midway, right] {} (6);
\end{tikzpicture}
\caption{The associated grounded dual graph.}
\label{fig:L_dual_graph}
\end{subfigure}
\caption{The simplicial complex $\LL$ and the associated dual graph in \cref{ex:cheeger_L}.}
\end{figure}
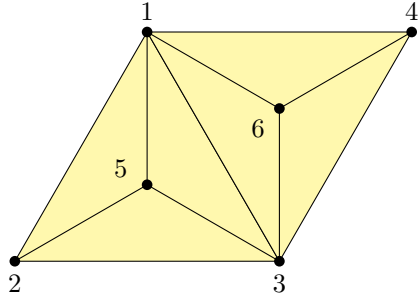
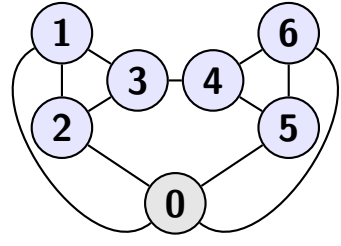
The associated grounded dual graph is displayed in \cref{fig:L_dual_graph}.
It is easy to see that the longest distance between a vertex to the grounded vertex $0$ is $2$,
thus according to \cref{prop:pers_cheeger_nobrch_bdry} 
 the Cheeger constant is $\varphi_{1}^\LL=\frac{1}{2}$.
\end{example}

\section{Persistent Cheeger constants for graph inclusions}
\label{sec.graph.comparisons}
In this section, we focus on persistent Cheeger constants for graphs. Throughout, let \(G=(V^G,E^G)\) be a connected graph and let \(H=(V^H,E^H)\) be a subgraph of \(G\). Note that \(\Delta_0^{H,G}\) depends only on the vertex set \(V^H\), and not on the edge set \(E^H\).

As observed in \cref{lem.kron}, the persistent Laplacian \(\Delta_0^{H,G}\) can be identified with the Laplacian \(\Delta_0^{\widetilde H}\) of a connected weighted graph \(\widetilde H=(V^H,E^{\widetilde H},w^{\widetilde H})\), where \(E^{\widetilde H}\) and \(w^{\widetilde H}\) denote its edge set and edge weights, respectively. Using this identification,
together with standard Cheeger constants for weighted graphs \cite{MR1395858,MR1421568},
M\'{e}moli et al. defined the following persistent Cheeger constants for graph pairs \cite[Definition~4.18]{pers_lap} and proved a corresponding persistent Cheeger inequality \cite[Theorem 4.21]{pers_lap}.

\begin{definition}\label{def:kron_cheeger}
For \(A\subseteq V^H\), let \(\chi_A\in\mathbb{R}^{V^H}\) denote its indicator vector, i.e.,
\[
(\chi_A)_v=
\begin{cases}
1, & v\in A,\\
0, & v\notin A.
\end{cases}
\]
The Kron reduction versions of the persistent Cheeger constants are defined by
\begin{equation}\label{eq:kron_cheeger_asymm}
h_{\operatorname{Kron}}^{H,G}
=
\min_{\substack{V^H=H_1\sqcup H_2\\ H_1,H_2\neq\varnothing}}
\frac{|V^H|\cdot(\chi_{H_1})^T\Delta_0^{H,G}\chi_{H_1}}{|H_1|\cdot |H_2|},
\end{equation}
and
\begin{equation}\label{eq:kron_cheeger_symm}
\varphi_{\operatorname{Kron}}^{H,G}
=
\min_{\substack{\varnothing\neq H_1\subseteq V^H\\ |H_1|\leq |V^H|/2}}
\frac{(\chi_{H_1})^T\Delta_0^{H,G}\chi_{H_1}}{|H_1|}.
\end{equation}
\end{definition}
While these persistent Cheeger constants do not extend to higher-degree simplicial complexes, it is natural to compare them with the persistent Cheeger constants that do. This comparison is the goal of the remainder of this section.

\subsection{Relating $h_{0}^{H, G}$ to $h_{\operatorname{Kron}}^{H, G}$}
In the case of graphs, for which the condition on a complete skeleton is trivially satisfied, 
the non-trivial Cheeger constant  \cref{eq:nt_cheeger_inq_graph} simplifies to the form
\begin{align}
\label{eq.complete.skeleton.graph}
h_{0}^{H, G}=
\min_{\substack{V^H=H_1\sqcup H_2\\
V^G=G_1\sqcup G_2\\\varnothing \neq H_i\subset G_i}}\frac{|V^H|\cdot |E(G_1, G_2)|}{|H_1|\cdot |H_2|}.
\end{align}

\begin{proposition}
\label{prop:ineq_cheger_kron_nonzero_1}
Let \(H\hookrightarrow G\) be an inclusion of graphs where $G$ is connected. Then
\[
h_{\operatorname{Kron}}^{H,G}\leq h_0^{H,G}.
\]
\end{proposition}

\begin{proof}
Let \(V^G=G_1\sqcup G_2\) and \(V^H=H_1\sqcup H_2\) be an admissible pair of partitions achieving \(h_0^{H,G}\), so that \(\varnothing\neq H_i\subseteq G_i\).

Since \(H_i\subset G_i\), the restriction of \(\chi_{G_1}\) to \(H\) is \(\chi_{H_1}\). By the variational characterization in \cref{lemma:schur_optimization},
\[
(\chi_{H_1})^T\Delta_0^{H,G}\chi_{H_1}\leq (\chi_{G_1})^T\Delta_0^G\chi_{G_1}.
\]
Finally,
\[
(\chi_{G_1})^T\Delta_0^G\chi_{G_1}=|E(G_1,G_2)|.
\]
Therefore
\[
(\chi_{H_1})^T\Delta_0^{H,G}\chi_{H_1}\leq |E(G_1,G_2)|.
\]
Multiplying by \(|H|/(|H_1|\cdot |H_2|)\) gives \(h_{\operatorname{Kron}}^{H,G}\leq h_0^{H,G}\).
\end{proof}

The Cheeger constants $h_0^{H, G}$ and $h_{\operatorname{Kron}}^{H, G}$ are in general not equivalent,
i.e.,
there is not a constant $c$ such that $ 
h_{0}^{H, G}  \leq c \cdot h_{\operatorname{Kron}}^{H, G}$. See \cref{ex:cheeger_equiv} for a counterexample.

\subsection{Relating $\varphi_0^{H, G}$ to $\varphi_{\operatorname{Kron}}^{H, G}$}
Next, we shall consider $\varphi_{\operatorname{Kron}}^{H, G}$ and the Cheeger constant $\varphi_0^{H, G}$ obtained via the $1$-Laplacian in \cref{def:nonzero_cheeger}. We obtain an inequality analogous to that in the previous section, but the proof is more involved.

\begin{proposition}\label{prop:ineq_cheger_kron_nonzero}
Let \(H\hookrightarrow G\) be an inclusion of finite graphs, and assume that \(G\) is connected. Then
\[
\varphi_{\operatorname{Kron}}^{H,G}\leq \varphi_0^{H,G}.
\]
\end{proposition}

\begin{proof}
Since \(\Delta_0^{H,G}\) can be regarded as the Laplacian of a weighted graph \(\widetilde H\) by \cref{lem.kron}, and since \(G\) is connected, we have:
\[
\ker(\Delta_0^{H,G})=\ker(PB^T)=\operatorname{span}(\one).
\]
Hence the \(1\)-orthogonality condition in \cref{def:nonzero_cheeger} reduces to \(x\perp^1\one\). We use the standard characterization of this condition \cite{MR3433638,MR3600065,MR3403773}: it is satisfied precisely when at most half of the entries of \(x\) are positive and at most half are negative. Thus, if \(|S|\leq |H|/2\), then \(\chi_S\perp^1\one\).

Let \(H^c:=V^G\setminus V^H\). Order the vertices of \(G\) so that the vertices in \(H\) come first, followed by the vertices in \(H^c\). With respect to this ordering, write the incidence matrix as
\[
I=\begin{bmatrix}B\\D\end{bmatrix},
\]
where the rows of \(B\) and \(D\) are indexed by \(H\) and \(H^c\), respectively.

Let
\[
x^*\in \argmin_{x\perp^1\one}\frac{\|PB^Tx\|_1}{\|x\|_1},\qquad f:=PB^Tx^*.
\]
Recall that \(P\) is the orthogonal projection from \(C_1^G\) onto \(C_1^{H,G}=\ker D\). Hence \(C_1^G=\ker D\oplus\operatorname{im}D^T\), so there exists \(y^*\in C_0^{H^c}\) such that
\[
B^Tx^*=f-D^Ty^*.
\]
Set \(z=\begin{bmatrix}x^*\\y^*\end{bmatrix}\). Then \(I^Tz=B^Tx^*+D^Ty^*=f=PB^Tx^*\). Therefore
\[
\|PB^Tx^*\|_1=\|I^Tz\|_1.
\]

For \(t\in\mathbb{R}\), define \(S_t:=\{u\in V^G:z_u>t\}\). By the ``layer-cake formula'' (cf. the proof in \cref{sec.app.proof.dirichlet}), we have
\[
\|I^Tz\|_1=\int_{-\infty}^{\infty}|E(S_t,S_t^c)|\,dt.
\]
Moreover, since \(\chi_{S_t}\) is an indicator vector and $G$ is unweighted,
\[
|E(S_t,S_t^c)|=\chi_{S_t}^T\Delta_0^G\chi_{S_t}.
\]
Let \(S_t^H:=S_t\cap H\). Since \(\chi_{S_t}\) is an extension of \(\chi_{S_t^H}\) from \(H\) to \(G\), \cref{lemma:schur_optimization} gives
\[
(\chi_{S_t^H})^T\Delta_0^{H,G}\chi_{S_t^H}\leq \chi_{S_t}^T\Delta_0^G\chi_{S_t}=|E(S_t,S_t^c)|.
\]
Thus
\[
\frac{\|PB^Tx^*\|_1}{\|x^*\|_1}\geq \frac{\int_{-\infty}^{\infty}(\chi_{S_t^H})^T\Delta_0^{H,G}\chi_{S_t^H}\,dt}{\|x^*\|_1}.
\]

The \(1\)-orthogonality condition \(x^*\perp^1\one\) implies that at most half of the entries of \(x^*\) are positive and at most half are negative. Hence, for \(t>0\), the set \(S_t^H\) has size at most \(|H|/2\), and for \(t<0\), the set \(H\setminus S_t^H\) has size at most \(|H|/2\). Also,
\[
\|x^*\|_1=\int_0^\infty |S_t^H|\,dt+\int_{-\infty}^0 |H\setminus S_t^H|\,dt.
\]
For \(t>0\), the defining inequality for \(\varphi_{\operatorname{Kron}}^{H,G}\) gives
\[
(\chi_{S_t^H})^T\Delta_0^{H,G}\chi_{S_t^H}\geq \varphi_{\operatorname{Kron}}^{H,G}|S_t^H|,
\]
with both sides equal to \(0\) if \(S_t^H=\varnothing\). For \(t<0\), the set \(H\setminus S_t^H\) is admissible whenever it is nonempty. Since \(\chi_{H\setminus S_t^H}=\one-\chi_{S_t^H}\) and \(\Delta_0^{H,G}\one=0\), we have
\[
(\chi_{H\setminus S_t^H})^T\Delta_0^{H,G}\chi_{H\setminus S_t^H}=(\chi_{S_t^H})^T\Delta_0^{H,G}\chi_{S_t^H}.
\]
Therefore
\[
(\chi_{S_t^H})^T\Delta_0^{H,G}\chi_{S_t^H}\geq \varphi_{\operatorname{Kron}}^{H,G}|H\setminus S_t^H|,
\]
again trivially if \(H\setminus S_t^H=\varnothing\). Integrating over \(t>0\) and \(t<0\), respectively, gives
\[
\int_{-\infty}^{\infty}(\chi_{S_t^H})^T\Delta_0^{H,G}\chi_{S_t^H}\,dt\geq \varphi_{\operatorname{Kron}}^{H,G}\|x^*\|_1.
\]
Thus
\[
\varphi_0^{H,G}=\frac{\|PB^Tx^*\|_1}{\|x^*\|_1}\geq \varphi_{\operatorname{Kron}}^{H,G}.
\]
\end{proof}
The Cheeger constants $\varphi_0^{H, G}$ and $\varphi_{\operatorname{Kron}}^{H, G}$ are not equivalent in general,
i.e., there is not a number $c$ such that $\varphi_{\operatorname{Kron}}^{H, G} \leq \varphi_0^{H, G}\leq c\cdot \varphi_{\operatorname{Kron}}^{H, G}$. See \cref{ex:cheeger_equiv} for a counterexample.
\subsection{Lower-bounding the Kron Cheeger constants}
In this section, we show that there is no equivalence between the Cheeger constants unless additional structure is imposed.

\begin{proposition}
There does not exist a constant $c$ such that either of the following holds for all inclusions of graphs $H\hookrightarrow G$:
\begin{enumerate}
    \item $h_{0}^{H, G}  \leq c \cdot h_{\operatorname{Kron}}^{H, G}$;
    \item $\varphi_0^{H, G}\leq c\cdot \varphi_{\operatorname{Kron}}^{H, G}$.
\end{enumerate}
    \label{ex:cheeger_equiv}
\end{proposition}
\begin{proof}
Consider a path graph $G$ on $n$ vertices and let $H=\{1, n\}$ be the two end-points; see \cref{fig:path_graph}.
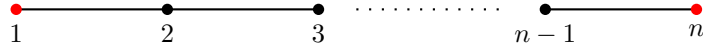
\begin{figure}[H]
\centering
\begin{tikzpicture}[
    dot/.style={circle, fill=black, inner sep=1.5pt},
    interval/.style={thick}
]

    \draw[interval] (0,0) -- (4,0);
    \draw[interval, loosely dotted] (4.5,0) -- (6.5,0); 
    \draw[interval] (7,0) -- (9,0);

    \node[dot, fill=red, label=below:{$1$}] at (0,0) {};
    \node[dot, label=below:{$2$}] at (2,0) {};
    \node[dot, label=below:{$3$}] at (4,0) {};
    \node[dot, label=below:{$n-1$}] at (7,0) {};
    \node[dot, fill=red, label=below:{$n$}] at (9,0) {};


\end{tikzpicture}
\caption{$G$ is a path graph on $n$ vertices, $H=\{1, n\}$.}
\label{fig:path_graph}
\end{figure}

First, the path graph $G$ is non-branching, and therefore the non-zero Cheeger constant 
\[\varphi_{0}^{H, G}=1\] by \cref{prop:pers_cheeger_nobrch_bdry}.

To compute $h_0^{H, G}$ (\cref{eq.complete.skeleton.graph}), one observes that since $H$ contains only two vertices $1$ and $n$, there is only one way to separate $H$ into two components. Thus, one simply needs
to partition the vertices of $G$ such that $|E(V_1^G, V_2^G)|$ is minimal and $1\in V_1^G$ and $n\in V_2^G$. This minimum is obviously $1$, and thus \[h_0^{H, G}=\frac{2\cdot 1}{1\cdot 1}=2.\]

For the two Kron Cheeger constants, we first observe that 
\[\Delta_0^{H, G}=\frac{1}{n-1}\begin{bmatrix}
    1 & -1\\
    -1 & 1
\end{bmatrix}.\]
Thus 
\[\varphi_{\operatorname{Kron}}^{H, G}=\frac{1}{n-1} \qquad \qquad h_{\operatorname{Kron}}^{H, G}=\frac{2}{n-1}\]
and when $n\to \infty$,
$\varphi_{\operatorname{Kron}}^{H, G}\to 0$ and $h_{\operatorname{Kron}}^{H, G}\to 0$.

\end{proof}
However,
if $H$ and $H^c$ have enough connections with each other,
and the vertices in $H^c$ are disconnected from each other,
the persistent Cheeger constants
$\varphi_0^{H, G}$ and $\varphi_{\operatorname{Kron}}^{H, G}$ are equivalent.
We first recall that if $\Delta_0^G$ is the Laplacian on a graph $G=(V, E)$,
and $\chi_S$ is the indicator vector associated with a subset $S\subset V$,
by a direct computation (\cref{app.proof.of.eq}),
we have the equality
\begin{equation}\label{eq:L_chi_1_norm}
\|\Delta_0^G\chi_S\|_1=2\chi_S^T\Delta_0^G\chi_S.
\end{equation}

\begin{proposition}\label{prop:kron_nonzero_cheeger_equiv}
Suppose that \(H^c:=V^G\setminus V^H\) is an independent set and that every vertex of \(H\) has exactly one neighbor in \(H^c\). Then
\[
\varphi_{\operatorname{Kron}}^{H,G}
\leq
\varphi_0^{H,G}
\leq
2\varphi_{\operatorname{Kron}}^{H,G}.
\]
\end{proposition}

\begin{proof}
The first part of the inequality is proved in \cref{prop:ineq_cheger_kron_nonzero},
we only need to prove the second part.
Since \(H^c\) is independent, every edge of \(G\) either lies in \(H\) or joins a vertex of \(H\) to a vertex of \(H^c\). Order the edge set as \(E^G=E^H\sqcup E_c\), where \(E_c:=E(H,H^c)\). Since every vertex of \(H\) has exactly one neighbor in \(H^c\), we may orient the edges in \(E_c\) from \(H\) to \(H^c\) and write the incidence matrix of $G$ (the matrix representation of $\partial_1^G$) as
\[
I=\begin{bmatrix}B\\D\end{bmatrix}=\begin{bmatrix}B_H&\id\\0&D_c\end{bmatrix},
\]
where \(B_H\) is the incidence matrix of the subgraph induced by \(H\), and \(D_c\) records the endpoints of the crossing edges in \(H^c\). Since \(D=[0\;\;D_c]\), we have \(\ker D=C_1^H\oplus\ker D_c\). Thus the orthogonal projection \(P:C_1^G\twoheadrightarrow\ker D \subseteq C_1^G\) splits as
\[
P=\begin{bmatrix}\id_{C_1^H}&0\\0&P_c\end{bmatrix},
\]
where \(P_c\) is the orthogonal projection from the ``crossing-edge'' subspace of $C_1^G$ onto \(\ker D_c\).  Hence, for \(x\in C_0^H\),
\[
PB^Tx=\begin{bmatrix}B_H^Tx\\P_cx\end{bmatrix}.
\]
Therefore
\[
\varphi_0^{H,G}=\min_{x\perp^1\one}\frac{\|PB^Tx\|_1}{\|x\|_1}=\min_{x\perp^1\one}\frac{\|B_H^Tx\|_1+\|P_cx\|_1}{\|x\|_1}.
\]

Furthermore, by \cref{eq:P_pers_up_lap},
\[
\Delta_0^{H,G}=(PB^T)^T(PB^T)=B_HB_H^T+P_cP_c^T=B_HB_H^T+P_c,
\]
since \(P_c\) is an orthogonal projection. Hence
\[
\varphi_{\operatorname{Kron}}^{H,G}
=
\min_{\substack{\varnothing\neq S\subseteq V^H\\ |S|\leq |V^H|/2}}
\frac{\chi_S^T(B_HB_H^T+P_c)\chi_S}{|S|}.
\]

Now let \(S\subseteq V^H\) be nonempty with \(|S|\leq |V^H|/2\). The entries of \(B_H^T\chi_S\) lie in \(\{-1,0,1\}\), so
\[
\|B_H^T\chi_S\|_1=\|B_H^T\chi_S\|_2^2=\chi_S^TB_HB_H^T\chi_S.
\]
The submatrix \(\begin{bmatrix}\id\\D_c\end{bmatrix}\) is the incidence matrix of the subgraph \(G_c\) obtained from \(G\) by deleting the edges of \(H\). For this subgraph, the corresponding persistent Laplacian is
\[
\Delta_0^{H,G_c}=(P_c\id)(P_c\id)^T=P_c.
\]
Since \(\Delta_0^{H,G_c}\) is the graph Laplacian of the corresponding Kron reduction, \cref{eq:L_chi_1_norm} gives
\[
\|P_c\chi_S\|_1=2\chi_S^TP_c\chi_S.
\]
Therefore
\[
\varphi_0^{H,G}\leq
\min_{\substack{\varnothing\neq S\subseteq V^H\\ |S|\leq |V^H|/2}}
\frac{\|B_H^T\chi_S\|_1+\|P_c\chi_S\|_1}{|S|}
\leq
\min_{\substack{\varnothing\neq S\subseteq V^H\\ |S|\leq |V^H|/2}}
\frac{2\chi_S^T(B_HB_H^T+P_c)\chi_S}{|S|}
=
2\varphi_{\operatorname{Kron}}^{H,G}.
\]
\end{proof}

\begin{example}
\label{ex.lower.bound.cheeger.graph}
Consider the graphs $H\subset G$ from \cite{pers_lap}, where $H$ consists of the three outermost vertices, and $\widetilde{H}$ is the associated Kron reduction; see \cref{fig.lowerbound.cheeger.graph.example}:
    \begin{figure}
    \centering
    \begin{subfigure}[b]{0.45\textwidth}
        \centering
        \begin{tikzpicture}[scale=1.5]
            \node[circle, fill=blue, minimum size=14pt, inner sep=0pt] (C) at (0,0) {};
            \node[circle, fill=red, minimum size=14pt, inner sep=0pt] (TL) at (150:1.5) {};
            \node[circle, fill=red, minimum size=14pt, inner sep=0pt] (TR) at (30:1.5) {};
            \node[circle, fill=red, minimum size=14pt, inner sep=0pt] (B) at (270:1.5) {};

            \draw[blue, line width=1.5pt] (C) -- node[below left=1pt] {\Large 1} (TL);
            \draw[blue, line width=1.5pt] (C) -- node[below right=1pt] {\Large 1} (TR);
            \draw[blue, line width=1.5pt] (C) -- node[right=2pt] {\Large 1} (B);

            \node at (0, 1.4) {\large $G$};
        \end{tikzpicture}
    \end{subfigure}
    \hfill
    \begin{subfigure}[b]{0.45\textwidth}
        \centering
        \begin{tikzpicture}[scale=1.5]
            \node[circle, fill=red, minimum size=14pt, inner sep=0pt] (TL) at (150:1.5) {};
            \node[circle, fill=red, minimum size=14pt, inner sep=0pt] (TR) at (30:1.5) {};
            \node[circle, fill=red, minimum size=14pt, inner sep=0pt] (B) at (270:1.5) {};

            \draw[red, line width=1.5pt] (TL) -- node[below=2pt] {\LARGE $\frac{1}{3}$} (TR);
            \draw[red, line width=1.5pt] (TL) -- node[left=4pt] {\LARGE $\frac{1}{3}$} (B);
            \draw[red, line width=1.5pt] (TR) -- node[right=4pt] {\LARGE $\frac{1}{3}$} (B);

            \node at (0, 1.4) {\large $\Tilde{H}$};
        \end{tikzpicture}
    \end{subfigure}
    \caption{The graphs from \cref{ex.lower.bound.cheeger.graph}.}
\label{fig.lowerbound.cheeger.graph.example}
\end{figure}
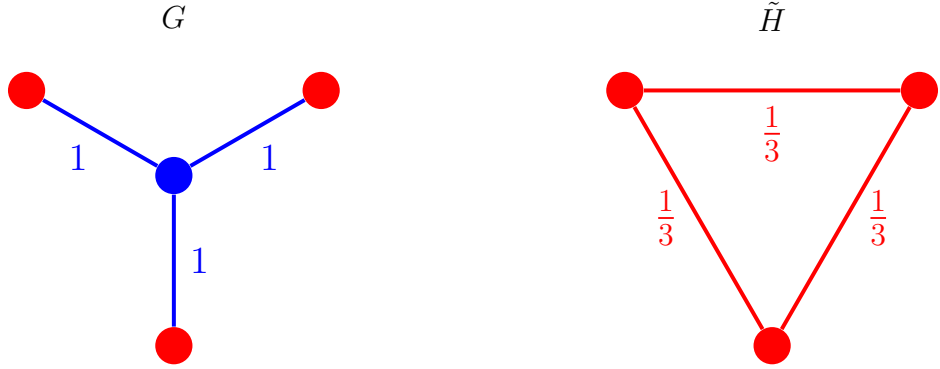
Since there are no edges internal to $H$, it follows from the proof of \cref{prop:kron_nonzero_cheeger_equiv} (with $B_H$ empty), that $\Delta_0^{H, G}$ equals the projection matrix $P\colon C_1^G\twoheadrightarrow C_1^{G,H}\subseteq C_1^G$,
\[P=\Delta_0^{H, G}=\Delta_0^{\widetilde{H}} = \frac{1}{3}\begin{bmatrix}
    2 & -1 & -1\\
    -1 & 2 & -1\\
    -1 & -1 & 2
\end{bmatrix}.\]
We claim that $\varphi_0^{H, G}=1$ and $\varphi_{\operatorname{Kron}}^{H, G}=\frac{2}{3}$,
and hence \cref{prop:kron_nonzero_cheeger_equiv} holds. 

To compute \(\varphi_0^{H,G}\), consider
\[
\varphi_0^{H,G}=\min_{x\perp^1\one}\frac{\|Px\|_1}{\|x\|_1}.
\]
By the characterization of \(1\)-orthogonality recalled in the proof of \cref{prop:ineq_cheger_kron_nonzero}, at most half of the entries of \(x\) are positive and at most half are negative. Hence, after permuting the vertices and using the symmetry of \(\widetilde H\), we may take \(x=(a,0,b)^T\) with \(a\leq 0\leq b\). Then
\[
\frac{\|Px\|_1}{\|x\|_1}
=
\frac{(b-2a)+|a+b|+(2b-a)}{3(b-a)}
=
1+\frac{|a+b|}{3(b-a)}.
\]
This is minimized when \(a+b=0\), and hence \(\varphi_0^{H,G}=1\).

On the other hand, since $|H|=3$,
we must have that $|S|=1$.
Thus,
\[\varphi_{\operatorname{Kron}}^{H, G}=
\frac{1}{3}\begin{bmatrix}
    1 & 0 & 0
\end{bmatrix}
\begin{bmatrix}
    2 & -1 & -1\\
    -1 & 2 & -1\\
    -1 & -1 & 2
\end{bmatrix}
\begin{bmatrix}
    1\\0\\0
\end{bmatrix}=
\frac{2}{3}.
\]
\end{example}

\section{Conclusion and outlook}
In this paper, we have shown that several Cheeger constants for
simplicial complexes admit natural extensions to the persistent setting.
This raises the question of whether there is a more unified approach to
this problem. For an inclusion \(\KK\hookrightarrow \LL\), there is no
Kron reduction for which \(\Delta_{q,\mathrm{up}}^{\KK,\LL}\) coincides
with the (non-persistent) Laplacian of a simplicial complex. It is therefore
natural to ask whether such a replacement object can be found in a
broader category, for instance among \(\Delta\)-complexes or
CW-complexes. Cheeger inequalities proved in that setting would then
translate directly to the persistent setting studied in this paper.

Furthermore, we introduced the persistent \(p\)-Laplacian. Given the
prominence of \(p\)-Laplacians for graphs, we are currently exploring
their role in the persistent setting of simplicial complexes.

Lastly, it might be the case that the upper bounds of \cref{thm.pers.cheeger} also apply in the non-orientable setting. We leave that to future research.

\appendix
\section{Facts from linear algebra}\label{sec.adjprop}
These facts are well-known and can be found in many textbooks on linear algebra or spectral graph theory. 
The following \emph{Courant--Fischer theorem} is well-known. 
\begin{theorem}
\label{thm.rayleigh}
Let $A$ be a symmetric positive semidefinite $n\times n$ matrix with respect to
a given inner product $\langle \cdot,\cdot\rangle$ on $\mathbb{R}^n$.  Denote
its eigenvalues by
\[
0 = \lambda_1 \le \lambda_2 \le \cdots \le \lambda_n,
\]
and assume $A$ has at least one positive eigenvalue.  
For $x \neq 0$, the Rayleigh quotient of $A$ is
\[
R_A(x) := \frac{\langle Ax, x\rangle}{\langle x, x\rangle}.
\]
Then the smallest non-zero eigenvalue $\lambda_2$ satisfies
\[
\lambda_2
= \min_{\substack{x \perp \ker A \\ x \neq 0}} R_A(x)
= \min_{\substack{x \perp \ker A \\ x \neq 0}}
   \frac{\langle Ax, x\rangle}{\langle x, x\rangle}.
\]
A non-zero vector $x\perp \ker A$ attains this minimum if and only if it is an
eigenvector of $A$ with eigenvalue $\lambda_2$.
\end{theorem}
Similar characterizations can be given for the higher eigenvalues, see, e.g., \cite{MR3077874}.

\begin{proposition}
    \label{prop:laplaceusefulLA}
\leavevmode
\begin{enumerate}
    \item Let $A \in \mathbb{R}^{m \times n}$ and $B \in \mathbb{R}^{n \times m}$. Then $AB \in \mathbb{R}^{m \times m}$ and $BA \in \mathbb{R}^{n \times n}$ have the same nonzero (complex) eigenvalues, counted with algebraic multiplicities. Moreover:
    \begin{itemize}
        \item If $v \in \mathbb{R}^m$ satisfies $AB v = \lambda v$ with $\lambda \neq 0$, then $Bv \in \mathbb{R}^n$ is a nonzero eigenvector of $BA$ with eigenvalue $\lambda$.
        \item Conversely, if $w \in \mathbb{R}^n$ satisfies $BA w = \lambda w$ with $\lambda \neq 0$, then $Aw \in \mathbb{R}^m$ is a nonzero eigenvector of $AB$ with eigenvalue $\lambda$.
    \end{itemize}

    \item Let $C, D \in \mathbb{R}^{n \times n}$ be real symmetric matrices satisfying $CD = DC$. Then $C$ and $D$ are simultaneously diagonalizable; that is, there exists an orthonormal basis of common eigenvectors $v$ such that
    \[
    Cv = \alpha v, \qquad Dv = \beta v.
    \]
\end{enumerate}
\end{proposition}

\paragraph{Properties of Adjoints}
Let $V$ be a finite-dimensional inner product space over $ \mathbb{R}$, and fix a basis $\{b_1, \dots, b_n\}$ for $V$. The \emph{Gram matrix} $G_V$ associated to this basis is the $n \times n$ matrix with entries
\[
(G_V)_{ij} = \langle b_i, b_j \rangle_V.
\]
In particular, for column vectors $v$ and $v'$ expressed in the basis above, we have 
\[\langle v, v'\rangle_V = v^tG_Vv'.\]

\begin{proposition}
    \label{prop:matrixformadjoint}
Let $f \colon V \to W$ be a linear map between finite-dimensional real inner product spaces, with Gram matrices $G_V$ and $G_W$ (in chosen bases of $V$ and $W$, respectively). If $[f]$ is the matrix of $f$ in these bases, then the matrix of the adjoint $f^* \colon W \to V$ in the same bases is
\[
[f^*] = G_V^{-1} [f]^T G_W.
\]
\end{proposition}

\section{Totally unimodular matrices and linear programming}\label{sec:TU}
\begin{definition}
We say a matrix $M$ is totally unimodular (TU) if the determinant of every squared submatrix of $M$ is equal to $1$, $-1$ or $0$.
\end{definition}
It can be proved that the incidence matrix with respect to a directed graph is totally unimodular,
This fact was first stated by Poincar\'{e} \cite{poincare};
see \cite[page 274]{schrijver1986theory}.

\begin{lemma}\label{lemma: tot_unimod}
Let $e_i, e_j\in \mathbb{R}^{|V|}$ be two standard vectors,
and define the operator $I^+: \im(I)\to \mathbb{R}^{|E|}$ as:
\[
I^+: w\mapsto \argmin_x\{ \|x\|_1 : Ix=w \}.
\]
Note that this minimizer is characterized by the orthogonality condition $\|x+y\|_1\geq \|x\|_1$ for any $y\in \ker (I)$.
We then have that
\[
I^+(e_i-e_j)\in \{-1, 0, 1\}^{|E|},
\]
and $\|I^+(e_i-e_j)\|_1$ is equal to the length of the shortest path between $i, j\in V$. 
\end{lemma}

\begin{proof}
We observe that computing $I^+(e_i-e_j)$ is equivalent to solving the optimization problem:
\[
\min_x \|x\|_1 \quad \text{subject to} \quad I x = e_i-e_j.
\]
If we depart $x = x^+ - x^-$ with $x^+, x^- \ge 0$, 
this optimization problem is equivalent to the following Linear Program (LP):
\begin{equation}\label{eq:LP}
\min \mathbf{1}^T x^+ + \mathbf{1}^T x^- \quad \text{subject to} \quad 
\begin{bmatrix}
I & -I\\  
\end{bmatrix}
\begin{bmatrix}
x^+ \\ x^- 
\end{bmatrix} = 
e_i-e_j.
\end{equation}
Since $I$ is TU (\cite{schrijver1986theory}),
it is straightforward to verify that the block matrix 
$\begin{bmatrix}
    I & -I
\end{bmatrix}$ is TU as well.
According to Hoffman-Kruskal theorem \cite[Corollary 19.2a]{schrijver1986theory},
the optimal solution to \cref{eq:LP} is integral,
i.e.,
$
\begin{bmatrix} x^+ \\ x^- 
\end{bmatrix} 
\in \mathbb{Z}^{2\times|E|}.
$
Thus,
the minimizer $x\in \mathbb{Z}^{|E|}$ is an integral flow.

We now prove that $x\in \{-1, 0, 1\}^{|E|}$.
According to the Flow Decomposition Theorem \cite[Theorem 3.5]{MR1205775}, the integral flow $x$ satisfying $Ix = e_i - e_j$ can be decomposed into exactly one simple path $p$ from $i$ to $j$, 
plus a set of simple closed cycles $C$:
\[
x = p + \sum_{c \in C} c,
\]
where $p\in \{-1, 0, 1\}^{|E|}$ and the cycles satisfy $\sum\limits_{c\in C} c\in \ker (I)$.
Since $x$ was chosen to minimize the $\ell^1$-norm,
the set of cycles $C$ must be empty. 
If $C$ were non-empty,
removing the cycles would yield the flow $p$,
which still satisfies $Ip = e_i - e_j$ but has a strictly smaller norm ($\|p\|_1 < \|x\|_1$), 
contradicting the minimality of $x$. 
Therefore, $C = \emptyset$ and $x = p \in \{-1, 0, 1\}^{|E|}$. 
Consequently,
$I^+(e_i-e_j)=p$,
and its $\ell^1$-norm corresponds exactly to the number of edges in the path $p$. Since it minimizes the $\ell^1$-norm, $p$ must be the shortest path between $i$ and $j$.
\end{proof}

If we pick a vertex $g\in V$ as a designated ground vertex,
and denote by $\tilde{I}$ the reduced incidence matrix associated with this grounded graph
(i.e.,
$\tilde{I}$ is obtained by deleting the $g$-th row from $I$),
the solution of the optimization problem with respect to $\tilde{I}^+$ is also restricted to $\{-1, 0, 1\}$.
We summarize the result as follows:
 
\begin{corollary}\label{cor:tot_unimod_ground}
Let $e_i\in \mathbb{R}^{|V|-1}$ be a standard vector for $i\in V\setminus \{g\}$.
Let $\tilde{I}^+: \im(I)\to \mathbb{R}^{|E|}$ be the map defined as
\[
\tilde{I}^+: w\mapsto \argmin_x\{ \|x\|_1 : \tilde{I}x=w \}.
\]
Then
\[
\tilde{I}^+(e_i)\in \{-1, 0, 1\}^{|E|},
\]
and $\|\tilde{I}^+(e_i)\|_1$ is equal to the length of the shortest path between $i$ and the ground vertex $g$.
\end{corollary}

\begin{proof}
Notice that the equation $\tilde{I}x=e_i$ implies $Ix=e_i-e_g$,
where $e_g$ denotes the standard basis vector in  $\mathbb{R}^{|V|}$ corresponding to the ground vertex $g$.
This is because the sum of the rows of the full incidence matrix $I$ is always zero. 
Hence,
the grounded problem $\tilde{I} x = e_i$ is mathematically identical to the ungrounded problem $I x = e_i - e_g$.
It naturally follows from \cref{lemma: tot_unimod} that $\tilde{I}^+(e_i)\in \{-1, 0, 1\}^{|E|}$,
and $\tilde{I}^+ e_i=I^+(e_i-e_g)$.
\end{proof}

\section{Missing details}
\subsection{Proof of \cref{thm:weighted_graph_cheeger_dirichlet} }
\label{sec.app.proof.dirichlet}
\begin{proof}
For \(x\in\mathbb R^m\), write \(g=W^{-1/2}x\). By the definitions of \(w_{ij}\) and \(b_i\),
\[
x^TMx=g^TXg=\sum_{i<j}w_{ij}(g_i-g_j)^2+\sum_i b_i g_i^2,\qquad \|x\|^2=\sum_i\mu_i g_i^2.
\]
In particular, if a $b_i>0$, then the matrix $M$ is positive definite and $\ker M =0$. Here we used that the interior graph associated with \(X\) is connected.
\paragraph{Upper bounds.}
We first prove the upper bounds. Suppose first that at least one row sum of \(X\) is positive, and let \(S\) achieve \(h_{\mathrm{vol}}^{\mathrm{Bdy}}\). Taking \(g=\chi_S\) (the indicator vector on $S$) and \(x=W^{1/2}g\), we get
\[
\frac{x^TMx}{\|x\|^2}
=
\frac{\operatorname{Cut}(S)+\operatorname{Bdy}(S)}{\operatorname{Vol}(S)}
=
h_{\mathrm{vol}}^{\mathrm{Bdy}},
\]
so \(\lambda_{\min}(M)\leq h_{\mathrm{vol}}^{\mathrm{Bdy}}\) by \cref{thm.rayleigh} noting that $\ker M =0$.

Now suppose that all row sums of \(X\) vanish. Let \(S\) achieve \(h_{\mathrm{vol}}\), with \(\operatorname{Vol}(S)\leq V_{\mathrm{tot}}/2\). Set \(\alpha=\operatorname{Vol}(S)/V_{\mathrm{tot}}\), and define
\[
g_i=
\begin{cases}
1-\alpha, & i\in S,\\
-\alpha, & i\notin S.
\end{cases}
\]
Then \(x=W^{1/2}g\) is orthogonal to \(\ker M = W^{1/2}\one\). Since \(b_i=0\) for all \(i\), we have
\[
x^TMx=\operatorname{Cut}(S),
\]
while
\[
\|x\|^2=\operatorname{Vol}(S)(1-\alpha)^2+(V_{\mathrm{tot}}-\operatorname{Vol}(S))\alpha^2
=\frac{\operatorname{Vol}(S)(V_{\mathrm{tot}}-\operatorname{Vol}(S))}{V_{\mathrm{tot}}}.
\]
Since \(\operatorname{Vol}(S)\leq V_{\mathrm{tot}}/2\), we have \(\|x\|^2\geq \operatorname{Vol}(S)/2\). Hence
\[
\lambda_1(M)\leq \frac{x^TMx}{\|x\|^2}\leq 2\frac{\operatorname{Cut}(S)}{\operatorname{Vol}(S)}=2h_{\mathrm{vol}},
\]
by \cref{thm.rayleigh}.
\paragraph{Lower bounds.}
We now prove the lower bounds.

\emph{(1) Dirichlet case.}
Assume that at least one row sum of \(X\) is positive, i.e. $M$ is positive definite and $\ker M =0$. 

Let \(x\) be an eigenvector for \(\lambda_{\min}(M)\), let \(g=W^{-1/2}x\), and set \(h_i=g_i^2\). Introduce an auxiliary ``ground'' vertex, indexed by \(0\), with fixed value \(\widetilde g_0=0\). For \(i\geq 1\), set \(\widetilde g_i=g_i\) and \(\widetilde h_i=\widetilde g_i^2\), while \(\widetilde h_0=0\). Define augmented weights by \(\widetilde w_{ij}=w_{ij}\) for \(i,j\geq 1\), and \(\widetilde w_{i0}=b_i\). Then
\[
x^TMx=\sum_{i<j}w_{ij}(g_i-g_j)^2+\sum_i b_i(g_i-\widetilde g_0)^2=\sum_{0\leq i<j\leq m}\widetilde w_{ij}(\widetilde g_i-\widetilde g_j)^2.
\]
Moreover,
\[
|\widetilde h_i-\widetilde h_j|=|\widetilde g_i^2-\widetilde g_j^2|=|\widetilde g_i-\widetilde g_j|\,|\widetilde g_i+\widetilde g_j|.
\]
Hence, applying Cauchy--Schwarz to \(a_{ij}:=\sqrt{\widetilde w_{ij}}|\widetilde g_i-\widetilde g_j|\) and \(c_{ij}:=\sqrt{\widetilde w_{ij}}|\widetilde g_i+\widetilde g_j|\), we obtain
\[
\left(\sum_{0\leq i<j\leq m}\widetilde w_{ij}|\widetilde h_i-\widetilde h_j|\right)^2
\leq
\left(\sum_{0\leq i<j\leq m}\widetilde w_{ij}(\widetilde g_i-\widetilde g_j)^2\right)
\left(\sum_{0\leq i<j\leq m}\widetilde w_{ij}(\widetilde g_i+\widetilde g_j)^2\right),
\]
and therefore
\[
x^TMx\geq
\frac{\left(\sum_{0\leq i<j\leq m}\widetilde w_{ij}|\widetilde h_i-\widetilde h_j|\right)^2}
{\sum_{0\leq i<j\leq m}\widetilde w_{ij}(\widetilde g_i+\widetilde g_j)^2}.
\]
The denominator is bounded by
\[
\sum_{0\leq i<j\leq m}\widetilde w_{ij}(\widetilde g_i+\widetilde g_j)^2
\leq 2\sum_{0\leq i<j\leq m}\widetilde w_{ij}(\widetilde g_i^2+\widetilde g_j^2)
=
2\sum_{0\leq i\leq m}\left(\sum_{j\neq i,0}w_{ij}+b_i\right)g_i^2
=
2\sum_{0\leq i\leq m}X_{ii}g_i^2
\leq 2\Delta\|x\|^2.
\]
For \(t>0\), let \(S_t:=\{i:h_i>t\}\). By the ``layer-cake formula'',
\[
\sum_{0\leq i<j\leq m}\widetilde w_{ij}|\widetilde h_i-\widetilde h_j|
=
\int_0^\infty
\sum_{0\leq i<j\leq m}\widetilde w_{ij}
|\chi_{S_t}(i)-\chi_{S_t}(j)|\,dt.
\]
where $\chi_A$ denote the indicator vector on $A$. 

Since the vertex \(0\) is never in \(S_t\), the integrand is \(\operatorname{Cut}(S_t)+\operatorname{Bdy}(S_t)\). Thus
\[
\sum_{0\leq i<j\leq m}\widetilde w_{ij}|\widetilde h_i-\widetilde h_j|
=
\int_0^\infty\left(\operatorname{Cut}(S_t)+\operatorname{Bdy}(S_t)\right)\,dt
\geq
h_{\mathrm{vol}}^{\mathrm{Bdy}}\int_0^\infty\operatorname{Vol}(S_t)\,dt.
\]
Finally,
\[
\int_0^\infty\operatorname{Vol}(S_t)\,dt=\sum_i\mu_i h_i=\sum_i\mu_i g_i^2=\|x\|^2.
\]
Substituting the numerator and denominator estimates gives
\[
x^TMx\geq
\frac{\left(h_{\mathrm{vol}}^{\mathrm{Bdy}}\|x\|^2\right)^2}{2\Delta\|x\|^2}
=
\frac{(h_{\mathrm{vol}}^{\mathrm{Bdy}})^2}{2\Delta}\|x\|^2.
\]
Dividing by \(\|x\|^2\) gives \(\lambda_{\min}(M)\geq (h_{\mathrm{vol}}^{\mathrm{Bdy}})^2/(2\Delta)\).

\emph{(2) Boundaryless case.}
Assume that all row sums of \(X\) vanish. Then \(b_i=0\) for all \(i\), and
\[
x^TMx=\sum_{i<j}w_{ij}(g_i-g_j)^2.
\]
Since the interior graph associated with \(X\) is connected, \(\ker M=\operatorname{span}(W^{1/2}\one)\). Let \(x\perp W^{1/2}\one\) be an eigenvector for \(\lambda_1(M)\), let \(g=W^{-1/2}x\), and set \(h_i=g_i^2\). The orthogonality condition is \(\sum_i\mu_i g_i=0\). Thus \(g\) has both nonnegative and nonpositive parts. Replacing \(g\) by \(-g\) if necessary, we may assume that its positive support \(S_+:=\{i:g_i>0\}\) satisfies \(\operatorname{Vol}(S_+)\leq V_{\mathrm{tot}}/2\). Define \(g_i^+:=\max(g_i,0)\). Edgewise, \(|g_i^+-g_j^+|\leq |g_i-g_j|\), so the energy of \(g^+\) is at most the energy of \(g\). Thus it suffices to prove the desired lower bound for a nonnegative function \(g\) whose support \(S_+\) satisfies \(\operatorname{Vol}(S_+)\leq V_{\mathrm{tot}}/2\).

For such a nonnegative \(g\), we use
\[
|h_i-h_j|=|g_i^2-g_j^2|=|g_i-g_j|\,|g_i+g_j|.
\]
By Cauchy--Schwarz,
\[
\left(\sum_{i<j}w_{ij}|h_i-h_j|\right)^2
\leq
\left(\sum_{i<j}w_{ij}(g_i-g_j)^2\right)
\left(\sum_{i<j}w_{ij}(g_i+g_j)^2\right).
\]
Thus
\[
x^TMx\geq
\frac{\left(\sum_{i<j}w_{ij}|h_i-h_j|\right)^2}
{\sum_{i<j}w_{ij}(g_i+g_j)^2}.
\]
The denominator is bounded by
\[
\sum_{i<j}w_{ij}(g_i+g_j)^2
\leq 2\sum_{i<j}w_{ij}(g_i^2+g_j^2)
=
2\sum_i\left(\sum_{j\neq i}w_{ij}\right)g_i^2
=
2\sum_iX_{ii}g_i^2
\leq 2\Delta\|x\|^2.
\]
For \(t>0\), let \(S_t:=\{i:h_i>t\}\). Since \(S_t\subseteq S_+\), we have \(\operatorname{Vol}(S_t)\leq V_{\mathrm{tot}}/2\). By the layer-cake formula,
\[
\sum_{i<j}w_{ij}|h_i-h_j|
=
\int_0^\infty\sum_{i<j}w_{ij}|\mathbf 1_{S_t}(i)-\mathbf 1_{S_t}(j)|\,dt
=
\int_0^\infty\operatorname{Cut}(S_t)\,dt.
\]
By the definition of \(h_{\mathrm{vol}}\),
\[
\int_0^\infty\operatorname{Cut}(S_t)\,dt
\geq
h_{\mathrm{vol}}\int_0^\infty\operatorname{Vol}(S_t)\,dt
=
h_{\mathrm{vol}}\sum_i\mu_i h_i
=
h_{\mathrm{vol}}\|x\|^2.
\]
Combining the numerator and denominator estimates gives
\[
x^TMx\geq
\frac{(h_{\mathrm{vol}}\|x\|^2)^2}{2\Delta\|x\|^2}
=
\frac{h_{\mathrm{vol}}^2}{2\Delta}\|x\|^2.
\]
Dividing by \(\|x\|^2\) gives \(\lambda_1(M)\geq h_{\mathrm{vol}}^2/(2\Delta)\).
\end{proof}

\subsection{Missing details in \cref{sec:cheeger_phi_psdmfld}}\label{sec:missing_proof}
\begin{proof}[\textbf{Proof of} \cref{lemma:ungrd_graph_diam}]
Let $w=I x$.
 We can regard $x \in \mathbb{R}^{|E|}$ as a flow vector on the graph $G$, 
 and $w \in \mathbb{R}^{|V|}$ as the resulting node imbalance.
Since the incidence matrix $I$ acts as a linear operator from $\mathbb{R}^{|E|}$ to $\mathbb{R}^{|V|}$,
we define an operator $I^{+}: \im (I)\to \mathbb{R}^{|E|}$ that maps a valid node imbalance $w$ back to the flow that satisfies this imbalance with the minimal $\ell^1$-norm:
\[I^{+}:w\mapsto \argmin_{x}\{ \|x\|_1 : Ix = w \}.\]
Note that this minimal $\ell^1$-norm solution is characterized by the orthogonality condition $\|x+y\|_1 \geq \|x\|_1$ for any $y \in \ker(I)$. 
We then have:
\begin{equation}\label{eq:min_w}
\min_{x\perp^1\ker(I)}\frac{\|Ix\|_1}{\|x\|_1}=
\min_{x\notin \ker(I)} \frac{\|I x\|_1}{\min\limits_{y\in \ker(I)}\|x+y\|_1}=
\min_{w\in \im I}\frac{\|w\|_1}{\|I^{+}w\|_1}.
\end{equation}
Let $\widehat{w}$ be a node imbalance that minimize the right-hand side of \cref{eq:min_w},
and let $\widehat{x}=I^{+}\widehat{w}$.
Since $I$ is a totally unimodular matrix,
the conformal decomposition property guarantees that any such flow imbalance $\widehat{w}$ can be decomposed as
\[
\widehat{w}=
\sum_{i\neq j} c_{ij} (e_i-e_j),
\]
where $e_i, e_j\in \mathbb{R}^{|V|}$ are standard basis vectors, $c_{ij}\geq 0$,
and $\|\widehat{w}\|_1=\sum\limits_{i\neq j} c_{ij} \|e_i - e_j\|_1=\sum\limits_{i\neq j} 2c_{ij}$.

On the other hand,
let $x_{ij}=I^{+} (e_i-e_j)$.
According to \cref{lemma: tot_unimod},
we obtain $x_{ij }\in \{-1, 0, 1\}^{|E|}$ for each pair of $e_i\neq e_j$.
Geometrically, $x_{ij}$ corresponds to the flow along the shortest path between nodes $i$ and $j$ in the underlying undirected graph.
Thus, 
$\|x_{ij}\|_1 = \dist(i, j)$.
By the triangle inequality for the $\ell^1$-norm,
we have:
\[
\frac{\|\widehat{w}\|_1}{\|\widehat{x}\|_1} = 
\frac{\sum c_{ij} \|e_i-e_j\|_1}{\| I^{+}\left(\sum c_{ij} (e_i-e_j)\right) \|_1} \geq 
\frac{\sum 2c_{ij}}{\sum c_{ij} \|x_{ij}\|_1}\geq
\min_{i, j}\frac{2}{\dist (i, j)}=\frac{2}{\diam(G)}.
\]
Since $\widehat{w}$ is chosen to minimize \cref{eq:min_w},
hence \cref{eq:ungrd_graph_diam} is proved.
\end{proof}

\begin{proof}[\textbf{Proof of} \cref{lemma:grd_graph_diam}]
Similarly to \cref{eq:min_w},
we have the equality
\begin{equation}\label{eq:min_w_ground}
\min_{x\perp^1\ker(R)}\frac{\|Rx\|_1}{\|x\|_1}=
\min_{w\in \im R}\frac{\|w\|_1}{\|R^{+}w\|_1},
\end{equation}
where $R^+:\im(R)\to \mathbb{R}^{|E|}$ is defined as mapping a node imbalance to the minimal $\ell^1$-norm flow:
\[
R^{+}:w\mapsto \argmin_{x}\{ \|x\|_1 : Rx = w \}.
\]
Let $\widehat{w}\in \im (R)$ be a vector that minimizes the right-hand side of \cref{eq:min_w_ground}.
By conformal decomposition,
$\widehat{w}$ can be decomposed into the standard basis vectors of $\mathbb{R}^{|V|-1}$:
\[
\widehat{w}=
\sum_i c_i(\sigma_i e_i),
\]
where $\sigma_i \in \{-1, 1\}$ represents the sign of the component,
$c_i\geq 0$,
and $\|\widehat{w}\|_1=\sum\limits_i c_i$.

Similarly to the proof of \cref{lemma:ungrd_graph_diam},
we take $x_i=R^{+}(\sigma_i e_i)$.
According to \cref{cor:tot_unimod_ground},
$x_i\in \{-1, 0, 1\}^{|E|}$ for each $i\in V$.
Geometrically, $x_i$ represents the minimal flow between vertex $i$ and the ground vertex $u$,
thus $\|x_i\|_1 = \dist(i, u)$.
Therefore,
\[
\frac{\|\widehat{w}\|_1}{\|\widehat{x}\|_1} = 
\frac{\sum c_{i}}{\| R^{+}\left(\sum c_{i}(\sigma_i  e_i)\right) \|_1} \geq 
\frac{\sum c_{i}}{\sum c_{i} \|x_{i}\|_1}\geq
\min_{i\in V}\frac{1}{\dist (i, u)}.
\]
Since $\widehat{w}$ is the vector that minimizes the right-hand side of \cref{eq:min_w_ground},
\cref{eq:grd_graph_diam} is proved.
\end{proof}

\subsection{Proof of \cref{eq:L_chi_1_norm}}
\label{app.proof.of.eq}
\begin{lemma}
Let \(G=(V,E)\) be an undirected graph with adjacency matrix \(A\), where \(A_{ij}=1\) if \(\{i,j\}\in E\) and \(A_{ij}=0\) otherwise. Let \(D\) be the diagonal degree matrix, \(D_{ii}=\sum_j A_{ij}\), and let \(\Delta_0^G=L:=D-A\). Then, for every \(S\subseteq V\),
\[
\|\Delta_0^G\chi_S\|_1=2\chi_S^T\Delta_0^G\chi_S=2|E(S,S^c)|.
\]
\end{lemma}

\begin{proof}
For \(i\in V\),
\[
(L\chi_S)_i=
\begin{cases}
\sum_{j\notin S}A_{ij}, & i\in S,\\
-\sum_{j\in S}A_{ij}, & i\notin S.
\end{cases}
\]
Hence
\[
\|L\chi_S\|_1=\sum_{i\in S}\sum_{j\notin S}A_{ij}+\sum_{i\notin S}\sum_{j\in S}A_{ij}=2|E(S,S^c)|,
\]
where we used the symmetry of \(A\). Moreover,
\[
\chi_S^TL\chi_S=\sum_{i\in S}(L\chi_S)_i=\sum_{i\in S}\sum_{j\notin S}A_{ij}=|E(S,S^c)|.
\]
\end{proof}
\bibliography{sample}

\end{document}